\documentclass[12pt]{amsart}

\usepackage{amsmath, amssymb, amsthm, amsfonts, bbm, dsfont}
\usepackage{graphicx}

\usepackage{mathrsfs}

\usepackage{hyperref}
\usepackage{pdfsync}
\usepackage{a4wide}
\newtheorem{thm}{Theorem}
\newtheorem{theorem}[thm]{Theorem}

\newtheorem{proposition}[thm]{Proposition}

\newtheorem{definition}[thm]{Definition}

\newtheorem{conjecture}[thm]{Conjecture}

\usepackage{pictexwd}
\usepackage{stackengine}
\usepackage{mathtools}
\usepackage{amssymb}
\usepackage{dsfont}

\usepackage{mathabx}

\newcommand{\eql}{\stackrel{\langle,\rangle}{=}}

\newcommand{\cc}{\mathbb{C}}

\newcommand{\PP}{\mathbb{P}}

\newcommand{\zzz}{\mathbb{Z}}
\newcommand{\bbD}{\mathbb{D}}

\newcommand{\calO}{\mathcal{O}}

\newcommand{\calL}{\mathcal{L}}

\newcommand{\rmR}{\mathrm{R}}
\newcommand{\rmL}{\mathrm{L}}
\newcommand{\rmT}{\mathrm{T}}
\newcommand{\rmbb}{\mathrm{B}}
\newcommand{\rmvv}{\mathrm{V}}
\newcommand{\rmdd}{\mathrm{D}}
\newcommand{\rmee}{\mathrm{E}}
\newcommand{\CT}{\mathrm{T}}

\newcommand{\fra}{\mathfrak{a}}

\newcommand{\quash}[1]{}

\usepackage{tikz}
\usetikzlibrary{shapes,fit,intersections}
\usetikzlibrary{decorations.markings}
\usetikzlibrary{decorations.pathreplacing}
\usetikzlibrary{patterns}
\usetikzlibrary{matrix,arrows}
\usepgflibrary{decorations.pathmorphing}

\usepackage{tikz-cd}

\newcommand{\Hr}{\mathrm{H}}
\newcommand{\wdh}{\widehat{\Hr}}

\newcommand{\Heis}{\mathsf{Heis}}

\newcommand{\Cor}{\mathfrak{C}}

\newcommand{\PT}{{\mathrm{PT}}}
\newcommand{\tch}{\widetilde{\mathsf{ch}}}

\newcommand{\Vir}{Virasoro }
\newcommand{\GW}{\mathrm{GW}}

\newcommand\ZZ{\mathsf Z}

\newcommand{\OO}{{\mathcal{O}}}

\newcommand{\LL}{{\mathbb{L}}}

\newcommand{\ch}{\mathsf{ch}}

\newcommand{\oM}{\overline{M}}

\DeclareMathOperator{\Res}{Res}
\DeclareMathOperator{\Hilb}{Hilb}

\DeclareMathOperator{\Aut}{Aut}

\DeclareMathOperator{\pt}{\mathsf{p}}

\newcommand{\rmrp}{\mathrm{R}}
\newcommand{\rmrg}{\mathrm{R}}
\newcommand{\CTp}{\mathrm{T}}

\begin{document}

\title{Virasoro constraints for stable pairs on toric $3$-folds}
\author{M. Moreira, A. Oblomkov, A. Okounkov, \and R. Pandharipande}
\newcommand{\Addresses}{{
  \bigskip
  \footnotesize

  M. Moreira, \textsc{Department of Mathematics, ETH Z\"urich}\par\nopagebreak
  \textit{E-mail address:} \texttt{moreira@math.ethz.ch}

  \medskip
  
  A.~Oblomkov, \textsc{Department of Mathematics and Statistics, Univ. of Massachusetts,
    Amherst}\par\nopagebreak
  \textit{E-mail address:} \texttt{oblomkov@math.umass.edu}

  \medskip

  A.~Okounkov, \textsc{Department of Mathematics,
    Columbia University, New York}\par\nopagebreak
  \textsc{Institute for Problems of Information Transmission, Moscow}\par\nopagebreak
  \textsc{Laboratory of Representation Theory and Mathematical Physics,
HSE, Moscow}\par\nopagebreak
  \textit{E-mail address:} \texttt{okounkov@math.columbia.edu}

  \medskip

  R.~Pandharipande, \textsc{Department of Mathematics, ETH Z\"urich}\par\nopagebreak
  \textit{E-mail address:} \texttt{rahul@math.ethz.ch}

}}

\date{August 2020}
\maketitle
\begin{abstract}
 Using new explicit formulas for the stationary $\GW/\PT$
 descendent correspondence for nonsingular
projective toric 3-folds, we show
 that the correspondence intertwines the Virasoro constraints 
in Gromov-Witten theory
for stable maps with
the Virasoro constraints for stable pairs
proposed in \cite{OblomkovOkounkovPandharipande18}. 
Since the  Virasoro constraints
in Gromov-Witten theory are known to hold in the toric case, we establish
the stationary Virasoro constraints for the theory of
 stable pairs on 
 toric $3$-folds. As a consequence, new Virasoro constraints for
tautological integrals over Hilbert schemes of points on surfaces
are also obtained. 
\end{abstract}

\maketitle
\setcounter{tocdepth}{1}
\tableofcontents
\setcounter{section}{-1}

\section{Introduction}
\label{sec:introduction}

\subsection{Stable pairs}\label{spint}
Let $X$ be a nonsingular projective $3$-fold. 
A {\em{stable pair}} $(F,s)$ on $X$ is a coherent sheaf $F$ on $X$
and a  section $s\in H^0(X,F)$ satisfying the following stability conditions:
\begin{itemize}
\item $F$ is \emph{pure} of dimension 1,
\item the section $s:\OO_X\to F$ has cokernel of dimensional 0.
\end{itemize}
To a stable pair, we associate the Euler characteristic and
the class of the support $C$ of the sheaf $F$,
$$\chi(F)=n\in \mathbb{Z} \  \ \ \text{and} \ \ \ [C]=\beta\in H_2(X,\mathbb{Z})\,.$$
For fixed $n$ and $\beta$,
there is a projective moduli space of stable pairs $P_n(X,\beta)$. 
Unless $\beta$ is an effective curve class, the moduli space
$P_n(X,\beta)$ is empty. An analysis of the
deformation theory and the construction of the virtual
cycle $[P_n(X,\beta)]^{vir}$ is given \cite{pt}.
We refer the reader to \cite{P,rp13}
for an introduction to the theory of stable pairs.

Tautological descendent classes are defined via universal structures over the
moduli space of stable pairs.
Let
$$\pi: X\times P_n(X,\beta) \rightarrow P_n(X,\beta)$$
be the projection to the second factor, and let
\[\mathcal{O}_{X\times P_n(X,\beta)}\rightarrow \mathbb{F}_n\]
be the universal stable pair on $X\times P_n(X,\beta)$. Let{\footnote{We will
always take singular cohomology with $\mathbb{Q}$-coefficients.}}
\begin{equation*}
\ch_k(\mathbb{F}_n-\mathcal{O}_{X\times P_n(X,\beta)})\in  H^*(X\times P_n(X,\beta))\, .
\end{equation*}
The following
{\em descendent classes} are our main objects of study:
$$\ch_k(\gamma) = \pi_*\left(\ch_k(\mathbb{F}_n-\mathcal{O}_{X\times P_n(X,\beta)})\cdot \gamma\right)
\in H^*(P_n(X,\beta))\, $$
for $k\geq 0$ and $\gamma\in H^*(X)$.
The summand $-\mathcal{O}_{X\times P_n(X,\beta)}$  only
affects $\ch_0$,
\begin{equation}\label{m449}
  \ch_0(\gamma) = -\int_X \gamma \, \in \, H^0(P_n(X,\beta))\, .
  \end{equation}
Since stable pairs are supported on curves, the vanishing
$$\ch_1(\gamma)=0$$
always holds. 

We will study the 
following descendent series:
\begin{equation}\label{ss33}
 \Big\langle {\mathsf{ch}}_{k_1}(\gamma_1)\cdots \mathsf{ch}_{k_m}(\gamma_m)\Big\rangle_{\beta}^{X,\PT}
\, =\, \sum_{n\in \mathbb{Z}}  q^n\, 
\int_{[P_n(X,\beta)]^{vir}}\prod_{i=1}^m\ch_{k_i}(\gamma_i)\, .
\end{equation}
For fixed curve class $\beta\in H_2(X,\mathbb{Z})$, the moduli space
$P_n(X,\beta)$ is empty for all sufficiently negative $n$. Therefore, the 
descendent series \eqref{ss33} has only finitely many polar terms.

\begin{conjecture}\cite{pt}  \label{ggtthh}
The stable pairs descendent series 
$$ \Big\langle \mathsf{ch}_{k_1}(\gamma_1)\cdots \mathsf{ch}_{k_m}(\gamma_m)\Big\rangle_{\beta}^{X,\PT}$$ 
is the Laurent expansion of a  rational function of $q$ for
all $\gamma_i\in H^*(X)$ and all $k_i\geq 0$.
\end{conjecture}

For Calabi-Yau 3-folds, Conjecture \ref{ggtthh} reduces
immediately to the rationality of the basic series
$ \langle\, 1\, \rangle_{\beta}^\PT$
proven via wall-crossing in \cite{Br,Toda}. In the presence of
descendent insertions, Conjecture \ref{ggtthh} has been proven for
 rich class of varieties 
 \cite{part1,PPstat,PP2,PPDC, PPQ}
including all nonsingular projective toric $3$-folds.

For our study of the $\GW/\PT$ descendent
correspondent and the Virasoro constraints,
modified stable pair descendent insertions will be more suitable for us. 
Let{\footnote{We set $\ch_\ell(\gamma)=0$ for $\ell<0$. 
  }
}
\[ \tch_k(\alpha)=\ch_k(\alpha)+\frac{1}{24}\ch_{k-2}(\alpha\cdot c_2 )\,, \]
where $c_2=c_2(T_X)$ is the second Chern class of the tangent bundle, 
and let 
$$\Big\langle \tch_{k_1}(\gamma_1)\cdots \tch_{k_m}(\gamma_m)\Big\rangle_{\beta}^{X,\PT}
\, =\, \sum_{n\in \mathbb{Z}}  q^n\, 
\int_{[P_n(X,\beta)]^{vir}}\prod_{i=1}^m\tch_{k_i}(\gamma_i)$$
be the corresponding descendent series.

\subsection{Virasoro constraints for stable pairs}
\label{sec:pt-vir-constraints}

Let $X$ be a nonsingular projective 3-fold with only
$(p,p)$-cohomology.{\footnote{Our results will be about
nonsingular projective toric varieties, but the formulas
here are all well-defined when there is no odd cohomology and
the Hodge classes in the even cohomology are all $(p,p)$.
To write the Virasoro constraints
for varieties with non-$(p,p)$ cohomology requires the Hodge grading
and signs. A treatment is presented in \cite {Mor} where
the Virasoro constraints are checked in several non-$(p,p)$ geometries.
The theory leads to surprising predictions for vanishings \cite{Mor}.
}}
Let
$$c_i = c_i(T_X) \in H^*(X)\, .$$
The simplest example is $\mathbb{P}^3$ with
$$ c_1= 4\mathsf{H}\, , \  \ \ c_1c_2 = 24 \mathsf{p}\, ,$$
where $\mathsf{H}$ and $\mathsf{p}$ are the classes of the
hyperplane and the point respectively.

Let \(\bbD^X_{\PT}\)
be the commutative $\mathbb{Q}$-algebra 
 with generators
 $$\big\{ \, \ch_i(\gamma)\, \big| \,  i\ge 0\, , \gamma\in H^*(X)\,
\big\} $$
subject to the natural relations
\begin{eqnarray*}
\ch_i(\lambda\cdot \gamma) & = & \lambda\, \ch_i(\gamma)\, , \\
 \ch_i(\gamma+\widehat{\gamma})& =&\ch_i(\gamma)+\ch_i(\widehat{\gamma})\, 
\end{eqnarray*}
for $\lambda\in \mathbb{Q}$ and $\gamma,\widehat{\gamma} \in H^*(X)$.

In order to define the Virasoro constraints for stable pairs, 
 we require three constructions in the algebra \(\bbD^X_{\PT}\):
\begin{enumerate}
\item[$\bullet$] Define the
derivation \(\mathrm{R}_k\) on
\(\bbD^X_{\PT}\)
by fixing the action on the generators:
\[\mathrm{R}_k(\ch_i(\gamma))=\left(\prod_{n=0}^k(i+d-3+n)\right)\ch_{i+k}(\gamma)\, ,\quad \gamma\in H^{2d}(X,\mathbb{Q})\, \]
for $k\geq -1$. In case $k=-1$, the product is empty and
$$\mathrm{R}_{-1}(\ch_i(\gamma)) = \ch_{i-1}(\gamma)\, .$$
\item[$\bullet$]
Define the element
\begin{equation*}
\mathsf{ch}_a\mathsf{ch}_b(\gamma) = \sum_i \mathsf{ch}_a(\gamma^L_i)
\mathsf{ch}_b(\gamma^R_i) \, \in\,  \bbD^X_{\PT}
\end{equation*}
where $\sum_i \gamma^L_i \otimes \gamma^R_i$
is the K\"unneth decomposition of the product, 
$$\gamma\cdot\Delta \in H^*(X\times X)\, ,$$
with the diagonal $\Delta$.
The notation
$$(-1)^{d^L d^R}(a+d^L-3)!(b+d^R-3)!\,
\mathsf{ch}_a\mathsf{ch}_b(\gamma)\, $$
will be used as shorthand for the sum
$$
\sum_i (-1)^{d(\gamma_i^L) d(\gamma_i^R)}(a+d(\gamma_i^L)-3)!
(b+d(\gamma_i^R)-3)!\, \mathsf{ch}_a(\gamma^L_i)
\mathsf{ch}_b(\gamma^R_i)\, ,
$$
where $d(\gamma_i^L)$ and $d(\gamma_i^R)$ are the (complex) degrees
of the classes.
All factorials with negative
arguments vanish. 

\vspace{5pt}
\item[$\bullet$]
Define the operator $\CT_k: \bbD^X_\PT\rightarrow \bbD^X_\PT$ by
\[\CT_k=-\frac{1}{2}\sum_{a+b=k+2}(-1)^{d^Ld^R}(a+d^L-3)!(b+d^R-3)!\, \ch_a\ch_b(c_1)+\frac{1}{24}\sum_{a+b=k}a!b!\, \ch_a\ch_b(c_1c_2)\ \]
for $k\geq -1$. The sum here
is over all ordered pairs $(a,b)$ satisfying $a+b=k+2$
with $a,b\geq 0$ (and all factorials with negative arguments vanish).
Written in terms of renormalized descendents, the formula simplifies to
\begin{equation}
  \label{ll99v}
  \CT_k=-\frac{1}{2}\sum_{a+b=k+2}(-1)^{d^Ld^R}(a+d^L-3)!(b+d^R-3)!\, \tch_a\tch_b(c_1)\, .
  \end{equation}
\end{enumerate}

\vspace{5pt}
\begin{definition} Let $\mathcal{L}^\PT_k:\bbD^X_\PT\rightarrow
\bbD^X_\PT$
for $k\geq -1$ be the operator
\begin{equation*}
  \mathcal{L}^\PT_k=\CT_k+ \mathrm{R}_k + (k+1)!\, \mathrm{R}_{-1}
  \mathsf{ch}_{k+1}(\mathsf{p})\, .
\end{equation*}
\end{definition}
\vspace{5pt}
Since $X$ is a nonsingular projective 3-fold with only $(p,p)$-cohomology,
Hirzebruch-Riemman-Roch implies 
$$\frac{c_1c_2}{24}=\mathsf{p} \in H^6(X)\, ,$$ 
where $\mathsf{p} \in H^6(X)$ in the point class. Hence,
for our paper,
we can write
\begin{equation}
 \label{f49907}
  \mathcal{L}^\PT_k=\CT_k+ \mathrm{R}_k + (k+1)!\, \mathrm{R}_{-1}
  \mathsf{ch}_{k+1}\Big(  \frac{c_1c_2}{24}   \Big)\, .
\end{equation}
The operators for more general varieties $X$ defined 
in \cite{Mor} specialize to \eqref{f49907} when all the cohomology
is $(p,p)$.

The operators \(\mathcal{L}^\PT_k\) 
impose constraints on descendent integrals in the theory of stable
pairs which are analogous to the Virasoro constraints of Gromov-Witten theory.
We formulate the stable pairs Virasoro constraints as follows.

\begin{conjecture}\label{vvirr}
\cite{OblomkovOkounkovPandharipande18} Let $X$ be a nonsingular projective 3-fold with only
$(p,p)$-cohomology, and let $\beta \in H_2(X,\mathbb{Z})$. 
 For all $k\geq -1$ and  \(D\in \bbD^X_{\PT}\), we have
  \[\Big\langle\calL^\PT_k(D) \Big\rangle_\beta^{X,\PT}=0\, .\]
\end{conjecture}


Our main result is a statement about {\em stationary descendents} for nonsingular projective {\em toric} 3-folds.
The subalgebra \(\bbD_\PT^{X+}\subset
\bbD^X_\PT\) of stationary descendents is generated{\footnote{Equivalently,
 $\bbD_\PT^{X+}$ is generated by $\big\{ \, \tch_i(\gamma)\, \big| \,  i\ge 0\, , \gamma\in H^{>0}(X,\mathbb{Q})
\,
\big\}$.}}
    by
$$\big\{ \, \ch_i(\gamma)\, \big| \,  i\ge 0\, , \gamma\in H^{>0}(X,\mathbb{Q})
\,
\big\}\, . $$
The operators \(\calL^\PT_k\) are easily seen to preserve $\bbD_\PT^{X+}$. Therefore,
the {\em stationary} \Vir constraints are well-defined.
We prove that the stationary \Vir constraints hold in the toric case.

\begin{theorem}\label{thm:Vir-toric}
  Let \(X\) be a nonsingular projective toric 3-fold, and let
\(\beta\in H_2(X,\mathbb{Z})\).  For all $k\geq -1$ and
\(D\in \bbD_\PT^{X+}\), we have
  \[\Big\langle\calL^\PT_k(D)\Big\rangle^{X,\PT}_\beta=0\, .\]
\end{theorem}

In the basic case of $\mathbb{P}^3$, Theorem \ref{thm:Vir-toric}
specializes to the Virasoro constraints for stable
pairs announced earlier in \cite{P} via \eqref{f49907}.
A table of data of
stable pairs descendent series for $\mathbb{P}^3$ is presented
in the Appendix. The Virasoro constraints are seen to  provide nontrivial    
relations.

\subsection{The Virasoro bracket}\label{virbra}
For $k\geq -1$, we introduce the operators
\begin{eqnarray*}
 \rmL^\PT_k&=&
-\frac{1}{2}\sum_{a+b=k+2}(-1)^{d^L d^R}(a+d^L-3)!(b+d^R-3)!\, 
\mathsf{ch}_a\mathsf{ch}_b(c_1)\\
& &+\frac{1}{24}\sum_{a+b=k}a!b!\, \mathsf{ch}_a\mathsf{ch}_b(c_1c_2)\\
           & &+ \mathsf{R}_k\, ,
\end{eqnarray*}
where the sum, as before,
is over ordered pairs $(a,b)$
with $a,b\geq 0$.

Our conventions with regard to the factorials
in the above definition of $\rmL_k^\PT$ differ slightly
from those of the definition of $\calL_k^\PT$. For $\rmL_k^\PT$,
{\em all terms with negative
    factorial vanish {\em except} for
    the term \((-1)!\, \mathsf{ch}_1(c_1)\)}.
  For example,
  we have \[\rmL^\PT_{-1}=\mathrm{R}_{-1}+(-1)!\, \mathsf{ch}_1(c_1)\mathsf{ch}_0(\mathsf{p})\, .\]
The new conventions will  play a role in the exceptional cases
in our analysis. We extend the action of $\rmR_k$ by
\[\rmR_k((-1)!\, \ch_1(c_1))= -(k-1)!\, \ch_{k+1}(c_1).\]
We view $(-1)!\ch_1(c_1)$ and
$$\rmR_{-1}((-1)!\ch_1(c_1))= -(-2)!\ch_0(c_1)$$
as formal symbols.

We define an equivalence relation
$\eql$ 
for operators $\mathcal{A}, \mathcal{B}:\mathbb{D}^X_\PT\rightarrow
\mathbb{D}^X_\PT$
by
$$ \mathcal{A}\,\eql\, \mathcal{B}  \ \ \ \ \leftrightarrow \ \ \ \
\langle \mathcal{A}(D) \rangle^{X,\PT}_\beta = \langle \mathcal{B}(D)\rangle^{X,\PT}_\beta \ \ \text{for all} \ \ D \in \mathbb{D}^X_\PT\ \ \text{and}\ \ 
\beta\in H_2(X,\mathbb{Z})\, .$$
Inside the bracket,
$\mathsf{ch}_0(\mathsf{p})$ acts as $-1$, and
$\mathsf{ch}_1(\gamma)$ acts as 0 for all $\gamma\in H^*(X)$.
Moreover, the formal symbols  $(-1)!\ch_1(c_1)$ and $(-2)!\ch_0(c_1)$
are {\em defined} to act as 0
inside the bracket.

Using the equivalence relation $\eql$,
we obtain the Virasoro bracket and the following
bracket with $\mathsf{ch}_k(\mathsf{p})$,
$$[\rmL^\PT_n,\rmL^\PT_k]\, \eql\, (k-n)\, \rmL^\PT_{n+k}\, ,\quad \ \ [\rmL^\PT_n,(k-1)!\, \mathsf{ch}_k(\mathsf{p})]\, \eql \,
(n+k)!\, \mathsf{ch}_{n+k}(\mathsf{p})\, .$$

The operators $\mathcal{L}^\PT_k$ are expressed in terms of  $\rmL^\PT_k$ by
$$\mathcal{L}^\PT_k \, \eql\,   \rmL^\PT_k
+ (k+1)!\, \rmL^\PT_{-1}\mathsf{ch}_{k+1}(\mathsf{p})\, .$$
The occurrences of  the negative factorial terms 
$(-1)!\ch_1(c_1)$ cancel on the right side.
The expressions $\rmL^\PT_k$ will play a role in the 
proof of Theorem \ref{thm:Vir-toric}.

\subsection{Path of the proof}
Our proof of Theorem \ref{thm:Vir-toric}
relies upon two central results. The first is 
the  Virasoro conjecture in Gromov-Witten theory
which has been proven for nonsingular projective toric varieties \cite{Giv}.
We refer the reader to the
extensive literature on the subject \cite{EHX, GetPan, Giv,
  Pandharipande03,OPVir,Tel}.
The second is 
 the stationary $\GW$/$\PT$ correspondence of \cite{part1,PPstat,PP2}
which was cast in terms of vertex operators in 
\cite{OblomkovOkounkovPandharipande18}
and has been proven for nonsingular projective toric 3-folds.
We show the  stationary $\GW$/$\PT$ correspondence intertwines the \Vir constraints of the two theories. Along the way, 
we derive a more explicit form for the stationary $\GW$/$\PT$ correspondence.
Our proof of Theorem \ref{thm:Vir-toric} yields the
following stronger statement.

\begin{theorem}\label{thm:main}
Let 
 \(X\) be a nonsingular projective 3-fold with only $(p,p)$-cohomology for which 
the following two properties are satisfied:
\begin{enumerate}
\item[(i)]
The stationary Virasoro constraints for the Gromov-Witten theory of $X$ hold.
\item[(ii)]
The stationary $\GW$/ $\PT$ correspondence holds. 
\end{enumerate}
Then, the stationary
Virasoro
constraints for the stable pairs theory of $X$ hold.
\end{theorem}

A challenge for the subject is to prove the
Virasoro constraints for stable pairs directly using the
geometry of the moduli of sheaves. New ideas will almost
certainly be required.

\subsection{Gromov-Witten theory}\label{p3f}
Let $X$ be a nonsingular projective 3-fold.
Gromov-Witten theory is defined via integration over the moduli
space of stable maps.

Let $C$ be a possibly disconnected curve with at worst nodal singularities.
The genus of $C$ is defined by $1-\chi(\OO_C)$. 
Let $\overline{M}'_{g,m}(X,\beta)$ denote the moduli space of stable maps
with possibly {disconnected} domain
curves $C$ of genus $g$ with {\em no} collapsed connected components of genus greater or equal to \(2\).
The latter condition{\footnote{The exclusion here of
    collapsed connected components of genus greater or equal to 2 matches
    the conventions of \cite{OblomkovOkounkovPandharipande18}.
    The definition of  $\overline{M}'_{g,m}(X,\beta)$ differs slightly
    from the definitions of \cite{PPDC,PPQ} where no collapsed
    connected components are permitted. The difference is minor, see
    Section 3 of  \cite{OblomkovOkounkovPandharipande18} for
    a discussion.}}
requires 
 each non-rational and non-elliptic connected component of $C$ to represent
 a nonzero class in $H_2(X,{\mathbb Z})$.

Let 
$$\text{ev}_i: \overline{M}'_{g,m}(X,\beta) \rightarrow X\, ,$$
$$ \LL_i \rightarrow \overline{M}'_{g,m}(X,\beta)$$
denote the evaluation maps and the cotangent line bundles associated to
the marked points.
Let $\gamma_1, \ldots, \gamma_m\in H^*(X)$, and
let $$\psi_i = c_1(\LL_i) \in H^2(\overline{M}'_{g,m}(X,\beta))\, .$$
The {\em descendent insertions}, denoted by $\tau_k(\gamma)$ for $k\geq 0$,
correspond 
to classes $\psi_i^k \text{ev}_i^*(\gamma)$ on the moduli space
of stable maps. 
Let
$$\Big\langle \tau_{k_1}(\gamma_{1}) \cdots
\tau_{k_m}(\gamma_{m})\Big\rangle^{X,\GW}_{g,\beta} = \int_{[\overline{M}'_{g,m}(X,\beta)]^{vir}} 
\prod_{i=1}^m \psi_i^{k_i} \text{ev}_i^*(\gamma_{_i})$$
denote the descendent
Gromov-Witten invariants.
The associated generating series is defined 
by
\begin{equation}
\label{abc}
\Big\langle   \tau_{k_1}(\gamma_{1}) \cdots
\tau_{k_m}(\gamma_{m})\Big\rangle^{X,\GW}_{\beta}= 
\sum_{g\in{\mathbb Z}} \Big \langle \prod_{i=1}^m
\tau_{k_i}(\gamma_{i}) \Big \rangle^{X,\GW}_{g,\beta} \ u^{2g-2}.
\end{equation}
Since the domain components must map nontrivially, an elementary
argument shows the genus $g$ in the  sum \eqref{abc} is bounded from below. 
 Foundational aspects of the theory
are treated, for example, in \cite{BehFan,FPn, LiTian}.

Using the above definitions, the string equation{\footnote{The
   standard correction term for the string equation occurs here  since
    we allow collapsed connected components of genus 0 in our definition of
    the Gromov-Witten descendent series.}} is easily checked:
\begin{equation}\label{sss}
\Big \langle \tau_{0}(1) \prod_{i=1}^m \tau_{k_i}(\gamma_i) \Big
\rangle^{X,\GW}_\beta=  
\Big\langle   \sum_{j=1}^m \prod_{i=1}^m \tau_{k_i-\delta_{i-j}}(\gamma_i)  \Big\rangle^{X,\GW}_\beta\, + \text{collapsed contributions}.
\end{equation}

The Gromov-Witten descendent insertions $\tau_k(\gamma)$ in \eqref{abc} are
defined for $k\geq 0$. We include the nonstandard descendent insertions
$\tau_{-2}(\gamma)$ and $\tau_{-1}(\gamma)$ by the rule:
\begin{equation}\label{ActionVacuum}
\Big \langle \tau_{k}(\gamma) \prod_{i=1}^m \tau_{k_i}(\gamma_i) \Big
\rangle^{X,\GW}_\beta= \frac{\delta_{k+2}}{u^2}\int_X\gamma\ \, 
\cdot\ \, \Big\langle    \prod_{i=1}^m \tau_{k_i}(\gamma_i)
\Big\rangle^{X,\GW}_\beta\, , \ \ \text{for $k<0$.}
\end{equation}
We impose  
Heisenberg relations \eqref{eq:Heis} on 
 the operators \(\tau_{k}(\gamma)\):
\begin{equation}\label{pp23pp}
  [\tau_k(\alpha),\tau_l(\beta)]=(-1)^{k}\frac{\delta_{k+l+1}}{u^2}\int_X\alpha
  \cdot\beta\, .
\end{equation}
In particular, the evaluation \eqref{ActionVacuum} applies
only after commuting the negative descendents to the left.

Assume now that $X$ has only
$(p,p)$-cohomology.
Let \(\bbD^X_{\GW}\)
be the commutative $\mathbb{Q}$-algebra 
 with generators
$$\big\{ \, \tau_i(\gamma)\, \big| \,  i\ge 0\, , \gamma\in H^*(X)\,
\big\} $$
subject to the natural relations
\begin{eqnarray*}
\tau_i(\lambda\cdot \gamma) & = & \lambda\, \tau_i(\gamma)\, , \\
 \tau_i(\gamma+\widehat{\gamma})& =&\tau_i(\gamma)+\tau_i(\widehat{\gamma})\, 
\end{eqnarray*}
for $\lambda\in \mathbb{Q}$ and $\gamma,\widehat{\gamma} \in H^*(X)$.
The subalgebra \(\bbD_\GW^{X+}\subset
\bbD^X_\GW\) of stationary descendents is generated by
$$\big\{ \, \tau_i(\gamma)\, \big| \,  i\ge 0\, , \gamma\in H^{>0}(X,\mathbb{Q})
\,
\big\}\, . $$

We will use Getzler's renormalization  \(\fra_k\)
of the Gromov-Witten descendents{\footnote{We use $\imath$ for the
square root of $-1$. The genus variable $u$ will
usually occur together with $\imath$.}}:
\begin{equation}\label{eq:Getzler}
  \sum_{n=-\infty}^\infty z^n\tau_n=\mathsf{Z}^0+\sum_{n> 0}\frac{(\imath
    uz)^{n-1}}{(1+zc_1)_n}\mathfrak{a}_n+\frac{1}{c_1}\sum_{n<0}\frac{(\imath
    uz)^{n-1}}{(1+zc_1)_n}
\mathfrak{a}_n\, ,
\end{equation}
\begin{equation*}
  \mathsf{Z}^0=\frac{z^{-2}u^{-2}}{\mathcal{S}\left(\frac{zu}{\theta}\right)}
      -z^{-2}u^{-2},  
\end{equation*}
where we use standard notation for the Pochhammer symbol
$$(a)_n=\frac{\Gamma(a+n)}{\Gamma(a)}\, .$$
For example\footnote{The constant term  
  $\frac{1}{24}\int_X \gamma c_2$ in the formula
  does not contribute unless $\gamma \in H^2(X)$.},
\begin{eqnarray}\label{eq:tau-01}
  \tau_0(\gamma) & = & \mathfrak{a}_{1}(\gamma)+\frac{1}{24}
                       \int_X \gamma c_2\, , \\
\tau_1(\gamma) & = & \frac{\imath u}{2}\mathfrak{a}_{2}(\gamma)-\mathfrak{a}_1(\gamma\cdot c_1)\, .
\end{eqnarray}
For $k\geq 2$ and  $\gamma\in H^{>0}(X)$, we have the general formula
\begin{multline}
 \label{eq:tau2a} \tau_k(\gamma)=
\frac{(\imath u)^k}{(k+1)!}\mathfrak{a}_{k+1}(\gamma)-\frac{(\imath u)^{k-1}}{k!}\left(\sum_{i=1}^k\frac1i\right)\mathfrak{a}_k(\gamma\cdot c_1) \\
  +\frac{(\imath u)^{k-2}}{(k-1)!}\left(\sum_{i=1}^{k-1}\frac{1}{i^2}+\sum_{1\le i<j\le k-1}\frac{1}{ij}\right)\mathfrak{a}_{k-1}(\gamma\cdot c_1^2) \, .
\end{multline}


\subsection{The $\GW/\PT$ correspondence for
essential descendents}
\label{sec:stat-gwpt-corr}

The subalgebra
$$\mathbb{D}^{X\bigstar}_{\PT}\subset \mathbb{D}^{X+}_{\PT}$$ of {\em essential
descendents} is generated by
  $$\big\{ \, \tch_i(\gamma)\,  | \, (i\geq 3, \gamma\in H^{>0}(X,\mathbb{Q}))\ 
   \text{or}\  (i=2, \gamma\in H^{>2}(X,\mathbb{Q}))\, \big\}\, .$$
   While closed formulas for the
   full $\GW$/$\PT$ descendent transformation of \cite{PPDC} are
   not known in full generality, the stationary theory is much better understood
   \cite{OblomkovOkounkovPandharipande18}.{\footnote{See \cite{MNOP1,MNOP2}
     for an earlier view of descendents and descendent transformations.}}
   The transformation takes the simplest form when
   restricted to essential descendents.

   The $\GW$/$\PT$ transformation restricted to the
   essential descendents is a linear map
 $$\Cor^\bullet:\mathbb{D}^{X\bigstar}_{\PT}\rightarrow\mathbb{D}^{X}_{\GW}$$
satisfying
$$\Cor^\bullet(1) =1$$
and
 is defined on monomials by
\[\Cor^\bullet\Big(\tch_{k_1}(\gamma_1)\dots\tch_{k_m}(\gamma_m)\Big)=\sum_{P \mbox{ set partition of }\{1,\dots,m\}}\prod_{S\in P}\Cor^\circ\Big(\prod_{i\in S}\tch_{k_i}
  (\gamma_i)\Big)\, .\]
The operations \(\Cor^\circ\) on $\mathbb{D}^{X\bigstar}_{\PT}$ are
\begin{multline}
  \label{eq:decay-intro}
\Cor^\circ\Big(\tch_{k_1+2}(\gamma)\Big)=\frac{1}{(k_1+1)!}\mathfrak{a}_{k_1+1}(\gamma)+\frac{(\imath u)^{-1}}{k_1!}\sum_{|\mu|=k_1-1}\frac{\fra_{\mu_1}\fra_{\mu_2}(\gamma\cdot c_1)}
{\text{Aut}(\mu)}\\  +\frac{(\imath u)^{-2}}{k_1!}\sum_{|\mu|=k_1-2}\frac{\fra_{\mu_1}\fra_{\mu_2}(\gamma\cdot c_1^2)}{\text{Aut}(\mu)}+
  \frac{(\imath u)^{-2}}{(k_1-1)!}\sum_{|\mu|=k_1-3}\frac{\fra_{\mu_1}\fra_{\mu_2}\fra_{\mu_3}(\gamma\cdot c_1^2)}{\text{Aut}(\mu)}\, ,
\end{multline}

 \begin{multline}
    \label{eq:2bump-intro}
    \Cor^\circ\Big(\tch_{k_1+2}(\gamma)\tch_{k_2+2}(\gamma')\Big)
=-\frac{(\imath u)^{-1}}{k_1!k_2!}\fra_{k_1+k_2}(\gamma\gamma') 
-\frac{(\imath u)^{-2}}{k_1!k_2!}\fra_{k_1+k_2-1}(\gamma\gamma'\cdot c_1)
\\-\frac{(\imath u)^{-2}}{k_1!k_2!}
    \sum_{|\mu|=k_1+k_2-2}\max(\max(k_1,k_2),\max(\mu_1+1,\mu_2+1))\frac{\fra_{\mu_1}\fra_{\mu_2}}{\text{Aut}(\mu)}(\gamma\gamma'\cdot c_1)\, ,
  \end{multline}

\begin{multline}
  \label{eq:triple-bump-intro}
  \Cor^\circ\Big(\tch_{k_1+2}(\gamma)\tch_{k_2+2}(\gamma')\tch_{k_3+2}(\gamma'')\Big)
=
\frac{(\imath u)^{-2}|k|}{k_1!k_2!k_3!}\fra_{|k|-1}(\gamma\gamma'\gamma'')\,
, \quad |k|=k_1+k_2+k_3\, . \end{multline}
The above sums are over {\em partitions} of $\mu$ of length $2$ or $3$.
The parts of  
$\mu$ are {\em positive} integers, and we  always write
$$\mu=(\mu_1,\mu_2)\ \ \ \text{and} \ \ \ \mu=(\mu_1,\mu_2,\mu_3)$$ with 
weakly decreasing parts.
In equations \eqref{eq:decay-intro}-\eqref{eq:triple-bump-intro}, 
we have $\ k_i\ge 0$, 
and all occurrences of
$\fra_0$ and \(\fra_{-1}\) are set to $0$.


The above formulas for the $\GW$/$\PT$ descendent correspondence 
are proven here from the vertex operator formulas of  
 \cite{OblomkovOkounkovPandharipande18} by a direct
 evaluation of the leading terms. In the toric case, we
 have the following explicit correspondence
 statement.{\footnote{A straightforward
    exercise using our new conventions is to show the abstract correspondence
    of Theorem \ref{thm:cor-main} is a consequence of
    \cite[Theorem 4]{PPDC}. The novelty of Theorem \ref{thm:cor-main}
  is the closed formula for the transformation.}

\begin{theorem}\label{thm:cor-main}  
Let $X$ be a nonsingular projective toric 3-fold. 
Let
$$ \prod_{i=1}^m \tch_{k_i}(\gamma_i) \in  \mathbb{D}^{X\bigstar}_{\PT}\, .$$ 
Let $\beta\in H_2(X,\mathbb{Z})$ with
$d_\beta = \int_\beta c_1(X)$.
Then, the
$\GW/\PT$ correspondence defined by formulas \eqref{eq:decay-intro}-\eqref{eq:triple-bump-intro} holds:
$$ (-q)^{-d_\beta/2}\,  \Big \langle \prod_{i=1}^m \tch_{k_i}(\gamma_i)\Big\rangle_\beta^{X,\PT}=
(-\imath u)^{d_\beta}\, \Big\langle\Cor^\bullet\Big(\prod_{i=1}^m\tch_{k_i}(\gamma_i)\Big)\Big\rangle^{X,\GW}_\beta\,
,$$
 after the change of variables $-q=e^{\imath u}$.
\end{theorem}

As direct consequence of the formulas  \eqref{eq:decay-intro}-\eqref{eq:triple-bump-intro}, the correspondence taken essential descendents on the
stable pairs side to stationary descendents on the stable pairs side.

\begin{proposition}\label{666}
  Let $D\in \mathbb{D}^{X\bigstar}_{\PT}$. 
  Under the $\GW/\PT$ transformation, we have
  $$\Cor^\bullet(D)\in \mathbb{D}^{X+}_{\GW}\, .$$
\end{proposition}




\subsection{Plan of the paper}
\label{sec:strategy-proof}

The key to our proof of Theorem~\ref{thm:Vir-toric} is an 
intertwining property of \(\Cor^\bullet\) with respect to Virasoro operators
for stable pairs and the Virasoro operators for stable maps. Via the intertwining property,
Theorem~\ref{thm:Vir-toric} is a consequence of the stationary $\GW/\PT$
correspondence of Theorem~\ref{thm:cor-main} and the Virasoro constraints
for  
 the Gromov-Witten theory of toric 3-folds.

 
The algebra \(\bbD^X_\PT \) carries a {\it bumping filtration}{\footnote{The bumping filtration is a filtration of vector spaces.}} 
\begin{equation}\label{nn33}
\bbD^0_{\PT}\subset\bbD^1_\PT\subset
\bbD^2_\PT \subset \bbD^3_\PT \subset
\dots \subset \bbD_\PT^X \, ,
\end{equation} 
where \(\bbD^k_\PT\) is
spanned by the monomials{\footnote{Via the empty monomial $(m=0)$,
$\bbD^0_\PT$ is spanned by the unit $1$.}}
 $$\prod_{i=1}^m\tch_{k_i}(\gamma_i)$$
for which \(\gamma_{s_1}\cdots\gamma_{s_l}=0\) for all subsets
$$S=\{s_1,\ldots,s_l\}\subset \{1,\dots,m\}\, ,\ \ \  l> k \, .$$
In general the filtration \eqref{nn33}
has infinite length. But
 if we restrict the filtration to  \(\bbD_\PT^{X\bigstar}\), the filtration truncates 
since 
$$\bbD_\PT^3\cap \bbD_\PT^{X\bigstar} = \bbD_\PT^{X\bigstar}\, .$$
The correspondence  
$$\Cor^\bullet: \mathbb{D}^{X\bigstar}_{\PT}\rightarrow\mathbb{D}^{X+}_{\GW}$$ 
respects the analogous
bumping filtration $\bbD^k_\GW\cap \bbD_\GW^{X+}$ on
$\bbD_\GW^{X+}$
with respect to the monomials
$$\prod_{i=1}^m\tau_{k_i}(\gamma_i)$$
for which \(\gamma_{s_1}\cdots\gamma_{s_l}=0\) for all subsets
$$S=\{s_1,\ldots,s_l\}\subset \{1,\dots,m\}\, ,\ \ \  l> k \, .$$
Our
proof of the intertwining is separated into a calculation for each of
the four steps of the restriction of the bumping filtration
on $\bbD^{X\bigstar}_\PT$.



We discuss the Virasoro constraints for  Gromov-Witten theory
in Section~\ref{sec:virasoro-constraints} and for stable pairs
in Section \ref{pppp}. The stationary Virasoro constraints
of Theorem \ref{thm:Vir-toric} are proven in
Section \ref{jjjj9} modulo the intertwining of Theorem \ref{thm:intertw}.
The proof of the intertwining property is given in four steps:

\begin{enumerate}
\item[($\mathsf{0}$)] We start in Section~\ref{sec:main-result} with the special case
where $D\in \bbD^0_\PT\cap\bbD_\PT^{X\bigstar} $
is the trivial monomial 1. The result is 
Proposition~\ref{prop:CTT} of Section \ref{sec:const-term-comp}.

\item[($\mathsf{1}$)] For $D\in \bbD^1_\PT\cap\bbD_\PT^{X\bigstar}$,
the required results are proven in Section~\ref{mono9}.

\item[($\mathsf{2}$)]
Proposition \ref{prop:double-bump1} and Proposition \ref{prop:double-bump2}
of Section~\ref{sec:bumping-computations}
imply 
the intertwining property for $D\in \bbD_\PT^2\cap\bbD_\PT^{X\bigstar}$.

\item[($\mathsf{3}$)]
We treat $D\in \bbD_\PT^3\cap\bbD_\PT^{X\bigstar}=\bbD_\PT^{X\bigstar}$
in Proposition \ref{prop:triple-bump} of Section \ref{sec:bumping-computations}
to complete the proof of Theorem \ref{thm:intertw}.
\end{enumerate}

Let $S$ be a nonsingular projective toric surface.
As a consequence of the stationary Virasoro constraints for 
$$X=S\times \PP^1 \ \ \text{and}\ \  \beta=n[\PP^1]\, ,$$
we obtain {\em new} Virasoro
constraints for the integrals of the tautological classes over
Hilbert schemes of points  
\(\Hilb^n(S)\) of  surfaces $S$ in Section~\ref{sec:viras-constr-surf}.
The case of {\em all} simply connected nonsingular projective surfaces
is proven in \cite{Mor}.

After a review of the $\GW/\PT$ descendent correspondence from the
perspective of \cite{OblomkovOkounkovPandharipande18}
in Section \ref{gwptrev}, we complete the proof of
Theorem \ref{thm:cor-main}  
in Section \ref{sec:residue-computation}. A list of descendent series in degree 1 for
$\mathbb{P}^3$ is given in Section \ref{sec:appendix}.

\subsection{Acknowledgments} 
We are grateful to
D.~Maulik, N.~Nekrasov, G.~Oberdieck, D.~Oprea,
A.~Pixton, J.~Shen, R.~Thomas, and Q.~Yin for  
many conversations about descendents and
descendent correspondences.

A.~Ob. was partially supported by NSF CAREER grant
DMS-1352398 and Simons Foundation. A.~Ob. also would like to thank the Forschungsinstitut f\"ur Mathematik 
and the  Institute for Theoretical Studies at ETH Z\"urich for hospitality during the visits in November 2018, June 2019, and January 2020.
The paper is based upon work supported by the National Science Foundation under Grant No. 1440140 while the first two authors were in residence at the Mathematical Sciences Research Institute in Berkeley during the Spring semester of 2018.
A.~Ok. was partially supported by the Simons Foundation as a Simons Investigator.
A.~Ok. gratefully acknowledges funding by the Russian
Academic Excellence Project ’5-100’ and RSF grant 16-11-10160.
R.~P. was partially supported by 
SNF-200020-182181, SwissMAP, and
the Einstein Stiftung.

The project has received funding from the European Research
Council (ERC) under the European Union Horizon 2020 research and
innovation program (grant agreement No. 786580).

\section{Virasoro constraints for Gromov-Witten theory}
\label{sec:virasoro-constraints}

\subsection{Overview}
We will discuss here the \Vir constraints for stable maps. 
The constraints are equivalent to a procedure 
for removing the descendents of the canonical class. 
The procedures may be interpreted as series of the reactions
(similar to the reactions discussed in the context of
the $\GW/\PT$ descendent correspondence in \cite[Section 3]{OblomkovOkounkovPandharipande18}). Our goal is to write the \Vir constraints for
Gromov-Witten theory in a form which is as close as possible to the \Vir
constraints of Conjecture \ref{vvirr} for stable pairs.

\subsection{ Gromov-Witten constraints: original form}
\label{sec:stat-spec-viras}
The Virasoro constraints in Gromov-Witten theory were first proposed{\footnote{The full conjecture also involves ideas of S. Katz.}}
in \cite{EHX}.
We recall here the original form following
\cite{Pandharipande03}. In
Section \ref{sec:gw-viras-constr}, a reformulation which is more
suitable for the $\GW/\PT$ correspondence
will be presented.

In the discussion below, we fix a basis of \(H^*(X)\),
\begin{equation}\label{bbbsss}
  \gamma_0,\dots,\gamma_r\, , \quad \gamma_i\in H^{p_i,q_i}(X)\, ,
  \end{equation}
for which \(\gamma_0=1\), \(\gamma_1=c_1\), and
\(\gamma_{r}=[\pt]\).
We assume{\footnote{For Calabi-Yau 3-folds, the Virasoro invariants
    are a consequence of the string and dilaton equations (and there
    are no non-trivial stationary invariants).}}
$c_1\ne 0$.
We also fix a dual basis
$$\gamma^\vee_0,\dots,\gamma^\vee_r\, , \quad  \int_X\gamma_i \gamma_j^\vee=\delta_{ij}\, .$$

The standard method of describing of the Virasoro constraints uses the generating function for the Gromov-Witten
invariants (see \cite[section 4]{Pandharipande03}):
\[F^X=\sum_{g\ge 0}u^{2g-2}\sum_{\beta\in H_2(X,\ZZ)} q^\beta\sum_{n\ge 0}\sum_{\stackrel{a_1,\dots,a_n}{k_1,\dots,k_n}}t_{k_1}^{a_1}\dots t_{k_1}^{a_1}\dots t_{k_n}^{a_n}\, 
  \big\langle \tau_{k_1}(\gamma_1)\dots\tau_{k_n}(\gamma_n)
  \big\rangle_{g,\beta}^{X,\mathrm{Con}}\, ,\]
where $\big\langle, \big\rangle^{X,\mathrm{Con}}_{g,\beta}$ is the standard integral over stable maps with connected domains (and stable contracted
components of all genera are permitted).

The degree \(\beta=0\) summand $F_0^X$ of $F^X$ does not require knowledge of curves in $X$.
We further split the degree 0 summand into summands of genus \(g\le 1\) and genus \(g\ge 2\):
\[F^X_0=F^X_{0,g\le 1}+F^X_{0,g\ge 2}\, .\]
The $g\leq 1$ summand takes the form
\[F_{0,g\leq 1}^X=u^{-2}\sum_{i,j,k}\bigg(\frac{t_0^it_0^jt_0^k}{3!}+\frac{t_0^it_0^jt_1^kt_0^0}{2!}\bigg)\int_X\gamma_i\gamma_j \gamma_k-\sum_i\bigg(\frac{t_0^i}{24}+\frac{t_1^it_0^0}{24}\bigg)\int_X\gamma_i c_2+ \ldots,\]
where the dots stand for terms divisible by $(t_0^0)^2$.
The $g\geq 2$ summand \(F^X_{0,g \ge 2}\) 
is determined by the string and dilaton equations from
the constant maps contributions of \cite[Theorem 4]{FabP}.

Let $\widetilde{F}^X$ be the summand of $F^X$ with $\beta\neq 0$.
We define
$$Z_{0,*}^X=\exp(F_{0,*}^X)\, ,\quad \widetilde{Z}^X=\exp(\widetilde{F})\, .$$
The Gromov-Witten
bracket $\big\langle, \big\rangle^{X,\mathrm{GW}}_{g,\beta}$
introduced in Section \ref{p3f} corresponds to
the partition function
$$Z^X_{0,g\le 1}\cdot \widetilde{Z}^X =
\sum_{g\ge \mathbb{Z}}u^{2g-2}\sum_{\beta\in H_2(X,\ZZ)} q^\beta\sum_{n\ge 0}\sum_{\stackrel{a_1,\dots,a_n}{k_1,\dots,k_n}}t_{k_1}^{a_1}\dots t_{k_1}^{a_1}\dots t_{k_n}^{a_n}\, 
  \big\langle \tau_{k_1}(\gamma_1)\dots\tau_{k_n}(\gamma_n)
  \big\rangle_{g,\beta}^{X,\mathrm{GW}}\, .$$
The full partition function 
$$Z^X=\exp(F^X)=Z^X_{0,g\le 1} \cdot Z^X_{0,g \ge 2}\cdot \widetilde{Z}^X\, $$
corresponds to the standard disconnected
Gromov-Witten
bracket $\big\langle, \big\rangle^{X,\bullet}_{g,\beta}$,
\[Z^X=\sum_{g\ge 0}u^{2g-2}\sum_{\beta\in H_2(X,\ZZ)} q^\beta\sum_{n\ge 0}\sum_{\stackrel{a_1,\dots,a_n}{k_1,\dots,k_n}}t_{k_1}^{a_1}\dots t_{k_1}^{a_1}\dots t_{k_n}^{a_n}\, 
  \big\langle \tau_{k_1}(\gamma_1)\dots\tau_{k_n}(\gamma_n)
  \big\rangle_{g,\beta}^{X,\bullet}\, .\]

The Virasoro operators \(\mathrm{L}_k\), \(k\in \mathbb{Z}_{\ge -1}\) are differential operators which satisfy the Virasoro relations,
\[[\mathrm{L}_k,\mathrm{L}_\ell]=(k-\ell) \mathrm{L}_{k+\ell}\, ,\]
and annihilate the partition function
\begin{equation}\label{vvirrt}
\mathrm{L}_k \, Z^X=0 \, .
\end{equation}
For 3-folds $X$, 
the operators are defined by:
\begin{align*}
  \mathrm{L}_k=& \sum_{m=0}^\infty\sum_{i=0}^{k+1}\bigg( [p_a+m-1]_i^k(C^i)^b_a\tilde{t}^a_m\partial_{b,m+k-i}\\
  &+ \frac{u^2}{2}(-1)^{m+1}[-p_a+1-m]_i^k(C^i)^{ab}\partial_{a,m}\partial_{b,k-m-i-1}  \bigg)\\
               &  +\frac{u^{-2}}{2}(C^{k+1})_{ab}t_0^at^b_0\\
  &-\frac{\delta_{k}}{24}\int_Xc_1c_2\, ,
\end{align*}
where the Einstein conventions for summing over  repeated indices are
followed{\footnote{Here $\delta$ denotes the $\delta$-function: $\delta_k=0$ unless $k=0$,  $\delta_0=1$, and $\delta_{ab}=\delta_{a-b}$.}},
\[\tilde{t}_m^a=t^a_m-\delta_{a0}\delta_{m1}\, ,\quad \partial_{a,m}=\partial/\partial t_m^a\, ,\]
and $[x]^k_j=e_{k+1-j}(x,x+1,\dots,x+k)$.
The tensors in the equation are defined in terms of the dual basis:
\[(C^i)^a_b=\int_X\gamma_a^\vee c_1^i\gamma_b\, ,\quad (C^i )_{ab}=\int_X\gamma_ac_1^i\gamma_b\, ,\quad (C^i)^{ab}=\int_X\gamma_a^\vee c_1^i\gamma_b^\vee\, .\]

\subsection{Gromov-Witten constraints:
  correspondence form}
\label{sec:gw-viras-constr}

We rewrite here the Virasoro constraints of Section
\ref{sec:virasoro-constraints} in the form
most natural for the $\GW/\PT$ descendent correspondence.
Since all of our results are for toric varieties,
we specialize our discussion here to the case where
$X$ is a nonsingular projective 3-fold with only $(p,p)$-cohomology.

We start by
defining derivations \(\mathrm{R}^j_k\) and quadratic differentials
$\mathrm{B}^k$ on
\(\mathbb{D}_{\GW}^X\) by fixing the action on the generators:
\vspace{5pt}
\begin{enumerate}
\item[$\bullet$] The action of the derivation
  $\mathrm{R}^j_k$ on $\tau_i(\gamma)$
  for $k\geq -1$, $0\leq j \leq 3$, and  \(\gamma\in H^{2d}(X)\) is
$$\mathrm{R}^j_k(\tau_i(\gamma))=[i+d-1]^k_j\, \tau_{k+i-j}(\gamma\cdot c_1^j)\, ,$$
 where  $[x]^k_j=e_{k+1-j}(x,x+1,\dots,x+k)$ and
all terms $\tau_{\ell<-2}(\theta)$ are set to zero. As a special case,
\[\mathrm{R}^j_{-1}(\tau_i(\gamma)) = \delta_{j}\, \tau_{i-1}(\gamma)\, .\]
We will use the notation $\mathrm{R}_k = \sum_{j=0}^3 \mathrm{R}_k^j\, .$

\vspace{5pt}
\item[$\bullet$] The action of the quadratic differential
  $\mathrm{B}^k$ on $\tau_0(\gamma)\tau_0(\gamma')$ is 
  \[\mathrm{B}^k(\tau_0(\gamma)\tau_0(\gamma'))=\int_X\gamma\gamma' c_1^k\, .\]
  On all other quadratics terms, $\mathrm{B}^k$ acts trivially.

\end{enumerate}
\vspace{5pt}

The differential operators \(\rmL^\GW_k\), for $k\geq -1$, are then defined by
the formula:
\[\rmL_k^{\GW}= \mathrm{R}_k+\frac{u^{-2}}2\mathrm{B}^{k+1}+\frac{(\imath u)^2}{2}\CT_k-\frac{\delta_{k}}{24}\int_Xc_1c_2\, , \]
 where $\CT_k = \sum_{j=0}^3 \CT_k^j$ and
\begin{equation}\label{ppdd}
  \CT^j_k=\sum_{m=-1}^{k-j+2}(-1)^{m+1}[2-m-d_L]_j^k\, :\tau_{m-1}\tau_{-m+k-j}(c_1^j):\ ,
  \end{equation}
where \(d_L\) is the degree of the left term in the
co-product{\footnote{Define the element
\begin{equation*}
\tau_a\tau_b(\gamma) = \sum_i \tau_a(\gamma^L_i)
\tau_b(\gamma^R_i) \, \in\,  \bbD^X_{\GW}
\end{equation*}
where $\sum_i \gamma^L_i \otimes \gamma^R_i$
is the K\"unneth decomposition of the product, 
$$\gamma\cdot\Delta \in H^*(X\times X)\, ,$$
with the diagonal $\Delta$.}}
(as in Section \ref{sec:pt-vir-constraints}). In formula
\eqref{ppdd},
the symbol \(::\) stands for the normal ordering convention:
{\em all negative descendents \(\tau_{<0}(\gamma)\)
are on the left of the positive descendents}.

A calculation then yields the Virasoro bracket and the following
bracket with $\tau_k(\mathsf{p})$:
\begin{equation}\label{qqq999}
[\rmL^\GW_n,\rmL^\GW_k]=(n-k)\, \rmL^\GW_{n+k}\, , \ \ \ \ \ \ [\rmL^\GW_n,(k+1)!\, \tau_k(\mathsf{p})]=  
(k+n+2)!\, \tau_{n+k}(\mathsf{p})\, .
\end{equation}

\vspace{8pt}

\begin{theorem}\label{vvirrGW}
(Givental \cite{Giv})  Let $X$ be a nonsingular projective toric 3-fold, 
and let\\ $\beta \in H_2(X,\mathbb{Z})$. 
For all $k\geq -1$ and
\(D\in \bbD^X_{\GW}\),
we have
 \[\Big\langle \rmL^\GW_k(D) \Big\rangle_\beta^{X,\bullet}=0\, .\]
\end{theorem}

Theorem \ref{vvirrGW}, which is exactly equivalent to
constraints \eqref{vvirrt} for toric 3-folds,
was proven by Givental in two steps:
\begin{enumerate}
\item[(i)] Using the
virtual localization formula of \cite{GP}, the Gromov-Witten
theory of $X$ is expressed in
terms of graphs sums with descendent integrals over the moduli spaces
of curves $\oM_{g,n}$
at the vertices.
\item[(ii)] The Virasoro constraints, conjectured by Witten \cite{Wit}
for $\oM_{g,n}$ and proven in \cite{Konts},
are then used to establish
 the Virasoro constraints for $X$.
\end{enumerate}
A second proof of Theorem \ref{vvirrGW}, via the Givental-Teleman classification{\footnote{We refer
the reader to \cite{P-ICM} for an introduction.}} of
semisimple CohFTs, was given in \cite{Tel}. For varieties
with non-semisimple Gromov-Witten theory,
the Virasoro constraints are known in very few cases.{\footnote{The main
known examples are based on the Virasoro constraints for
curves proven in \cite{OPVir}.}}

  
\subsection{Gromov-Witten constraints: stationary form}
\label{sec:changes-variables-gw}
We rewrite the \Vir constraints  in
Gromov-Witten theory of Section \ref{sec:gw-viras-constr}
in a form which preserves the algebra of  stationary descendents,  
$$  \bbD^{X+}_{\GW}\subset  \bbD^X_{\GW}
\, .$$



We fix a basis \eqref{bbbsss} of the cohomology of $X$
which satisfies the following
further conventions.
Let $$\gamma_1,\dots,\gamma_s\in H^2(X)$$
be a basis with \(\gamma_1=c_1\).
Let $$\gamma_{2s},\dots,\gamma_{s+1}\in H^4(X)$$
be a dual basis with respect to the Poincar\'e pairing.
Let  
$$\gamma_0=1\in H^0(X)\, ,\ \ \ \gamma_{2s+1}=\pt\in H^6(X)$$
span the rest of the cohomology.{\footnote{To match
    with \eqref{bbbsss}, $r=2s+1$.}}
 The K\"unneth decomposition of the diagonal is
\[\Delta=\sum_{i=0}^{2s+1}\gamma_{i}\otimes\gamma_{2s+1-i}\, .\]

Consider the  term $\CT_k$.
The only place for descendents of \(1\) to appear in the operator
\(\mathrm{L}^\GW_k\) 
is in
\(\CT_k^0\). As most of the terms of \(\CT^0_k\) vanish by definition,
we find
\begin{equation}\label{ffggyy}
\frac{1}{2}\CT^0_k=(k+1)!\, :\tau_0(1)\tau_{k-1}(\pt):\, .
\end{equation}
We denote the rest of the term by \(\CT'_k\),
$$\CT_k = \CT'_k + \CT_k^0\, .$$
Inside the bracket $\langle, \rangle^{X,\bullet}_\beta$,
 the insertion \(\tau_0(1)\) can be removed by the string equation. 
We are therefore led to define  the
operator
\[\mathcal{L}_k^\GW=\frac{(\imath u)^2}2\CT'_k+\mathrm{R}_k
+\frac{u^{-2}}2\mathrm{B}^{k+1}
  +(\imath u)^2(k+1)!\,\mathrm{R}_{-1}\tau_{k-1}(\pt)\, ,\quad\ \ \CT'_k=\sum_{j>0}\CT^j_k\, ,\]
where \(\rmR_k=\sum_{j=0}^3 \rmR_k^j\) and
\(\mathrm{R}_{-1}\) is the differentiation defined on the generators by
\[\mathrm{R}_{-1}\tau_{k}(\gamma)=\tau_{k-1}(\gamma)\, .\]
Inside the bracket  $\langle, \rangle^{X,\bullet}_\beta$,
we have\footnote{Note
  \(\mathcal{L}_0^{\GW}=\widetilde{\rmL}_0^{\GW}\).}
\begin{equation}\label{kk33}
  \mathcal{L}_k^{\GW}\, \eql\, \widetilde{\rmL}_k^{\GW}+(\imath u)^2(1-\delta_k)(k+1)!\, \widetilde{\rmL}_{-1}^{\GW}\tau_{k-1}(\pt)\, ,
\end{equation}
where we have modified the \Vir operators to exclude the descendents of \(1\):
\begin{equation}
  \label{eq:L-til}
  \widetilde{\rmL}_k^\GW\, =\, \rmL_k^\GW - \frac{(\imath u)^2}{2} \CT^0_k \, =\,
\frac{(\imath u)^2}{2}\CT'_k
+ \rmR_k+\frac{u^{-2}}2\mathrm{B}^{k+1}-\frac{\delta_k}{24}\int_Xc_1c_2\, . 
\end{equation}
Though the operators \(\mathcal{L}_k^\GW\) no longer satisfy the \Vir bracket,
the  operators \(\mathcal{L}_k^\GW\)
preserve the  subalgebra  \(\mathbb{D}_{\GW}^{X+}\subset \mathbb{D}_{\GW}^X\).

\begin{proposition}\label{ggkk}
Let $X$ be a nonsingular projective toric 3-fold, 
and let\\ $\beta \in H_2(X,\mathbb{Z})$. 
For all $k\geq -1$ and 
\(D\in \bbD^{X+}_{\GW\circ}\),
we have
 \[\Big\langle \calL^\GW_k(D) \Big\rangle_\beta^{X,\bullet}=0\, .\]
\end{proposition}

\begin{proof}
  The case \(k=0\) follows because \[\mathcal{L}_0^{\GW}-\rmL_0^{\GW}=\CT_0^0=2:\tau_{0}(1)\tau_{-1}(\pt):\]
  and \(\Big\langle\CT_0^0\dots \Big\rangle^{X,\bullet}_\beta=0\). For the
  other case the argument is below.
  
  Using \eqref{kk33} and \eqref{eq:L-til}, we have
\begin{multline}\label{kkk222}
\Big\langle \calL^\GW_k(D) \Big\rangle_\beta^{X,\bullet} =
\Big\langle \rmL_k^\GW(D) +(\imath u)^2(k+1)!\, \rmL_{-1}^{\GW}(\tau_{k-1}(\pt) D)
\Big \rangle^{X,\bullet}_\beta\\
-\frac{(\imath u)^2}{2} \Big \langle \CT^0_k(D)
+(\imath u)^2(k+1)!\, \CT^0_{-1} (\tau_{k-1}(\pt)D)
\Big \rangle^{X,\bullet}_\beta\, .
\end{multline}
The first bracket on the right side of \eqref{kkk222}
vanishes by Theorem \ref{vvirrGW}.
We can write the second bracket on the right as
\begin{multline*}
\frac{(\imath u)^2}{2}\Big \langle \CT^0_k(D)
+(\imath u)^2(k+1)!\, \CT^0_{-1} (\tau_{k-1}(\pt)D)
\Big \rangle^{X,\bullet}_\beta =\\
{(\imath u)^2}\Big \langle (k+1)!\, \tau_0(1) \tau_{k-1}(\pt) D
+(\imath u)^2(k+1)!\, \tau_0(1)\tau_{-2}(\pt) \tau_{k-1}(\pt)D 
\Big \rangle^{X,\bullet}_\beta
\end{multline*}
using  \eqref{ffggyy}. 
The right side of the
above equation, after applying the commutator \eqref{ActionVacuum}, is
$${(\imath u)^2}\Big \langle (k+1)!\, \tau_0(1) \tau_{k-1}(\pt) D
+(\imath u)^2(k+1)!\, \tau_{-2}(\pt) \tau_0(1) \tau_{k-1}(\pt)D 
\Big \rangle^{X,\bullet}_\beta\, ,$$
which vanishes after applying \eqref{pp23pp}.
\end{proof}

In our study of the $\GW/\PT$ descendent correspondence,
we are interested in the Gromov-Witten 
bracket $\big\langle, \big\rangle^{X,\mathrm{GW}}_{g,\beta}$
of Section \ref{p3f}
instead of the standard disconnected bracket
$\big\langle, \big\rangle^{X,\bullet}_{g,\beta}$.
Therefore, the following result is important for
our study.

\begin{proposition}\label{ggkk17}
Let $X$ be a nonsingular projective toric 3-fold, 
and let\\ $\beta \in H_2(X,\mathbb{Z})$. 
For all $k\geq -1$ and 
\(D\in \bbD^{X+}_{\GW\circ}\),
we have
 \[\Big\langle \calL^\GW_k(D) \Big\rangle_\beta^{X,\GW}=0\, .\]
\end{proposition}

\begin{proof}
  Since $\calL^\GW_k$ preserves $\bbD^{X+}_{\GW}$, we have
  $$\calL^\GW_k(D) \in \bbD^{X+}_{\GW}\, .$$
  Since the Gromov-Witten invariants corresponding
  to collapsed connected components of genus at least 2
  {\em always} vanish
  in the presence of stationary descendents,
  $$
  \Big\langle \calL^\GW_k(D) \Big\rangle_\beta^{X,\bullet}
= Z^X_{0,g\geq 2}\Big{|}_{\{t^i_k=0\}}  \cdot \Big\langle \calL^\GW_k(D) \Big\rangle_\beta^{X,\GW} \, .$$
  Since $\Big\langle \calL^\GW_k(D) \Big\rangle_\beta^{X,\bullet}$ vanishes
  by Proposition \ref{ggkk} and
  $$Z^X_{0,g\geq 2}\Big{|}_{\{t^i_k=0\}} = \exp\left(\sum_{g=2}^\infty
  (-1)^g u^{2g-2}\, \frac{\chi(X)}{2} \, \int_{\overline{M}_g}\lambda_{g-1}^3\right)$$
  is invertible{\footnote{See \cite[Theorem 4]{FabP} for the evaluation.}},
  $\Big\langle \calL^\GW_k(D) \Big\rangle_\beta^{X,\GW}$ also vanishes.
\end{proof}



\section{Theorem \ref{thm:Vir-toric}: Virasoro constraints for stable pairs} \label{pppp}

\subsection{Intertwining property}
\label{sec:intertw-prop}

We have already defined
the operators $\rmL_k^\PT$ and $\calL^\PT_k$
on $\mathbb{D}_\PT^X$ in
Sections \ref{sec:pt-vir-constraints} and \ref{virbra}:
$$\rmL_k^\PT= \CT_k+\rmR_k \, ,
\ \ \
\calL^\PT_k=
\rmL_k^\PT+(k+1)!\, \rmL_{-1}^\PT\ch_{k+1}(\pt)\, ,$$
for $k\geq -1$.
We also have
\begin{equation}
  \label{ppp999}
[\rmL_n^\PT,\rmL_k^\PT]\, \eql\, (k-n)\rmL_{n+k}^\PT\, , \ \ \ 
[\rmL_n^\PT,(k-1)!\, \ch_k(\pt)]\, \eql\, (n+k)!\, \ch_{n+k}(\pt)\, .
\end{equation}
Equations \eqref{ppp999} are parallel to equations \eqref{qqq999} in Gromov-Witten
theory.

The main computation of the paper is the 
{\em intertwining property} which relates
the Virasoro operators for the stable pairs and Gromov-Witten
theories via the descendent correspondence.
We separate the argument into two cases: $k\leq  0$
and $k \geq 1$. Proposition \ref{PPP111} covers the
$k\leq 0$ case. 
The
$k\geq 1$  case treated in Theorem \ref{thm:intertw} is harder.

Proposition \ref{PPP111} is
proven in Section \ref{ppff11} except for steps at the
end of the proof which will be completed in the proof of
Theorem \ref{thm:intertw} in
Sections  \ref{sec:main-result}-\ref{sec:bumping-computations}.
The argument is an intricate calculation based on a strategy of filtration.

\begin{proposition}\label{PPP111}
  For  \(k=-1,0\) and 
    \(D\in \mathbb{D}^{X\bigstar}_{\PT}\), we have 
 \[\Cor^\bullet\circ \rmL_{k}^\PT(D)\, = \, (\imath u)^{-k}\, 
   \widetilde{\rmL}_{k}^\GW\circ \Cor^\bullet(D)\, \]
 after the restrictions $\tau_{-2}(\pt)=1$ and $\tau_{-1}(\pt)=0$.
  \end{proposition}


\begin{theorem}\label{thm:intertw} For all \(k\ge 1\) and
  \(D\in \mathbb{D}^{X\bigstar}_{\PT}\), we have 
  \[\Cor^\bullet\circ \rmL_k^\PT(D)=(\imath u)^{-k}\, \widetilde{\rmL}_k^\GW\circ \Cor^\bullet(D)\, \]
  after the restrictions
  $\tau_{-2}(\pt)=1$ and $\tau_{-1}(\gamma)=0$ for \(\gamma\in H^{>2}(X)\).
\end{theorem}


The evaluations of left sides of the equalities in Proposition \ref{PPP111} and
Theorem \ref{thm:intertw} require a slight generalization of the formulas
\eqref{eq:decay-intro}-\eqref{eq:triple-bump-intro} which govern the descendent correspondence on
$\mathbb{D}^{X\bigstar}_\PT$. Additional rules are required
for
\begin{equation}\label{fpp23}
  \tch_0(\gamma),  \tch_1(\gamma)\ {\text{for $\gamma \in H^{>0}(X)$}}
  \ \ {\text{and}}\ \ 
  \  \tch_2(\delta)\ {\text{for  $\delta \in H^2(X)$.}}
  \end{equation}

The required rules take a very simple form since
$\rmL_k^\PT(D)$  
is at most linear{\footnote{$\rmL_1^\PT(D)$ has a single quadratic
    term in the classes \eqref{fpp23} given by $\tch_1(\mathsf{p})
    \tch_2(c_1)$ which causes no difficulty since $\tch_1(\mathsf{p})$
  does not interact.}}
in the classes \eqref{fpp23} over $\mathbb{D}^{X\bigstar}_\GW$:
\begin{equation}\label{eq:except-bump}
  \Cor^\circ(\tch_0(\gamma)) = -\int_X\gamma\, , \ \ \ \ \
  \Cor^\circ(\tch_0(\gamma)M)=0\, ,
\end{equation}

\begin{equation*}
  \Cor^\circ(\tch_1(\gamma)) =  0 \, , \ \ \ \ \ \ \ \ \ \ \ \ 
  \Cor^\circ(\tch_1(\gamma)M) = 0\, , 
\end{equation*}
where $M\in \mathbb{D}^{X\bigstar}_\PT$.
For $\Cor^\circ(\tch_2({\delta})M)$ with $M\in \mathbb{D}^{X\bigstar}_\PT$,
formulas \eqref{eq:decay-intro}-\eqref{eq:triple-bump-intro} apply unchanged.
The above rules are compatible with the $\GW/\PT$ descendent correspondence
and will be established in Section \ref{sec:residue-computation}.


The restrictions $\tau_{-2}(\pt)=1$ and
\(\tau_{-1}(\pt)=0\)  in Proposition \ref{PPP111} are
well-defined since both  $\Cor^\bullet\circ \rmL_k^\PT(D)$
and $\widetilde{\rmL}_k^\GW\circ \Cor^\bullet(D)$, \(k=0,-1\) will be seen to lie in 
the commutative algebra generated by $\tau_{-2}(\pt)$, \(\tau_{-1}(\gamma)\), and
$\bbD_\GW^{X+}$.  The commutation with
 \(\tau_{-2}(p)\) and \(\tau_{-1}(\pt)\) follows from
\eqref{pp23pp}.

Similarly, the restrictions  $\tau_{-2}(\pt)=1$ and
\(\tau_{-1}(\gamma)=0\) for \(\gamma\in H^{>2}(X)\)  in Theorem \ref{thm:intertw} are well-defined since both  $\Cor^\bullet\circ \rmL_k^\PT(D)$
and $\widetilde{\rmL}_k^\GW\circ \Cor^\bullet(D)$, \(k>0\) will be seen to lie in 
the commutative algebra generated by $\tau_{-2}(\pt)$, \(\tau_{-1}(\gamma)\), and
$\bbD_\GW^{X\bigstar}$. The algebra \(\bbD_{\GW}^{X\bigstar}\) is generated
by the {\it essential descendents}
$$\big\{ \, \tau_i(\gamma)\,  | \, (i\geq 0, \gamma\in H^{>0}(X,\mathbb{Q}))\ 
   \text{or}\  (i=0, \gamma\in H^{>2}(X,\mathbb{Q}))\, \big\}\, .$$
Again, commutation follows from \eqref{pp23pp}.


\subsection{Conventions for $(-1)!\ch_1(c_1)$}
In order to complete the definitions of
the left sides of Proposition \ref{PPP111}
and Theorem \ref{thm:intertw}, we must also include the term
$(-1)! \ch_1(c_1)$ in the descendent correspondence $\Cor^\bullet$ since
such terms occur in $\rmL_k^\PT$.

\vspace{8pt}
\noindent $\bullet$ The first case is
\[\Cor^\circ((-1)!\mathsf{ch}_1(c_1))=0\, .\]

\vspace{5pt}
\noindent $\bullet$
The non-vanishing  bumping term is given by 
 \begin{multline}
    \label{eq:2bump-minus-1-factorial}
    \Cor^\circ\Big((-1)!\ch_{1}(c_1)\tch_{k_1+2}(\gamma)\Big)
=-\frac{(\imath u)^{-1}}{k_1!}\Bigg(\fra_{k_1-1}(c_1\gamma) 
+(\imath u)^{-1}\fra_{k_1-2}(c_1\gamma\cdot c_1)
\\+(\imath u)^{-1}k_1
    \sum_{|\mu|=k_1-3}\frac{\fra_{\mu_1}\fra_{\mu_2}}{\text{Aut}(\mu)}(c_1\gamma\cdot c_1)\Bigg)\, ,
  \end{multline}
  where \(k_1\ge 2\).
\vspace{8pt}

\noindent $\bullet$ The higher bumping term is 
\[ \Cor^\circ((-1)! \ch_1(c_1)\tch_{k_1+2}(\gamma)\tch_{k_2+2}(\gamma'))=\frac{(\imath u)^{-2}(k_1+k_2-1)}{k_1!k_2!}\fra_{k_1+k_2-2}(c_1\gamma\gamma')\, ,\]
\(k_1,k_2\ge 0, k_1+k_2>1\).
There is also an exceptional higher bumping term
\[ \Cor^\circ((-1)! \ch_1(c_1)\tch_{2}(\gamma)\tch_3(\gamma'))
  =\tau_{-2}(c_1\gamma\gamma')\, .\]

\subsection{Proof of Proposition \ref{PPP111} }  \label{ppff11}

  The cases  \(k=-1,0\) are special in two ways:
\begin{enumerate}
\item[(i)]  We must  use the exceptional cases of the operator \(\Cor^\circ\),
  in the analysis for $k=-1,0$.
 \item[(ii)] While  the operator
  \(\widetilde{\rmL}_{k}^\GW\) for \(k=-1,0\) has quadratic part
  $\frac{u^{-2}}{2}B^{k+1}$,
   \(\widetilde{\rmL}_k^\GW\) is a first order operator
   acting on the stationary sector of descendent algebra for $k>0$.
\end{enumerate}
   For these reasons, we treat the $k=-1,0$ cases separately here.

   The restrictions in the statement of Proposition \ref{PPP111} 
   allow us   freely use
\begin{equation}\label{bb55}
  \mathsf{ch}_0(\mathsf{p})=-1\, ,
  \end{equation}
   which is compatible with $\Cor^\bullet$. Similarly, we can use
   \begin{equation}  \mathsf{ch}_1(\mathsf{p})=0\, . \label{bb66}
     \end{equation}

     
 Let us write down the corresponding operators explicitly:
    \[\rmL_{-1}^{\PT}=\rmR_{-1}-(-1)!\, \mathsf{\ch}_1(c_1),\quad \widetilde{\rmL}_{-1}^\GW=\rmR_{-1}+\frac{u^{-2}}{2} \rmbb^0\, .\]
    \[\rmL_{0}^{\PT}=\rmR_0-\widetilde{\mathsf{ch}}_2(c_1)-\frac{1}{2}
   \mathsf{ch}_1\mathsf{ch}_1(c_1),\quad \widetilde{\rmL}_0^{\GW}=\rmR_0+\frac{u^{-2}}{2}\rmbb^1-\tau_0(c_1)-\frac{1}{24}\int_Xc_1c_2\, .\]
 We have used \eqref{bb55} for $\rmL_{-1}^{\PT}$. For
 ${\rmL}_0^{\PT}$, only the $d_L=d_R=2$ summand is
 nonzero by \eqref{bb66}.

 \vspace{8pt}
 \noindent{\bf{Step 1. }}  We check the statement for \(D=1\).
\vspace{8pt}

 The left side of
 the equality of Proposition  \ref{PPP111} for \(k=-1\) is
    \[\Cor^\bullet(\rmL_{-1}^\PT(D))=-\Cor^\bullet((-1)!\mathsf{ch}_1(c_1))=0\, .\]
    The right side of the equality,
   \[ \imath u\,  \widetilde{\rmL}_{-1}^\GW(\Cor^\bullet(1))=
    \imath u\, \widetilde{\rmL}_{-1}^\GW(1)=0\, , \]
    matches. 
 For \(k=0\), the left side for $D=1$ is
    \[\Cor^\bullet(\rmL_{0}^\PT(1))=-\Cor^\bullet(\widetilde{\mathsf{ch}}_2(c_1))=-\fra_1(c_1)=-\tau_0(c_1)-\frac{1}{24}\int_X c_1c_2\, .\]
   The right side,
   \[\widetilde{\rmL}_0^\GW(\Cor^\bullet(1))=\widetilde{\rmL}_0^\GW(1)=-\tau_0(c_1)-\frac{1}{24}\int_X c_1c_2\, ,\]
   matches.

   \vspace{8pt}
   \noindent{\bf{Step 2. }}
   We check the statement for \(D=\widetilde{\mathsf{ch}}_{k+2}(\gamma)\) with $k\geq 0$.
\vspace{8pt}

   We
must expand both sides of the
    equality of Proposition  \ref{PPP111}
    in terms of \(\tau\). The following formula will be used:
    \begin{multline}\label{eq:Cor-m1}
      (\imath u)^k \Cor^\circ(\widetilde{\mathsf{ch}}_{k+2}(\gamma))=
      \tau_k(\gamma)
      +\left(\sum_{i=1}^k\frac{1}{i}\right)\tau_{k-1}(\gamma\cdot c_1)+
      \left(\sum_{1\le i<j\le k}\frac{1}{ij}\right)\tau_{k-2}(\gamma\cdot c_1^2)
     \\ +\sum_{|\mu|=k-1}\frac{\mu_1!\mu_2!}{\text{Aut}(\mu) k!}\bigg(\tau_{\mu_1-1}\tau_{\mu_2-1}(\gamma\cdot c_1)+
     \big(\sum_{i=1}^{\mu_1-1}\frac{1}{i}\big)\tau_{\mu_1-2}(\gamma \cdot c_1^2)\tau_{\mu_2-1}(\pt)+\big(\sum_{i=1}^{\mu_2-1}\frac{1}{i}\big)\tau_{\mu_1-1}(\pt)\tau_{\mu_2-2}(\gamma\cdot c_1^2)\bigg)\\
     +\sum_{|\mu|=k-2}\frac{\mu_1!\mu_2!}{\text{Aut}(\mu)k!}\tau_{\mu_1-1}\tau_{\mu_2-1}(\gamma\cdot c_1^2) +\sum_{|\mu|=k-3}\frac{\mu_1!\mu_2!\mu_3!}{\text{Aut}(\mu) (k-1)!}\tau_{\mu_1-1}\tau_{\mu_2-1}\tau_{\mu_3-1}(\gamma\cdot c_1^2)\, .
      \end{multline}

      We split the analysis of the difference for 
      \begin{equation}
        \label{eq:diff-Lm1:1}
      \Cor^\bullet\circ \rmL_{-1}^\PT(D)- \imath u\, 
      \widetilde{\rmL}_{-1}^\GW\circ \Cor^\bullet(D)   
    \end{equation}
    in stages according to the \(\tau\) degree of terms.
    The second term of the difference is
    simpler since
  $$  \imath u\, 
      \widetilde{\rmL}_{-1}^\GW\circ \Cor^\bullet(D)    =
  \imath u\,   \rmR_{-1}(\Cor^\circ(\widetilde{\mathsf{ch}}_k(\gamma)))\, $$
  and the latter is a easy modification of \eqref{eq:Cor-m1}.
  The  first term is more involved since there are two parts:
  the action of \(\rmR_{-1}\) and the interaction with
  \((-1)!\mathsf{ch}_1(c_1)\).

  \vspace{8pt}
\noindent $\bullet$    We first  study the \(\tau\) linear terms of \((\imath u)^{k-1}\Cor^\bullet\circ \rmL_{-1}^\PT(D)\):
    \begin{multline*}
      \bigg( \tau_{k-1}(\gamma)+\Big(\sum_{i=1}^{k-1}\frac{1}{i}\Big)\tau_{k-2}(\gamma\cdot c_1) +\Big( \sum_{1\le i<j\le k-1}\frac{1}{ij}\Big)\tau_{k-3}(\gamma\cdot c_1^2)\bigg) 
      \\
      +\bigg( \frac{(\imath u)^{k-2}}{k!}\fra_{k-1}(\gamma\cdot c_1)+\frac{(\imath u)^{k-3}}{k!}\fra_{k-2}(\gamma\cdot c_1^2)\bigg)=
      \\      \bigg( \tau_{k-1}(\gamma)+\Big(\sum_{i=1}^{k-1}\frac{1}{i}\Big)\tau_{k-2}(\gamma\cdot c_1) +\Big( \sum_{1\le i<j\le k-1}\frac{1}{ij}\Big)\tau_{k-3}(\gamma\cdot c_1^2)\bigg)
      \\
      +\frac{1}{k}\bigg(\tau_{k-2}(\gamma\cdot c_1) +\Big(\sum_{i=1}^{k-2}\frac{1}{i}\Big)\tau_{k-3}(\gamma\cdot c_1^2)\bigg)+
      \frac{1}{k(k-1)}\tau_{k-3}(\gamma\cdot c_1^2)\, .
    \end{multline*}
    We have used here  bumping with \((-1)!\mathsf{ch}_1(c_1)\) from
\eqref{eq:2bump-minus-1-factorial}
  to obtain the expression in the second line
  and  an inversion{\footnote{See \eqref{eq:inv_tau2a}  for the
      full formula for the inversion.}}
  of \eqref{eq:tau2a} to justify the second equality.
  After collecting together the coefficients
  in front of the \(\tau\)'s in the last expression,
  we obtain \(R_{-1}(\Cor^\circ(\widetilde{\mathsf{ch}}_k(\gamma)))\),
  exactly as expected.

  \vspace{8pt}
\noindent $\bullet$  We study next the \(\tau\)-quadratic term of \eqref{eq:diff-Lm1:1}.
  Consider first the terms that have a co-product
    \((\gamma\cdot c_1)^L_i\otimes (\gamma\cdot c_1)^R_i\) as argument. Bumping with \((-1)!\mathsf{ch}_1(c_1)\) does not produce
    such terms -- only the terms of the second line of \eqref{eq:Cor-m1} contributes to the terms of \eqref{eq:diff-Lm1:1}.
    These terms cancel  exactly.

    \vspace{8pt}
\noindent $\bullet$
    The \(\tau\)-quadratic terms of difference \eqref{eq:diff-Lm1:1} with argument     \((\gamma\cdot c_1^2)^L_i\otimes
    (\gamma\cdot c_1^2)^R_i\) are slightly more involved. The  second term of the difference has terms:
    \begin{multline*}
      \sum_{|\mu|=k-2}\frac{\mu_1!\mu_2!}{\text{Aut}(\mu)(k-1)!}\bigg(\Big(\sum_{i=1}^{\mu_1-1}\frac1i\Big)\tau_{\mu_1-2}(\gamma\cdot c_1^2)\tau_{\mu_2-1}(\pt)
      +\Big(\sum_{i=1}^{\mu_2-1}\frac1i\Big)\tau_{\mu_1-1}(\pt)\tau_{\mu_2-2}(\gamma\cdot c_1^2)\bigg)\\
      +\sum_{|\mu|=k-3} \frac{\mu_1!\mu_2!}{\text{Aut}(\mu)(k-1)!}\tau_{\mu_1-1}\tau_{\mu_2-1}(\gamma\cdot c_1^2)\, ,
    \end{multline*}
    where the term on the second line is a result of bumping with \((-1)!\mathsf{ch}_1(c_1)\). After simplifying
   the last expression,
    we obtain the corresponding \(\tau\)-quadratic term of \(R_{-1}(\Cor^\circ(\widetilde{\mathsf{ch}}_{k+2}(\gamma)))\) as expected.

    \vspace{8pt}
 \noindent $\bullet$ The last step is to analyze the \(\tau\)-cubic terms of the difference \eqref{eq:diff-Lm1:1}. Since bumping
 with \(\mathsf{ch}_1(c_1)\) is trivial,
 the terms match exactly.
 \vspace{8pt}

    Similarly, we must analyze the difference
    \begin{equation} \label{vv44l}
      \Cor^\bullet\circ \rmL_{0}^\PT(D)-
      \widetilde{\rmL}_{0}^\GW\circ \Cor^\bullet(D)\, .   
    \end{equation}
    Since both \(\rmR_0\) on the stable pairs side and \(\rmR_0^0\) on the
    Gromov-Witten side scale the descendents by the  complex cohomological degree,
    the difference \eqref{vv44l} is equal{\footnote{Note both $\rmR_0^2$
        and $\rmR_0^3$ are 0.}} to
    \begin{equation}
        \label{eq:diff-L0:1}
-\Cor^\bullet\left((\widetilde{\mathsf{ch}}_2+\mathsf{ch}_1^2/2)(c_1)\cdot D\right)-\bigg(\rmR_0^1+\frac{u^{-2}}2\mathrm{B}^1-\tau_0(c_1)-\frac{1}{24}\int_Xc_1c_2\bigg)\circ\Cor^\bullet(D)\, .
     \end{equation}
     If \(D=\widetilde{\mathsf{ch}}_{k+2}(\gamma)\) then \(\mathrm{B}^1\circ\Cor^\bullet(D)=0\). We have already proved that the
     difference vanishes for \(D=1\). Since \[\Cor^\bullet(\mathsf{ch}_1\mathsf{ch}_1(c_1)\tch_{k+2}(\gamma))=0\, ,\] the difference \eqref{eq:diff-L0:1} is
     equal to
     \begin{equation}\label{eq:diff-short}
    -\Cor^\circ(\widetilde{\mathsf{ch}}_2(c_1)\widetilde{\mathsf{ch}}_{k+2}(\gamma))
       -\rmR^1_0(\Cor^\circ(\widetilde{\mathsf{ch}}_{k+2}(\gamma)))\, .\end{equation}
     Comparing formulas  \eqref{eq:decay-intro} and
     \eqref{eq:2bump-intro}, we conclude that the latter difference vanishes.

     Indeed, let us expand both terms of \eqref{eq:diff-short}. First,
     \begin{multline*}
       \Cor^\circ(\tch_2(c_1)\tch_{k+2}(\gamma))=-\frac{(\imath u)^{-1}}{k!}\fra_k(\gamma\cdot c_1)-\frac{(\imath u)^{-2}}{k!}\fra_{k-1}(\gamma\cdot c_1^2)
       -\frac{(\imath u)^{-2}}{(k-1)!}\sum_{|\mu|=k-2}\frac{\fra_{\mu_1}\fra_{\mu_2}}{\Aut(\mu)}(\gamma \cdot c_1^2)\\=
       -(\imath u)^{-k}\Bigg(\tau_{k-1}(\gamma\cdot c_1)+\left(\sum_{i=1}^{k-1}\frac{1}{i}\right)\tau_{k-2}(\gamma\cdot c_1^2)\Bigg)-\frac{(\imath u)^{-k}}{k}\tau_{k-2}(\gamma \cdot c_1^2)\\-\frac{(\imath u)^{-k+2}}{(k-1)!}
       \sum_{|\mu|=k-2}\frac{\mu_1!\mu_2!}{\Aut(\mu)}\tau_{\mu_1-1}\tau_{\mu_2-1}(\gamma\cdot c_1^2)\, .
     \end{multline*}
On the other hand,
     \begin{multline*}
       \Cor^\circ(\tch_{k+2}(\gamma))=\frac{1}{(k+1)!}\fra_{k+1}(\gamma)+
       \frac{(\imath u)^{-1}}{k!}\sum_{|\mu|=k-1}\frac{\fra_{\mu_1}\fra_{\mu_2}}{\Aut(\mu)}(\gamma\cdot c_1)+\dots\\
 =      (\imath u)^{-k}\Bigg(\tau_k(\gamma)+\left(\sum_{i=1}^k\frac{1}{i}\right)\tau_{k-1}(\gamma \cdot c_1)\Bigg)+\frac{(\imath u)^{-k+2}}{k!}\sum_{|\mu|=k-1}\frac{\mu_1!\mu_2!}{\Aut(\mu)}
       \tau_{\mu_1-1}\tau_{\mu_2-1}(\gamma \cdot c_1)+\dots\, ,
     \end{multline*}
     where we have used dots to stand for the
     terms that are of complex cohomological degree \(3\).
     Since
     \[\rmR_0^1(\tau_k(\gamma))=\tau_{k-1}(\gamma\cdot  c_1)\, ,\]
     all the omitted terms are annihilated by \(\rmR_0^1\). The remaining 
     terms of the difference \eqref{eq:diff-short} cancel. 
     

 \vspace{8pt}
   \noindent{\bf{Step 3. }}
   We check the statement for \(D=\tch_{k_1+2}(\gamma_1)\tch_{k_2+2}(\gamma_2)\) with $k_i\geq 0$.
\vspace{8pt}
     
     We start with the
     difference \eqref{eq:diff-Lm1:1}:
     \begin{multline}
       \label{eq:diff-Lm1:2}
      \Cor^\circ(\rmR_{-1}(\tch_{k_1+2}(\gamma_1)\tch_{k_2+2}(\gamma_2)))-\Cor^\circ((-1)!\tch_1(c_1)\tch_{k_1+2}(\gamma_1)\tch_{k_2+2}(\gamma_2))\\ -(\imath u)\rmR_{-1}(\Cor^\circ(\tch_{k_1+2}(\gamma_1)\tch_{k_2+2}(\gamma_2)))
       -(\imath u)\frac{u^{-2}}{2}\mathrm{B}^0(\Cor^\circ(\tch_{k_1+2}(\gamma_1)),\Cor^\circ(\tch_{k_2+2}(\gamma_2))))\, .
     \end{multline}
     Vanishing of the last expression follows from Proposition
     \ref{prop:double-bump1} and Proposition~\ref{prop:double-bump2}.

     The difference \eqref{vv44l} as above is equivalent
     to \eqref{eq:diff-L0:1}.
 Since we have already shown the vanishing for \(D=1\) and \(D=\tch_{k+2}(\gamma)\), we need only to check the vanishing
     of
\begin{multline}\label{ttt556}
  -\Cor^\circ(\tch_2(c_1)D)-\frac{1}{2}\Cor^\bullet(\mathsf{ch}_1\mathsf{ch}_1(c_1)D)-R_0^1(\Cor^\circ(D))\\
  -\frac{u^{-2}}{2}\mathrm{B}^1(\Cor^\circ(\tch_{k_1+2}(\gamma_1)),\Cor^\circ(\tch_{k_2+2}(\gamma_2)))\, .
\end{multline}
   The vanishing
   follows from Propositions \ref{prop:double-bump1} and \ref{prop:double-bump2}.
%

 \vspace{8pt}
   \noindent{\bf{Step 4. }}
   We check the statement for
\(D=\tch_{k_1+2}(\gamma_1)\tch_{k_2+2}(\gamma_2)\tch_{k_3+2}(\gamma_3)\) with $k_i\geq 0$.
 
\vspace{8pt}

The result follows immediately from the triple bumping
relation \eqref{eq:triple-bump} which holds in complete generality. No special cases require extra attention.
\qed

\subsection{Proof of Theorem \ref{thm:Vir-toric}} \label{jjjj9}
The vanishings 
\begin{equation}\label{jj823}
\langle \calL_{-1}^\PT(D)\rangle^{X,\PT}_\beta=0\ \ \ \ \text{and}
\ \ \ \ \langle \calL_0^\PT(D)\rangle^{X,\PT}_\beta=0
\end{equation}
are simple to prove for all \(D\in \mathbb{D}_{\PT}^{X}\).
For $$\calL_{-1}^\PT= \rmR_{-1} + \rmR_{-1}\ch_0(\mathsf{p})\, ,$$
the vanishing \eqref{jj823}
is immediate from the definition of $\rmR_{-1}$ and \eqref{m449}.
For $$\calL_{0} =  \rmR_0-{\tch}_2(c_1) -\frac{1}{2} \ch_1\ch_1(c_1)+ \rmR_{-1} \ch_1(\mathsf{p})\, $$
the vanishing \eqref{jj823} follows from the definition of $\rmR_0$,
the virtual dimension constraints, and the divisor equation:
\begin{equation*}
 \Big\langle {\ch_2(c_1) \, \mathsf{ch}}_{k_1}(\gamma_1)\cdots \mathsf{ch}_{k_m}(\gamma_m)\Big\rangle_{\beta}^{X,\PT}
 \, =\, \int_\beta c_1 \cdot \Big\langle { \mathsf{ch}}_{k_1}(\gamma_1)\cdots \mathsf{ch}_{k_m}(\gamma_m)\Big\rangle_{\beta}^{X,\PT}\, .
\end{equation*}

We now assume $k\geq 1$.
Using the intertwining property of 
Theorem \ref{thm:intertw},
  the stationary $\GW/\PT$ correspondence of Theorem \ref{thm:cor-main}, and
  the Virasoro constraints in Gromov-Witten theory,
  we can prove the stationary Virasoro constraints for stable
  pairs in the toric case.

  Let \(D\in \mathbb{D}_{\PT}^{X+}\), so $D$ is a monomial in the operators
   $$\big\{ \, \tch_i(\gamma)\,  | \, i\geq 0,\  \gamma\in H^{>0}(X,\mathbb{Q})\ \big\} .$$
   The first step is to check by hand
   that the Virasoro constraints
\begin{equation}\label{vvvv}
  \Big\langle\calL^\PT_k(D)\Big\rangle^{X,\PT}_\beta=0\,
  \end{equation}
   of Theorem \ref{thm:Vir-toric} are compatible all with insertions
   of the form
\begin{equation}\label{exxx}
  \tch_0(\gamma),  \tch_1(\gamma)\ {\text{for $\gamma \in H^{>0}(X)$}}
  \ \ {\text{and}}\ \ 
  \  \tch_2(\delta)\ {\text{for  $\delta \in H^2(X)$.}}
  \end{equation}
   If any of the operators \eqref{exxx} appear in $D$,
   the Virasoro constraints \eqref{vvvv} are true if
   the Virasoro constraints are true for the monomial obtained
   by dividing $D$ by the occurring operators \eqref{exxx}.
   We can therefore
  reduce to the case where $D$ is a monomial in
the operators
  $$\big\{ \, \tch_i(\gamma)\,  | \, (i\geq 3, \gamma\in H^{>0}(X,\mathbb{Q}))\ 
  \text{\em or}\  (i=2, \gamma\in H^{>2}(X,\mathbb{Q}))\, \big\} .$$
  In other words, $D\in \mathbb{D}^{X\bigstar}_\PT$.

  The next step is to apply Theorem \ref{thm:cor-main}:
\begin{equation}
\label{lwlwlw}
(-q)^{d_\beta/2}\, \langle \calL_k^\PT(D)\rangle^{X,\PT}_\beta
=(-\imath u)^{d_\beta}\, \langle\Cor^\bullet(\calL_k^\PT(D))\rangle^{X,\GW}_\beta
\end{equation}
for all $k\geq 1$.
By the construction of the correspondence \cite{PPDC},
 the descendents of the point class do not interact with other descendents:
\begin{equation}\label{eq:Cpoint}\Cor^\bullet(\tch_{k+2}(\pt)D)=(\imath u)^{-k}\tau_k(\pt)\Cor^\bullet(D)\, ,
\end{equation}
for every \(D\in \mathbb{D}_{\PT}^{X\bigstar}\).

By combining \eqref{lwlwlw}, \eqref{eq:Cpoint},  and the intertwining statement of Theorem~\ref{thm:intertw},  we see
\begin{eqnarray*}
\langle\Cor^\bullet(\calL_k^\PT(D))\rangle^\GW_\beta &=&
\langle\Cor^\bullet(\rmL_k^\PT(D))\rangle^\GW_\beta + (k+1)!
\, \langle\Cor^\bullet(\rmL_{-1}^\PT(\ch_{k+1}(\pt) D))\rangle^\GW_\beta
\\
& = &
(\imath u)^{-k}\langle \widetilde{\rmL}_k^\GW(\Cor^\bullet(D))\rangle_\beta^\GW+
(\imath u)^{2-k}(k+1)!\, \langle \widetilde{\rmL}_{-1}^\GW(\tau_{k-1}(\pt)\Cor^\bullet(D))\rangle_\beta^\GW\\
&=&  (\imath u)^{-k}\langle\calL_k^\GW(\Cor^\bullet(D))\rangle^{\GW}_\beta\\
&=& 0\, ,\end{eqnarray*}
where the last equality is by Proposition \ref{ggkk17} which
may be applied since
$$\Cor^\bullet(D) \in  \mathbb{D}_{\GW}^{X+}\, $$
by Proposition \ref{666}.
We conclude $$\langle \calL_k^\PT(D)\rangle^{X,\PT}_\beta=0$$
as required.
\qed

\vspace{10pt}

We could have also used the intertwining property of Proposition \ref{PPP111}
to prove the stable pairs vanishings \eqref{jj823} for
\(D\in \mathbb{D}_{\PT}^{X+}\),
 but some additional care must be taken since
the insertions $\ch_0(\mathsf{p})$ and $\ch_1(\mathsf{p})$ which
occur in  the terms
$$(k+1)!
\, \langle\Cor^\bullet(\rmL_{-1}^\PT(\ch_{k+1}(\pt) D))\rangle^\GW_\beta$$
for $k=-1$ and $0$ are not covered by Proposition \ref{PPP111}. We leave the
details to the reader.

\section{Intertwining I: Basic case}
\label{sec:main-result}

\subsection{Overview}
After an explicit study of
various terms of the stationary Gromov-Witten Virasoro constraints  in Section
\ref{ltt},
we prove  
Theorem \ref{thm:intertw} in the basic case
\(D=1\) in Section \ref{sec:const-term-comp}.


\subsection{Leading term} \label{ltt}
We analyze here the stationary Virasoro constraints on the Gromov-Witten
side defined in Section \ref{sec:changes-variables-gw}.

The leading term \(\CT^1_k\) of \(\CT'_k\) is of the form
\[\frac{1}{2}\CT^1_k=\frac{k!}{u^2}\, \tau_k(c_1)+\frac{1}{2}\sum_{a+b=k-2}
(-1)^{d^L-1}(a+d^L-1)!(b+d^R-1)!\tau_a\tau_b(c_1)\, ,\]
where \(a,b\ge 0\) in the sum. By the following result,
the term \(\CT'_k\)   simplifies if we use the modified descendents \(\fra_i\).
\begin{proposition}
  \label{prop:CT} For all \(k\geq -1\),
  \[\CT_k'=-(\imath u)^{k-2}\sum_{a+b=k+2}(-1)^{d^Ld^R}(a+d^L-3)!(b+d^R-3)!\frac{\fra_{a-1}\fra_{b-1}(c_1)}{(a-1)!(b-1)!}\, ,\]
  where the sum over all \(a,b\ge 0\) and we use convention \(\fra_0=0,\, 
  \fra_{-1}/(-1)!=\tau_{-2}\).
\end{proposition}
\begin{proof}

  
  Using formula \eqref{eq:tau2a}, 
we expand \(\CT'_k\) in terms of \(\fra_i\) to show that the quadratic and cubic in \(c_1\) terms cancel.
  In the computation, we compare  the expressions
  \[[-a]^k_2=(-1)^a a!(k-a)!\left(\sum_{i=1}^{k-a}\frac{1}i-\sum_{i=1}^a\frac{1}i\right),\quad a\ge 0\, ,\  k\ge a\, ,\]
  \[[-a]^k_3=(-1)^aa!(k-a)!\left(\sum_{1\le i<j\le k-a}\frac{1}{ij}+\sum_{1\le i<j\le a}\frac{1}{ij}-\left(\sum_{i=1}^{k-a}\frac{1}{i}\right)
      \left(\sum_{i=1}^a\frac{1}{i}\right)\right),\quad a\ge 0,k\ge a\]
  with the coefficients in \eqref{eq:tau2a}.

 The transformation \eqref{eq:tau2a} simplifies 
if we use the following operators and short-hand notations for the sums:
  \[\tilde{\fra}_k=\frac{(\imath u)^{k-1}}{k!}\fra_k\, , \quad \chi_l^k=\sum_{j=1}^k\frac{1}{j^l}\, , \quad \chi_{1,1}^k=\sum_{1\le i<j \le k}\frac{1}{ij}\, .\]
  In the formulas below, all operators \(\tilde{\fra}_0\) are set to be zero.
  We apply transformation to \(\CT_k^1\) to obtain:
  \begin{multline}
    \sum_{m=-1}^{k+1}(-1)^{d_L-1}(m+d_L-2)!(k-m-d_L+2)!\ \times 
\\\Bigg(\tilde{\fra}_m\tilde{\fra}_{-m+k}(c_1)- \left(\chi_{1}^{m} -\chi_{1}^{k-m-1}\right)\tilde{\fra}_{m}\tilde{\fra}_{-m+k-1}(c_1^2)  \label{eq:T1k-a3}\\ 
+\Big(\chi_1^{m}\chi_1^{k-m-2}+\chi_2^{m}+\chi_{1,1}^{m}+\chi_2^{-m+k-2}+\chi_{1,1}^{-m+k-2}\Big)\tilde{\fra}_{m}\tilde{\fra}_{-m+k-2}(c_1^3)\Bigg)\, .
  \end{multline}
  To write the transformation of \(\CT^2_k\),
 we split the sum for \(\CT^2_k\) into
  two subsums, the first with \(d_L=2\) and the second with \(d_L=3\):
    \begin{multline*}
    \sum_{m=-1}^k(-1)(m)!(k-m)!(\chi_1^{k-m}-\chi_1^{m})
    \Big(\tilde{\fra}_m(c_1^2)\tilde{\fra}_{-m+k-1}(\pt)-\chi_1^{m-1}\tilde{\fra}_{m-1}\tilde{\fra}_{-m-k-1}(c_1^3)\Big)+\\
    (m+1)!(k-m-1)!(\chi_1^{k-m-1}-\chi_1^{m-1})
    \Big(\tilde{\fra}_m(\pt)\tilde{\fra}_{-m+k-1}(c_1^2)-\chi_1^{k-m-2}\tilde{\fra}_m\tilde{\fra}_{-m-k-2}(c_1^2)\Big).
  \end{multline*}
  Finally, the transformation of \(\CT^3_k\) to \(\fra\) variables is
  \begin{gather*}
    \sum_{m=-1}^k(m+1)!(k-m-1)!\Big(\chi_{1,1}^{m-1}+\chi_{1,1}^{k-m-1}-\chi_1^{m-1}\chi_1^{k-m-1}\Big)
    \tilde{\fra}_m\tilde{\fra}_{-m+k-2}(c_1^3).
  \end{gather*}
  After summing the terms \(\CT_k^j\) for \(j=1,2,3\),
 we find that only the first term in
  \eqref{eq:T1k-a3} does not cancel.
\end{proof}

\subsection{Intertwining for $D=1$}
\label{sec:const-term-comp}
 For the most of computations in Section
\ref{sec:main-result}, 
 we will require the simplest
case of the stationary $\GW/\PT$ transformation
 $\Cor^\bullet$ of Section 
\ref{sec:stat-gwpt-corr},
\begin{multline}
  \label{eq:decay}
\Cor^\circ(\tch_{k+2}(\gamma))=\frac{1}{(k+1)!}\mathfrak{a}_{k+1}(\gamma)
+\frac{(\imath u)^{-1}}{k!}\sum_{|\mu|=k-1}\frac{\fra_{\mu_1}\fra_{\mu_2}(\gamma\cdot c_1)}{\Aut(\mu)}
+\\\frac{(\imath u)^{-2}}{k!}\sum_{|\mu|=k-2}\frac{\fra_{\mu_1}\fra_{\mu_2}
(\gamma\cdot c_1^2)}{\Aut(\mu)}+
  \frac{(\imath u)^{-2}}{(k-1)!}\sum_{|\mu|=k-3}\frac{\fra_{\mu_1}\fra_{\mu_2}\fra_{\mu_3}(\gamma\cdot c_1^2)}{\Aut(\mu)}\, .
\end{multline}
 Our first result is
the simplest case of Theorem \ref{thm:intertw}.

\begin{proposition}\label{prop:CTT}
  For all \(k\geq 1\),  we have
\[\Cor^\bullet(\rmL^\PT_k(1))= (\imath u)^{-k}\, \widetilde{\rmL}_k^\GW(1)\, .\]
\end{proposition}
\begin{proof} 
  Since the operators \(\rmR_k\) annihilate $1$,  we must prove \begin{equation}\label{eq:ctrel}
    \Cor^\bullet(\CT_k)= {(\imath u)^{-k}}\left(\frac{(\imath u)^2}{2}\CT'_k\right)\, .
  \end{equation}
From Section \ref{sec:pt-vir-constraints}, we have the following formula on the stable pairs side:
\[\CT_k=-\frac{1}{2}\sum_{a+b=k+2}(-1)^{d^Ld^R}(a+d^L-3)!(b+d^R-3)!\, \tch_a\tch_b(c_1)\, .\]
On the Gromov-Witten side, we have
  \[\CT_k'=-(\imath u)^{k-2}\sum_{a+b=k+2}(-1)^{d^Ld^R}(a+d^L-3)!(b+d^R-3)!\frac{\fra_{a-1}\fra_{b-1}(c_1)}{(a-1)!(b-1)!}\, \]  
by Proposition \ref{prop:CT}.
Using \eqref{eq:decay}, the quadratic term in the $\fra$-insertions of
$\Cor^\bullet(\CT_k)$ exactly matches the full right side of \eqref{eq:ctrel}.
We will prove the other terms of $\Cor^\bullet(\CT_k)$ all vanish.

  The stable pairs term $\CT_k$ is the sum of three subsums:
  \begin{multline}\label{eq:CTk1}
    \frac{1}{2}\sum_{a+b=k+2}\bigg( (a-2)!b!\, \tch_{a}(c_1)\tch_{b}(\pt) +a!(b-2)!\, \tch_{a}(\pt)\tch_{b}(c_1)
    \\ 
    -(a-1)!(b-1)!
    \sum_{s+1\le \bullet,\star\le 2s}\alpha_{\bullet \star}\, \tch_{a}(\gamma_\bullet)
    \tch_{b}(\gamma_\star)\bigg)\, ,
  \end{multline}
  where last term uses{\footnote{We use the subscripts $\bullet$ and
      $\star$ in order to avoid $i,j,a,b$ which are already taken.}}
  $$ c_1\cdot\gamma_{2s+1-\bullet}=\sum_\star
  \alpha_{\bullet \star} \gamma_\star \, .$$
After applying $\Cor^\bullet$ to \eqref{eq:CTk1}  
we obtain quadratic, cubic, and quartic  monomials in \(\fra\).
We will show the cubic and quartic terms vanish.

  We start with the
analysis of the  quartic term of \(\Cor^\bullet(\CT_k)\). The first term \eqref{eq:CTk1} yields the quartic part: 
  \begin{gather*}
    \frac{1}{2}\int_X c_1^3\cdot \sum_{a+b=k+2}\bigg((a-2)!b!\cdot\frac{\fra_{b-1}(\pt)}{(b-1)!}\cdot
    \frac{(\imath u)^{-2}}{(a-3)!}\sum_{|\mu|=a-5}\frac{\fra_{\mu_1}(\pt)\fra_{\mu_2}(\pt)\fra_{\mu_3}(\pt)}{\Aut(\mu)}
    \\+(b-2)!a!\cdot\frac{\fra_{a-1}(\pt)}{(a-1)!}\cdot
    \frac{(\imath u)^{-2}}{(b-3)!}\sum_{|\mu|=b-5}\frac{\fra_{\mu_1}(\pt)\fra_{\mu_2}(\pt)\fra_{\mu_3}(\pt)}{\Aut(\mu)}\bigg)\, .
  \end{gather*}
  The last term of \eqref{eq:CTk1} yields the following quartic part (with
the sum over the same range of $a$ and $b$ as above):
  \[-\frac{1}{2}\int_X c_1^3\, \cdot\, (a-1)!(b-1)!\cdot\frac{(\imath u)^{-1}}{(a-2)!}\sum_{|\mu'|=a-3}\frac{\fra_{\mu'_1}(\pt)\fra_{\mu'_2}(\pt)}{\Aut(\mu')}\cdot\frac{(\imath u)^{-1}}{(b-2)!}
    \sum_{|\mu''|=b-3}\frac{\fra_{\mu''_1}(\pt)\fra_{\mu''_2}(\pt)}{\Aut(\mu'')}\, ,\]
  where, in both formulas, we have used convention \(|\mu|=\sum_i\mu_i\).
 
These two quartic parts cancel each other. 
Indeed, let us analyze the factor in
  front of $$\frac{1}{2(\imath u)^{2}}\int_X c_1^3\cdot \fra_{\lambda_1}(\pt)\fra_{\lambda_2}(\pt)\fra_{\lambda_3}(\pt)\fra_{\lambda_4}(\pt)$$
in  both expressions. 
For simplicity,
  let us assume \(|\Aut(\lambda)|=1\). 
Then, the factor in the first quartic part is a sum with four terms:
\begin{equation}\label{44667} 
\sum_{i=1}^4\left(\lambda_i+1\right)\left(\sum_{j\ne i}(\lambda_j+1)\right)\, .
\end{equation}
 The factor in the second formula is a sum with three terms:
\begin{equation}\label{44668}  
-\sum (\lambda_{i_1}+\lambda_{i_2}+2)(\lambda_{j_1}+\lambda_{j_2}+2)\, ,
\end{equation}
  where the sum is over all splittings 
$$\{1,2,3,4\}=\{i_1,i_2\}\cup\{j_1,j_2\}\,.$$
 The factors \eqref{44667} and \eqref{44668}
 are sums of twelve monomials of \(\lambda_i+1\) and
  are opposites of each other. The case when \(|\Aut(\lambda)|>1\) is analogous.

  Finally, we analyze the cubic terms. Let us first analyze the cubic terms of the form \(\fra_i(\pt)\fra_j(\pt)\fra_l(\pt)\). Since 
$$\ch_{k+2}(c_1)\ch_0(\pt)=
  (-1)\ch_{k+2}(c_1)\, ,$$
 the cubic part of the first term of \eqref{eq:CTk1} with \(b=0\) is:
  \begin{equation}\label{eq:cub1}
  -k\int_X\frac{c_1^3}{2(\imath u)^2}\sum_{|\mu|=k-1}\frac{\fra_{\mu_1}(\pt)\fra_{\mu_2}(\pt)\fra_{\mu_3}(\pt)}{\Aut(\mu)}\, .\end{equation}
A similar cubic part
 is produced by the second term of \eqref{eq:CTk1} with \(a=0\).

The other cubic parts of the first term of \eqref{eq:CTk1} are:
  \begin{equation}
    \label{eq:cubic_term}
    \int_Xc_1^3\sum_{a+b=k+2}\frac{b}{2 (\imath u)^2}\fra_{b-1}(\pt) \sum_{|\mu|=a-4}\frac{\fra_{\mu_1}(\pt)\fra_{\mu_2}(\pt)}{\Aut(\mu)}
    +\frac{b}{2\imath u}\fra_{b-1}(\pt)\sum_{|\mu|=a-3}\frac{\fra_{\mu_1}\fra_{\mu_2}(c_1^2)}{\Aut(\mu)}\, .
  \end{equation}
Similar term is yielded by the second term of \eqref{eq:CTk1}.
  
  If \(\Aut(\mu)=1\), then the factor  in front of monomial
  $$\frac{1}{2(\imath u)^{2}}\fra_{\lambda_1}(\pt)\fra_{\lambda_2}(\pt)\fra_{\lambda_3}(\pt)$$ of \eqref{eq:cubic_term}
 is the sum of three terms
  $$ (\lambda_1+1)+(\lambda_2+1)+(\lambda_3+1)$$
 and, hence, cancels with corresponding monomial from \eqref{eq:cub1}.

  The cubic part of the last term of \eqref{eq:CTk1} is 
\begin{multline*} 
-\frac{(a-1)}{2 \imath u} \sum_{\bullet,\star}\alpha_{\bullet \star}\, \fra_{b-1}(\gamma_\star)\sum_{|\mu|=a-3}\frac{\fra_{\mu_1}\fra_{\mu_2}(c_1\cdot\gamma_\bullet)}{\Aut(\mu)}\\
-
    \frac{(b-1)}{2 \imath u}\sum_{\bullet,\star}\alpha_{\bullet \star}\, \fra_{a-1}(\gamma_\star)\sum_{|\mu|=b-3}\frac{\fra_{\mu_1}\fra_{\mu_2}(c_1\cdot\gamma_\bullet)}{\Aut(\mu)}\, ,
\end{multline*}
  over all \(a,b\geq 0\) satisfying \(a+b=k+2\). 
The sum cancels with the last term of~\eqref{eq:cubic_term}. \end{proof}


\section{Intertwining II: Non-interacting insertions}
\label{nonintS}
\subsection{Overview}
The main
result of Section \ref{nonintS}
is a proof of Theorem \ref{thm:intertw} for 
\begin{equation}\label{gmm3}
D\in \mathbb{D}^1_\PT \cap \mathbb{D}^{X\bigstar}_{\PT\circ}
\, ,
\end{equation}
where
 \(D\) is a product of \(\tch_{k_i}(\gamma_i)\) satisfying
 $$\gamma_i\cdot \gamma_j=0\ \ \text{for}\ \   i\ne j\, .$$

 We 
treat the singleton $D= \tch_{k}({\mathsf{p}})$  in
 Proposition \ref{prop:shiftpt}. An intricate
 computation is required for Proposition \ref{prop:no-bump} 
which settles the 
cases $D= \tch_{k}(\gamma)$ where
$$\gamma\in H^i(X)\ \ \
{\text{for $i=2$ and 4.}}$$
Finally, in Section \ref{mono9}, the general case  \eqref{gmm3}
is formally deduced from the singletons.

  \subsection{Intertwining shift operators}
\label{sec:intertw-shift-oper}
We first relate 
the operators $\rm R_k$ appearing in the Virasoro constraints on the
stable pairs  and Gromov-Witten sides. Recall,
\begin{equation}\label{55pp55}
  \tch_k(\alpha)=\ch_k(\alpha)+\frac{1}{24}\ch_{k-2}(\alpha\cdot c_2 )\,,
  \end{equation}
so $\tch_k(\pt)=\ch_k(\pt)$.


\begin{proposition}\label{prop:shiftpt}
  For all $k\geq 1$ and all $i\geq 2$,  we have
  \[\Cor^\bullet(\rmR_k( \ch_{i}(\pt)))=(\imath u )^{-k} \, \rmR_k(\Cor^\bullet( \ch_{i}(\pt)))\, .\]
\end{proposition}
\begin{proof}

  The left side of the equation is
  \[\Cor^\bullet(\rmR_k(\ch_i(\pt)))=\Cor^\bullet\left(\frac{(i+k)!}{(i-1)!}\ch_{i+k}(\pt)\right)=
    \frac{(i+k)!}{(i-1)!}\frac{\fra_{i+k-1}(\pt)}{(i+k-1)!}=\frac{(i+k)}{(i-1)!}
\, \fra_{i+k-1}(\pt) \, ,\]
where we have used the definition of $\rmR_k$ for stable pairs and
equation \eqref{eq:decay-intro} for the correspondence.

The right side of the equation is
 \begin{eqnarray*}\rmR_k(\Cor^\bullet(\ch_i(\pt)))&=&\rmR_k\left(\frac{\fra_{i-1}(\pt)}{(i-1)!}\right)\\
&=&\rmR_k\left(\frac{\tau_{i-2}(\pt)}{(\imath u)^{i-2}}\right)\\
&=&\frac{(i+k)!}{(i-1)!}\frac{\tau_{i+k-2}(\pt)}{(\imath u)^{i-2}}\\ &=&
  \frac{(i+k)}{(i-1)!}(\imath u)^k\, \fra_{i+k-1}(\pt)\, ,
\end{eqnarray*}
where we have  used \eqref{eq:decay-intro} for the
correspondence, 
equation \eqref{eq:tau2a}, and
the definition of $\rmR_k$ for Gromov-Witten theory.
The two sides match.
\end{proof}

\begin{proposition} \label{prop:no-bump}
  For all $k\geq 1$, $\tch_i(\gamma)\in \mathbb{D}^{\bigstar X}_\PT$,
  \(\gamma\in H^{\ge 2}(X)\) we have
  \[\Cor^\bullet(\rmL^\PT_k(\tch_i(\gamma)))=(\imath u)^{-k}\, 
\widetilde{\rmL}^{\GW}_k(\Cor^\bullet(\tch_i(\gamma)))\, .\]
  \end{proposition}

  \begin{proof}
 We start with the easiest case and proceed
to the hardest case.

\vspace{15pt}
    \noindent {\underline{\bf Case \(\gamma\in H^6(X).\)}}
    The case \(\gamma=\pt\) follows immediately from the previous results:
    \begin{eqnarray*}\Cor^\bullet(\rmL^\PT_k(\ch_i(\pt)))&=&\Cor^\bullet(
                                \CTp_k\, \ch_i(\pt) + \rmrp_k(\ch_i(\pt)))\\
                                                         &=& \Cor^\bullet(\CTp_k)\Cor^\bullet(\ch_i(\pt))+ (\imath u)^{-k}\rmrg_k(\Cor^\bullet(\ch_i(\pt))\\
      &=&
          (\imath u)^{-k}\tilde{\rmL}^{\GW}_k(\Cor^\bullet(\ch_i(\pt)))\, .
          \end{eqnarray*}
The second equality follows from Proposition~\ref{prop:shiftpt} and \eqref{eq:Cpoint}. The third equality uses \eqref{eq:ctrel}.

\vspace{15pt}    
\noindent\underline{    {\bf Case \(\gamma\in H^4(X).\)}}
We compute the
    difference
    \begin{equation}\label{eq:diffL}
      (\imath u )^k\Cor^\bullet(\rmR_k(\tch_i(\gamma)))-\rmR_k(\Cor^\bullet(\tch_i(\gamma)))\, .
    \end{equation}
    Since $\gamma\cdot c_2=0$, we have $\tch_k(\gamma)=\ch_k(\gamma)$
    by \eqref{55pp55}.
    
 We start by expanding the first term of the difference:
    \begin{eqnarray*}
      \Cor^\circ(\rmR_k(\ch_i(\gamma)))&=&\Cor^\circ\left(\frac{(i+k-1)!}{(i-2)!}\ch_{k+i}(\gamma)\right)\\
      & =&\frac{(i+k-1)!}{(i-2)!}\left(\frac{\fra_{i+k-1}(\gamma)}{(i+k-1)!}                                                                                                              
      +
        \frac{(\imath u)^{-1}}{(i+k-2)!}\sum_{\mu_1+\mu_2=i+k-3}\frac{\fra_{\mu_1}\fra_{\mu_2}}{2}(\gamma \cdot c_1)\right)\, .
    \end{eqnarray*}
To proceed, we invert  the correspondence \eqref{eq:tau2a}:
    \begin{equation}
      \label{eq:inv_tau2a}
  \frac{(\imath u)^{k}\fra_{k+1}}{(k+1)!}(\gamma)=\tau_k(\gamma)+\left(\sum_{i=1}^k\frac{1}{i}\right)\tau_{k-1}(\gamma\cdot c_1)+\left(\sum_{1\le i<j\le k}\frac{1}{ij}\right)\tau_{k-2}(\gamma\cdot c_1^2)\, .
\end{equation}
We then obtain
\begin{multline}\label{eq:RHSchL}
(\imath u)^k  \Cor^\circ(\rmR_k(\ch_i(\gamma)))=\frac{(i+k-1)!}{(i-2)!}\left(\frac{\tau_{i+k-2}(\gamma)}{(\imath u)^{i-2}}+\left(\sum_{j=1}^{i+k-2}\frac{1}{j}\right)\frac{\tau_{i+k-3}(\gamma \cdot c_1)}{(\imath u)^{i-2}}\right.\\
      \left.+
  \frac{(\imath u)^{-i+4}}{(i+k-2)!}\sum_{\mu_1+\mu_2=i+k-3}\mu_1!\mu_2!\frac{\tau_{\mu_1-1}\tau_{\mu_2-1}}{2}(\gamma\cdot c_1)\right).
\end{multline}

We write the second term of the difference as
\begin{equation}\label{eq:LHSchL}
  \rmR_k(\Cor^\circ(\ch_i(\gamma)))=\rmR_k\left(\frac{\fra_{i-1}(\gamma)}{(i-1)!}+\frac{(\imath u)^{-1}}{(i-2)!}\sum_{\mu_1+\mu_2=i-3}\frac{\fra_{\mu_1}\fra_{\mu_2}}{2}(\gamma \cdot c_1)\right)\, .
\end{equation}
After applying the inversion \eqref{eq:inv_tau2a}, we have
$$
  \rmR_k\left(\frac{\tau_{i-2}(\gamma)}{(\imath u)^{i-2}}+\left(\sum_{j=1}^{i-2}\frac{1}{j}\right)\frac{\tau_{i-3}(\gamma \cdot c_1)}{(\imath u)^{i-2}}+\frac{(\imath u)^{4-i}}{(i-2)!}\sum_{\mu_1+\mu_2=i-3}\mu_1!\mu_2!
    \frac{\tau_{\mu_1-1}\tau_{\mu_2-1}}{2}(\gamma \cdot c_1)\right)\, .$$
We expand the above expression fully to obtain 
\begin{small}
\begin{multline}\label{eq:LHSexpanded}
\hspace{150pt}  \frac{(i+k-1)!\tau_{i+k-2}(\gamma)}{(\imath u)^{i-2}(i-2)!}
  \\+\frac{(i+k-1)!}{(\imath u)^{i-2}(i-2)!}\left(\sum_{j=i-1}^{k+i-1}\frac1j\right)\tau_{i+k-3}(\gamma \cdot c_1)
  +\frac{(i+k-1)!}{(\imath u)^{i-2}(i-2)!}\left(\sum_{j=1}^{i-2}\frac{1}{j}\right)\tau_{i-k+3}(\gamma \cdot c_1)\\
  +
  \frac{(\imath u)^{-i+4}}{(i-2)!}\sum_{\mu_1+\mu_2=i-3}\left((\mu_1+k+1)!\mu_2!\frac{\tau_{\mu_1+k-1}\tau_{\mu_2-1}}{2}(\gamma \cdot c_1)
    +\mu_1!(\mu_2+k+1)!\frac{\tau_{\mu_1-1}\tau_{\mu_2+k-1}}{2}(\gamma \cdot
    c_1)\right),
\end{multline}
\end{small}
where we have used formula
$$[i]^k_1=\frac{(i+k)!}{(i-1)!}\sum_{j=i}^{i+k}\frac1j$$ in 
the expansion of \(\rmR_k(\tau_{i-2}(\gamma))\).

To complete our computation of the difference \eqref{eq:diffL}, we observe
several cancellations.
The first term of \eqref{eq:RHSchL}
cancels with first term of \eqref{eq:LHSexpanded}.
The second term of
  \eqref{eq:RHSchL}  almost cancels with the sum of the second and third terms
  of \eqref{eq:LHSexpanded}, the only terms that does not cancel is
  \begin{equation}\label{eq:Non-cancel}
    -\frac{(i+k-2)!}{(\imath u)^{i-2}(i-2)!}\tau_{i+k-3}(\gamma\cdot c_1)
\end{equation}  
  Finally, we rewrite the last term of \eqref{eq:RHSchL}
  as
  \[\frac{(\imath u)^{-i+4}}{(i-2)!}\sum_{\mu_1+\mu_2=i+k-3}(\mu_1+1)!\mu_2!\frac{\tau_{\mu_1-1}\tau_{\mu_2-1}}{2}(\gamma \cdot c_1)+\mu_1!(\mu_2+1)!\frac{\tau_{\mu_1-1}\tau_{\mu_2-1}}{2}(\gamma \cdot c_1)\, .\]
  Then, we see that the last term of \eqref{eq:LHSchL} cancels with the
  last term of \eqref{eq:RHSchL}  if \(\mu_1\ge k+1\) and \(\mu_2\ge k+1\). Thus the difference
  \eqref{eq:diffL} equals
\begin{multline}\label{eq:DCDC}
  \frac{(\imath u)^{-i+4}}{(i-2)!}\left(\sum_{\mu_1+\mu_2=i+k-3,\, \mu_1\le k}(\mu_1+1)!\mu_2!\frac{\tau_{\mu_1-1}\tau_{\mu_2-1}}{2}(\gamma \cdot c_1)\right.\\\left.+
    \sum_{\mu_1+\mu_2=i+k-3,\, \mu_2\le k}\mu_1!(\mu_2+1)!\frac{\tau_{\mu_1-1}\tau_{\mu_2-1}}{2}(\gamma \cdot c_1)\right).
\end{multline}

We now include the $\rmT_k$ and $\rmT_k'$ terms in the difference. We have
\begin{multline}\label{eq:diffExp}
  (\imath u)^k\Cor^\bullet(\rmL^{\PT}_k(\ch_i(\gamma)))-{\widetilde{\rmL}}^\GW_k(\Cor^\bullet(\ch_i(\gamma)))=\\
  (\imath u )^k\Cor^\bullet(\rmR_k(\ch_i(\gamma)))-\rmR_k(\Cor^\bullet(\ch_i(\gamma))) +(\imath u )^k\Cor^\bullet(\rmT_k(\ch_i(\gamma)))-\frac{(\imath u)^2}{2}\rmT'_k(\Cor^\bullet(\ch_i(\gamma)))
  \end{multline}
Using \eqref{eq:ctrel}, the $\rmT_k$ and $\rmT_k'$ terms in \eqref{eq:diffExp} simplify to 
\begin{multline}\label{fr559}
 \frac{(\imath u)^{k}}{2}\sum_{a+b=k+2}(a-2)!b!
\Cor^\circ\left(\frac{\tch_{a}(c_1)\ch_i(\gamma)}{(\imath u)^{b-2}}\right)\tau_{b-2}(\pt)\\
 \frac{(\imath u)^{k}}{2}\sum_{a+b=k+2}
a!(b-2)!\tau_{a-2}(\pt)\Cor^\circ\left(
    \frac{\tch_{b}(c_1)\ch_i(\gamma)}{(\imath u)^{a-2}}\right)\, .
  \end{multline}

  To complete our proof, we require
  the bumping formula \eqref{eq:2bump-intro}:
  \begin{equation}
    \label{eq:2bump}
    \Cor^\circ(\tch_{k_1+2}(c_1)\tch_{k_2+2}(\gamma))=-\frac{1}{k_1!k_2!}
    (\imath u)^{-1}\fra_{k_1+k_2}(c_1\gamma)\, .
\end{equation}
Since \(\gamma\in H^4(X)\), all the other terms of \eqref{eq:2bump-intro} vanish. 
We apply the bumping formula \eqref{fr559}.
In particular, the first term of \eqref{fr559},
  \[(a-2)!b!
  \Cor^\circ\left(\frac{\tch_{a}(c_1)\ch_i(\gamma)}{(\imath u)^{b-2}}\right)\tau_{b-2}(\pt)=- (\imath u)^{-a-b-i+6}\frac{(a+i-2)!b!}{(i-2)!}\tau_{a+i-3}(\gamma\cdot c_1)\tau_{b-2}(\pt)\, \]
cancels with the first term of \eqref{eq:DCDC}. Similarly, the second term of \eqref{fr559} cancels with the second term of \eqref{eq:DCDC}.

Let us observe that the term of last expression with \(a=1\)  by the exceptional
bumping \eqref{eq:2bump-minus-1-factorial} turns into the terms of \eqref{eq:DCDC}
with \(\mu_1=k\) or \(\mu_2=k\). Similarly, the
term with \(b=0\) cancels out with the term \eqref{eq:Non-cancel}.

Also the assumption
\(\tch_i(\gamma)\in \mathbb{D}^{X\bigstar}_\PT\) implies that \(i\ge 2\) thus
no negative factorials appear in the above computations.

\vspace{15pt}

\noindent\underline{{\bf Case \(\gamma\in H^2(X).\)}} 

If \(\gamma\in H^2(X)\),
the $\rmT_k$ and $\rmT_k'$ terms of the formula \eqref{eq:diffExp} acquires  extra summands:
    \begin{multline}\label{eq:diffExp-long}
  (\imath u)^k\Cor^\bullet(\rmL^\PT_k(\tch_i(\gamma)))-\widetilde{\rmL}_k^\GW(\Cor^\bullet(\tch_i(\gamma)))=\\
 \hspace{-150pt}  (\imath u )^k\Cor^\bullet(\rmR_k(\tch_i(\gamma)))-\rmR_k(\Cor^\bullet(\tch_i(\gamma)))\\+
  \frac{(\imath u)^{k}}{2}\bigg[\sum_{a+b=k+2}(a-2)!b!
  \Cor^\circ\left(\frac{\tch_{a}(c_1)\ch_i(\gamma)}{(\imath u)^{b-2}}\right)\tau_{b-2}(\pt)+a!(b-2)!\tau_{a-2}(\pt)\Cor^\circ\left(
    \frac{\tch_{b}(c_1)\ch_i(\gamma)}{(\imath u)^{a-2}}\right)\\-
  \sum_{a+b=k+2}(a-1)!(b-1)!\sum_{0<\bullet,\star <2s+1}
  \alpha_{\bullet \star}\left(\Cor^\circ(\tch_{a}(\gamma_\bullet)\cdot
     \ch_i(\gamma))\Cor^\circ(\tch_{b}(\gamma_\star))\right.\\\left.+\Cor^\circ(\tch_{a}(\gamma_\bullet))\Cor^\circ(\tch_{b}(\gamma_\star)\cdot\ch_i(\gamma))\right)\bigg]\, ,
 \end{multline}
 where we have
 used\footnote{In \eqref{eq:diffExp-long}, the elements \(\gamma_\bullet,
   \gamma_\star\) are of complex cohomological degree \(2\).}
$c_1\cdot\gamma_{2s+1-\bullet}=\sum_\star \alpha_{\bullet \star}\gamma_\star$.
Nevertheless, the strategy used in the previous case
can be pursued also
for \(\gamma\in H^2(X)\).
The computation, which is carried out below,  is of course more complicated.

We will study the 
    difference
    \begin{equation}\label{eq:diffL+}
      (\imath u )^k\Cor^\bullet(\rmR_k(\tch_i(\gamma)))-\rmR_k(\Cor^\bullet(\tch_i(\gamma)))\, 
    \end{equation}
  with \(\gamma\in H^2(X)\).
  The expansion of the first term is:
  \begin{multline}\label{eq:H-no-bump1}
    (\imath u)^k\Cor^\circ(\rmR_k(\tch_i(\gamma)))
    =(\imath u)^k\frac{(k+i-2)!}{(i-3)!}\Cor^\circ(\tch_{i+k}(\gamma))\\
    =(\imath u)^k\frac{(k+i-2)!}{(i-3)!}\left(\frac{\fra_{i+k-1}(\gamma)}{(i+k-1)!}    +\frac{(\imath u)^{-1}}{(i+k-2)!}\sum_{|\mu|=i+k-3}\frac{\fra_{\mu_1}\fra_{\mu_2}}{\Aut(\mu)}(\gamma \cdot c_1)
\right.\\
    \left. \hspace{100pt}
      +\frac{(\imath u)^{-2}}{(i+k-2)!}\sum_{|\mu|=i+k-4}
      \frac{\fra_{\mu_1}\fra_{\mu_2}}{\Aut(\mu)}(\gamma \cdot c_1^2)\right.\\
    \left.+\frac{(\imath u)^{-2}}{(i+k-3)!}\sum_{|\mu|=i+k-5}\frac{\fra_{\mu_1}\fra_{\mu_2}\fra_{\mu_3}}{\Aut(\mu)}(\gamma \cdot c_1^2)\right)\,.
  \end{multline}

  The second term of the difference \eqref{eq:diffL+}
  is more involved since we must  transform the descendents \(\fra\) to the standard descendents \(\tau\) before applying the
  shift operator \(\rmR_k\):
   \begin{multline}\label{eq:H-no-bump2}
     \rmR_k(\Cor^\circ(\tch_i(\gamma)))=\\(\imath u)^{-(i-2)}\rmR_k\left(\tau_{i-2}(\gamma)+\left(\sum_{j=1}^{i-2}\frac{1}{j}\right)
       \tau_{i-3}(\gamma \cdot c_1)+\left(\sum_{1\le j<l\le i-2}\frac{1}{jl}\right)\tau_{i-4}(\gamma\cdot c_1^2)\right)\\
     +(\imath u)^{-(i-5)}\rmR_k\left(\frac{(\imath u)^{-1}}{(i-2)!}\left(\sum_{|\mu|=i-3}\frac{\mu_1!\mu_2!}{\Aut(\mu)}
         \left(\tau_{\mu_1-1}\tau_{\mu_2-1}(\gamma \cdot c_1)+
     \left[\left(\sum_{j=1}^{\mu_1}\frac{1}{j}\right)\tau_{\mu_1-2}\tau_{\mu_2-1}\right.\right.\right.\right.\\ \left.\left.\left.\left.
         +\left(\sum_{j=1}^{\mu_2-1}\frac{1}{j}\right)\tau_{\mu_1-1}\tau_{\mu_2-2}\right](\gamma\cdot c_1^2)\right)\right)+\frac{1}{(i-3)!}\sum_{|\mu|=i-5}\frac{\mu_1!\mu_2!\mu_3!}{\Aut(\mu)}
     \tau_{\mu_1-1}\tau_{\mu_2-1}\tau_{\mu_3-1}(\gamma\cdot c_1^2)\right)\, .
 \end{multline}
 Notice the upper limits in the first harmonic sum is \(\mu_1\),
 the terms with \(j=\mu_1\) correspond to the third term of \eqref{eq:decay-intro}.

 We will study the right hand side of \eqref{eq:diffExp-long}
 using \eqref{eq:H-no-bump1} and \eqref{eq:H-no-bump2} in three steps corresponding to the $\tau$-degree.

 \vspace{15pt}
\noindent $\bullet$
 Consider first
 the \(\tau\)-linear terms.
 The \(\tau\)-linear terms of \eqref{eq:H-no-bump1} are
 \begin{multline}\label{eq:cubic-exp1}
 (\imath u)^k \frac{(i+k-2)!}{(i-3)!}\left(\frac{1}{(\imath u)^{i+k-2}}\left( \tau_{i+k-2}(\gamma)+\left(\sum_{j=1}^{i+k-2}\frac{1}{j}\right)\tau_{i+k-3}(c_1\cdot\gamma)\right.\right.
 \\ \left.\left.      +\left(\sum_{1\le j<l\le i+k-2}\frac{1}{jl}\right)\tau_{i+k-4}(c_1^2\cdot \gamma)\right)\right).\end{multline}
The \(\tau\)-linear terms of \eqref{eq:H-no-bump2} is more complicated:
\begin{multline}\label{eq:cubic-exp2}
  (\imath u)^{-i+2}\frac{(i+k-2)!}{(i-3)!}\left(\tau_{i+k-2}(\gamma)+\left(\sum_{j=i-2}^{i+k-2}\frac{1}{j}\right)\tau_{i+k-3}(\gamma\cdot c_1)\right.\\
  \left.+\left(\sum_{i-2\le j<l\le i+k-2}\frac{1}{jl}\right)
    \tau_{i+k-4}(\gamma\cdot c_1^2)+\left(\sum_{j=1}^{i-2}\frac{1}{j}\right)\left[
      \tau_{i+k-3}(\gamma \cdot c_1)+\left(\sum_{j=i-2}^{i+k-2}\frac{1}{j}\right)\tau_{i+k-4}(\gamma \cdot c_1^2)\right]\right.\\
  \left.+\left(\sum_{1\le j<l\le i-2}\frac{1}{jl}\right)
    \tau_{i+k-4}(\gamma\cdot c_1^2)\right)\, .
\end{multline}

The $\tau_{i+k-2}(\gamma)$  terms of \eqref{eq:cubic-exp1} and \eqref{eq:cubic-exp2} match, so cancel in the difference  \eqref{eq:diffL+}. The $\tau_{i+k-3}(\gamma\cdot c_1)$ terms in
\eqref{eq:cubic-exp1} and \eqref{eq:cubic-exp2} almost cancel: the difference is
\begin{equation}\label{diff11}
  (\imath u)^{-i+2}\frac{(i+k-2)!}{(i-2)!}\tau_{i+k-3}(\gamma\cdot c_1)\, .
  \end{equation}
For the $\tau_{i+k-4}(\gamma\cdot c_1^2)$ terms,
we split the prefactor  in 
\eqref{eq:cubic-exp2} as
$$ \sum_{j=1}^{i-2}\frac{1}{j}=
\frac{1}{i-2}+\sum_{j=1}^{i-3}\frac{1}{j}\, $$
and the last coefficient of \eqref{eq:cubic-exp2}
as
$$\sum_{1\le j<l\le i-2}\frac{1}{jl} =
\sum_{1\le j<l\le i-3}\frac{1}{jl}+ \frac{1}{i-2}\sum_{1\le j\le i-3}
\frac{1}{j}\, .$$
Then, we see 
the difference of the $\tau_{i+k-4}(\gamma\cdot c_1^2)$ terms
in \eqref{eq:cubic-exp1}
and \eqref{eq:cubic-exp2} is
\begin{equation}\label{diff22}
(\imath u)^{-i+2}\frac{(i+k-2)!}{(i-2)!}\left(\sum_{j=1}^{i+k-2}\frac{1}{j}\right)
\tau_{i+k-4}(c_1^2\cdot\gamma)\, .
\end{equation}

On the right hand side of equation \eqref{eq:diffExp-long}, the  \(\tau\)-linear terms
\eqref{diff11} and \eqref{diff22}
of the difference
 \eqref{eq:diffL+} 
are canceled with
$$\frac{(\imath u)^{k}}{2}\bigg[k!0!
  \Cor^\circ\left(\frac{\tch_{k+2}(c_1)\ch_i(\gamma)}{(\imath u)^{-2}}\right)\tau_{-2}(\pt)+0!k!\tau_{-2}(\pt)\Cor^\circ\left(
    \frac{\tch_{k+2}(c_1)\ch_i(\gamma)}{(\imath u)^{-2}}\right)\bigg]$$
using \(\tau_{-2}(\pt)=1\).
In fact, after applying  \eqref{eq:2bump-intro}, we find
\begin{multline*}
  (\imath u)^k\frac{k!}{(\imath u)^{-2}}\Cor^\circ(\tch_{k+2}(c_1)\tch_i(\gamma))
\\ = - \frac{(\imath u)^{k+1}}{(i-2)!}\bigg(\fra_{k+i-2}(c_1\gamma)
    +(\imath u)^{-1} \fra(c_1^2\gamma)\bigg)+\dots
  \\
=  -\frac{(\imath u)^{2-i}}{(i-2)!}\bigg((k+i-2)!\bigg(\tau_{k+i-3}(\gamma\cdot c_1)+
  (\sum_{i=1}^{k+i-3}\frac{1}{i})\tau_{k+i-4}(\gamma\cdot c_1^2)
  \bigg)+
  (k+i-3)!\tau_{k+i-4}(\gamma\cdot c_1^2)  \bigg)+\dots\, .
\end{multline*}
where the dots stand for the \(\tau\)-quadratic terms. The second equality
follows from 
the formula \eqref{eq:inv_tau2a}.

\vspace{15pt}
\noindent $\bullet$
Consider next the \(\tau \)-quadratic terms. We start with the 
quadratic terms of complex cohomological degree \(2\). The corresponding
terms from \eqref{eq:H-no-bump1}
are:
\begin{equation}\label{eq:diff-quadL1}
  (\imath u)^k\frac{(k+i-2)!}{(i-3)!}\sum_{|\mu|=i+k-3} \frac{(\imath u)^{-\mu_1-\mu_2+2}\mu_1!\mu_2!}{(\imath u)(i+k-2)!\Aut(\mu)}
  \tau_{\mu_1-1}\tau_{\mu_2-1}(\gamma \cdot c_1)\, .
\end{equation}
The computation of the corresponding terms in \eqref{eq:H-no-bump2} are more involved since the action of the shift operator \(\rmR_k\) depends on the
complex cohomological degree of the descendent:
\begin{multline}\label{eq:diff-quadL2}
  (\imath u)^{-i+4}\frac{1}{(i-2)!}\sum_{|\mu|=i-3}\frac{\mu_1!\mu_2!}
  {\Aut(\mu)}\left(\frac{(\mu_1+k)!}{(\mu_1-1)!}\tau_{\mu_1+k-1}(\gamma\cdot
    c_1)\tau_{\mu_2-1}(\pt) +\right.\\ \left.
    \frac{(\mu_2+k+1)!}{(\mu_2)!}\tau_{\mu_1-1}(\gamma\cdot c_1)\tau_{\mu_2-1+k}(\pt)    \right)\, .
    \end{multline}
The linear combination of the first term of \eqref{eq:diff-quadL2} with \(\mu_1+k-1=a\) and second term with \(\mu_1-1=a\)
is equal to the corresponding term of \eqref{eq:diff-quadL1} with \(\mu_1-1=a\).
Hence, these cancel in the difference. Similar cancellations happen with rest of
the terms. The resulting  difference  of \eqref{eq:diff-quadL1}
and \eqref{eq:diff-quadL2}
is 
\begin{multline}\label{nmm3}
  \frac{(\imath u)^{-i+4}}{(i-3)!}\left(\sum_{|\mu|=i+k-3,\, \mu_1\le k}\mu_1!\mu_2!\frac{\tau_{\mu_1-1}\tau_{\mu_2-1}}{\Aut(\mu)}(\gamma \cdot c_1)
  \right.\\\left. +
    \sum_{|\mu|=i+k-3,\, \mu_2\le k}\mu_1!\mu_2!\frac{\tau_{\mu_1-1}\tau_{\mu_2-1}}{\Aut(\mu)}(\gamma \cdot c_1)\right)
\end{multline}

We will cancel \eqref{nmm3}
with the \(\tau\)-quadratic terms of complex cohomological degree 2
in the sum
\begin{multline}\label{eq:const-quad}
  \frac{(\imath u)^{k}}{2}\bigg[\sum_{a+b=k+2}(a-2)!b!
  \Cor^\circ\left(\frac{\tch_{a}(c_1)\ch_i(\gamma)}{(\imath u)^{b-2}}\right)\tau_{b-2}(\pt)+a!(b-2)!\tau_{a-2}(\pt)\Cor^\circ\left(
    \frac{\tch_{b}(c_1)\ch_i(\gamma)}{(\imath u)^{a-2}}\right)\\-
  \sum_{a+b=k+2}(a-1)!(b-1)!\sum_{\bullet,\star}
  \alpha_{\bullet \star}\, \left(\Cor^\circ(\tch_{a}(\gamma_\bullet)\cdot
     \ch_i(\gamma))\Cor^\circ(\tch_{b}(\gamma_\star))\right.\\\left.+\Cor^\circ(\tch_{a}(\gamma_\bullet))\Cor^\circ(\tch_{b}(\gamma_\star)\cdot\ch_i(\gamma))\right)\bigg]\, .
 \end{multline}
 More precisely, the first and second terms of the last expression yield
 \[\frac{(\imath u)^{-i+4}}{2}\bigg[  b \frac{(a+i-4)!(b-1)!}{(i-2)!}\tau_{a+i-5}(\gamma\cdot c_1)\tau_{b-2}(\pt )+
     a \frac{(a-1)!(b+i-4)!}{(i-2)!}\tau_{a-2}(\pt)\tau_{b+i-5}(\gamma\cdot c_1)
   \bigg]\, ,\]
 and the last two terms yield{\footnote{The sum over $\bullet,\star$
     with coefficient $\alpha_{\bullet\star}$ is implicit.}}
 \begin{multline*}
   -\frac{(\imath u)^{-i-4}}{2}\bigg[  (a-1)\frac{(a+i-4)!(b-1)!}{(i-2)!}
   \tau_{a+i-5}(\gamma_\bullet\cdot \gamma)\tau_{b-2}(\gamma_\star)+\\
    (b-1)\frac{(a-1)!(b+i-4)!}{(i-2)!}
   \tau_{a-2}(\gamma_\bullet\cdot \gamma)\tau_{b+i-5}(\gamma_\star)
   \bigg]\, .\end{multline*}
The cancellation then follows from 
$$ \sum_{\bullet,\star}\alpha_{\bullet\star}\, (\gamma_\bullet\cdot \gamma)\otimes \gamma_\star=\pt\otimes (\gamma\cdot c_1) \ \ \ 
 \text{and}\ \ \  \sum_{\bullet,\star}\alpha_{\bullet\star}\, \gamma_\bullet\otimes(\gamma_\star\cdot \gamma)=(\gamma\cdot c_1)\otimes \pt \, .$$
 We have cancelled all \(\tau\)-quadratic terms of complex cohomological degree 2
 in \eqref{eq:diffExp-long}.

 Let us also observe that the terms of \eqref{eq:const-quad} with \(a=1\)
 and with \(b=1\) cancel out by exceptional bumping with \eqref{eq:2bump-minus-1-factorial} with the term of \eqref{nmm3} with
 \(\mu_1=k\) or \(\mu_2=k\).

A longer computation is needed to deal with \(\tau\)-quadratic terms of
complex cohomological degree 3. Since all such terms have
\(\gamma\cdot c_1^2\) as an argument, we drop the cohomology insertion
from the notation. The corresponding
terms from \eqref{eq:H-no-bump1} are:
\begin{multline}
  \label{eq:quad-P1}
  (\imath u)^k\frac{(k+i-2)!}{(i-3)!}\sum_{|\mu|=i+k-4}\frac{(\imath u)^{-\mu_1-\mu_2}}{\Aut(\mu)(i+k-2)!}\left(\mu_1!\mu_2!+(\mu_1+1)!\mu_2!\left(\sum_{j=1}^{\mu_1}\frac{1}{j}\right)\right.\\ \left.
    +\mu_1!(\mu_2+1)!\left(\sum_{j=1}^{\mu_2}\frac{1}{j}\right)\right)\tau_{\mu_1-1}\tau_{\mu_2-1}\, .
\end{multline}
The corresponding terms from \eqref{eq:H-no-bump2} are:
\begin{multline}
  \label{eq:quad-P2}
  \frac{(\imath u)^{-i+4}}{(i-2)!}\sum_{|\mu|=i-3}\frac{\mu_1!\mu_2!}{\Aut(\mu)}
    \left[\frac{(\mu_1+k)!}{(\mu_1-1)!}\left(\sum_{j=\mu_1}^{\mu_1+k}\frac{1}{j}\right)
      \tau_{\mu_1+k-2}\tau_{\mu_2-1}\right.\\
    \left.+\frac{(\mu_2+k)!}{(\mu_2-1)!}\left(\sum_{j=\mu_2}^{\mu_2+k}\frac{1}{j}\right)\tau_{\mu_1-1}\tau_{\mu_2+k-2}+\left(\sum_{j=1}^{\mu_1}\frac{1}{j}\right)\left(\frac{(\mu_1+k)!}{(\mu_1-1)!}\tau_{\mu_1+k-2}\tau_{\mu_2-1}\right.\right.
    \\ \left.\left.+\frac{(\mu_2+k+1)!}{\mu_2!}\tau_{\mu_1-2}\tau_{\mu_2+k-1}\right)+
      \left(\sum_{j=1}^{\mu_2-1}\frac{1}{j}\right)\left(\frac{(\mu_2+k)!}{(\mu_2-1)!}\tau_{\mu_1-1}\tau_{\mu_2+k-2}\right.\right.\\
    \left.\left.+\frac{(\mu_1+k+1)!}{\mu_1!}\tau_{\mu_1+k-1}\tau_{\mu_2-2}\right)\right]\, .
  \end{multline}
  
  The expression \eqref{eq:quad-P2} is simplified by the
  following strategy. We number the six $\tau$-quadratic terms by their
  order of occurrence in \eqref{eq:quad-P2}.
  The first term
  of \eqref{eq:quad-P2} combines with the third term.  The second term
  combines with the fifth term.
  We also
  split off the summands with \(j=\mu_1+k\) and \(j=\mu_2+k\) from
  the first and  second terms 
  respectively, as well as the summand with \(j=\mu_1\) from the third term. Then, \eqref{eq:quad-P2} equals
  \begin{multline}
    \label{eq:quad-P2-1}
    \frac{(\imath u)^{-i+4}}{(i-2)!}\left(\sum
      \mu_1\frac{(\mu_1+k-1)!\mu_2!}{\Aut(\mu)}\tau_{\mu_1+k-2}\tau_{\mu_2-1}+
      \mu_2\frac{\mu_1!(\mu_2+k-1)!}{\Aut(\mu)}\tau_{\mu_1-1}\tau_{\mu_2+k-2}\right.\\
    \left.+      \mu_1\frac{(\mu_1+k)!\mu_2!}{\Aut(\mu)}\left(\sum_{j=1}^{\mu_1+k-1}\frac{1}{j}\right)\tau_{\mu_1+k-2}\tau_{\mu_2-1}
    +\mu_2\frac{\mu_1!(\mu_2+k)!}{\Aut(\mu)}\left(\sum_{j=1}^{\mu_2+k-1}\frac{1}{j}\right)\tau_{\mu_1-1}\tau_{\mu_2+k-2}\right.\\ \left. +(\mu_2+k+1)\frac{\mu_1!(\mu_2+k)!}{\Aut(\mu)}
      \left(\sum_{j=1}^{\mu_1-1}\frac{1}{j}\right)\tau_{\mu_1-2}\tau_{\mu_2+k-1}\right.\\
\left.+    (\mu_1+k+1)\frac{(\mu_1+k)!\mu_2!}{\Aut(\mu)}
  \left(\sum_{j=1}^{\mu_2-1}\frac{1}{j}\right)\tau_{\mu_1+k-1}\tau_{\mu_2-2}\right.\\
\left. +\frac{(\mu_1+k)!\mu_2!}{\Aut(\mu)}\tau_{\mu_1+k-2}\tau_{\mu_2-1}+
\frac{(\mu_1-1)!(\mu_2+k+1)!}{\Aut(\mu)}\tau_{\mu_1-2}\tau_{\mu_2+k-1}\right),
  \end{multline}
  where the sum is over $\mu_1\geq \mu_2$, \(|\mu|=i-3\).

  
Let us fix an integer \(a\) satisfying \(a>k-2\).
We observe that the sum of the first term
from the first line \eqref{eq:quad-P2-1} 
with \(\mu_1=a+2-k\) and the second term in the last line with
\(\mu_2=a+1-k\) will cancel with the first term of \eqref{eq:quad-P1} with \(\mu_1=a+1\). Also, the sum of the second term from the first line with
\(\mu_2=a+2-k\) and the first term of the last line with
\(\mu_1=a+1-k\) 
will cancel with the first term of \eqref{eq:quad-P1} with \(\mu_2=a+1\).

Similarly, the sum
  of the first term for the second line of \eqref{eq:quad-P2-1} with \(\mu_1=a+2-k\)
  and the first term from the third line of \eqref{eq:quad-P2-1} with \(\mu_1=a+2\)
  cancels with the second term of \eqref{eq:quad-P1} with \(\mu_1=a+1\).
  Finally, the sum of the second term from the second line
  of \eqref{eq:quad-P2-1} with \(\mu_1=a+1\) and the last term from
  the last line of \eqref{eq:quad-P2-1} with \(\mu_1=a+1-k\) cancels with
  the last term of \eqref{eq:quad-P1} with \(\mu_1=a+1\).

  After all of these cancellations, we are left with
  \begin{multline}
  \label{eq:quad-P-diff}
  \sum \frac{(\imath u)^{-i+4}\mu_1!\mu_2!}{\Aut(\mu)(i-3)!}\left(1+(\mu_1+1)\left(\sum_{j=1}^{\mu_1}\frac{1}{j}\right)
    +(\mu_2+1)\left(\sum_{j=1}^{\mu_2}\frac{1}{j}\right)\right)\tau_{\mu_1-1}\tau_{\mu_2-1}\, ,
\end{multline}
where \(\sum\) is the sum of two sub sums: the first is over \(\mu_1+\mu_2=i+k-4\),
\(\mu_1\le k\) and the second is over \(\mu_1+\mu_2=i+k-4\),
\(\mu_2\le k\).

In the difference \eqref{eq:diffExp-long}, the  expression \eqref{eq:quad-P-diff}
is canceled by
the corresponding $\tau$-quadratic terms of complex cohomological degree 3
of  \eqref{eq:const-quad}.
More precisely, the first and second terms of \eqref{eq:const-quad} yield
 \[\frac{(\imath u)^{-i+4}}{2}\bigg[  b \frac{(a+i-4)!(b-1)!}{(i-2)!}\tau_{a+i-6}(\gamma\cdot c^2_1)\tau_{b-2}(\pt )+
     a \frac{(a-1)!(b+i-4)!}{(i-2)!}\tau_{a-2}(\pt)\tau_{b+i-6}(\gamma\cdot c_1^2)
     \bigg]\, ,\]
   after we apply \eqref{eq:2bump-intro}  to these terms and
   drop \(\tau\)-cubic terms and the terms of cohomological
   degree other than \(3\).
   In particular, the factors in first and second terms
   are produced by
   the \(\fra\)-linear term of \eqref{eq:2bump-intro} 
   proportional to \(c_1\).
   
 The last two terms of \eqref{eq:const-quad} yield{\footnote{The sum over $\bullet,\star$
     with coefficient $\alpha_{\bullet\star}$ is implicit.}}
 \begin{multline*}
   -\frac{(\imath u)^{-i+4}}{2}\bigg[  (a-1)\frac{(a+i-4)!(b-1)!}{(i-2)!}
   \tau_{a+i-6}(\gamma_\bullet\cdot \gamma)\tau_{b-2}(\gamma_\star\cdot c_1)+\\
    (b-1)\frac{(a-1)!(b+i-4)!}{(i-2)!}
   \tau_{a-2}(\gamma_\bullet\cdot \gamma)\tau_{b+i-6}(\gamma_\star\cdot c_1)
   \bigg]\, ,\end{multline*}
 after we apply only the parts of \eqref{eq:decay-intro} and
 \eqref{eq:2bump-intro}
 that are not \(c_1\)-proportional, then we use the
 \(\fra\) to \(\tau\) the transition formula \eqref{eq:inv_tau2a}
 and drop the \(\tau\) cubic terms and the terms of
 homological degree other than \(3\).

Together these two sums  combine and cancel the first term of \eqref{eq:quad-P-diff}. To cancel the last two terms of 
\eqref{eq:quad-P-diff}, we follow the same pattern.
We first apply \(c_1^0\)-part of \eqref{eq:2bump-intro}  to
the first and second terms of \eqref{eq:const-quad} and then
apply \(c_1\)-part of the \(\fra\) to \(\tau\) transition formula \eqref{eq:inv_tau2a}. Next, we apply the
\(c_1^0\)-parts of \eqref{eq:decay-intro} and the \eqref{eq:2bump-intro}
 and the \(c_1^1\)-part of
\eqref{eq:inv_tau2a} to the last two terms of \eqref{eq:const-quad}.
After dropping the \(\tau\)-cubic terms and the terms
of complex cohomological degree other than \(3\), we exactly cancel the remaining terms
of \eqref{eq:quad-P-diff}.

\vspace{15pt}
\noindent $\bullet$
Consider finally the \(\tau\)-cubic terms. The cohomological arguments
of these terms are \(c_1^2\cdot \gamma\), so as in the previous computation,
we drop the cohomology insertion from the notation.

After expanding the corresponding terms of \eqref{eq:H-no-bump1}, we obtain:
\begin{equation}\label{eq:exp-easy1}
  \frac{(\imath u)^{-i}(i+k-2)}{(i-3)!}\sum_{|\mu|=i+k-5}\frac{\mu_1!\mu_2!\mu_3!}{\Aut(\mu)}\tau_{\mu_1-1}\tau_{\mu_2-1}\tau_{\mu_3-1}\, .\end{equation}
On the other hand, the corresponding terms from \eqref{eq:H-no-bump2} are
more complicated:
\begin{multline*}
  \frac{(\imath u)^{-i}}{(i-3)!}\sum_{|\mu|=i-5}\frac{\mu_1!\mu_2!\mu_3!}{\Aut(\mu)}\left(\frac{(\mu_1+k+1)!}{\mu_1!}\tau_{\mu_1+k-1}\tau_{\mu_2-1}\tau_{\mu_3-1}\right.\\\left.+\frac{(\mu_2+k+1)!}{\mu_2!}\tau_{\mu_1-1}\tau_{\mu_2+k-1}\tau_{\mu_3-1}+\frac{(\mu_3+k+1)!}{\mu_3!}\tau_{\mu_1-1}\tau_{\mu_2-1}\tau_{\mu_3+k-1}\right),
\end{multline*}
In \eqref{eq:exp-easy1}, we have
\(i+k-2=\sum_{j=1}^3 (\mu_j+1)\). Therefore,
the difference between the last two expressions  is the sum of the monomials
\begin{equation}\label{eq:coeff-diff}
  \left(\sum_{j, \mu_j\le k-2}(\mu_j+2)\right)\frac{(\imath u)^{-i} \mu_1!\mu_2!\mu_3!}{(i-3)!\Aut(\mu)}\tau_{\mu_1-1}\tau_{\mu_2-1}\tau_{\mu_3-1}\, .\end{equation}


Let us restrict our attention to the case when \(i\) is  bigger than \(k\), the other cases are analogous.
After applying the reaction from the last line of (\ref{eq:2bump-intro}),
we obtain a formula for the expressions in the second line of (\ref{eq:diffExp-long}):
\begin{multline}\label{eq:bump-split1}
 (a-2)!b!
 \Cor^\circ\left(\frac{\tch_{a}(c_1)\tch_i(\gamma)}{(\imath u)^{b-2}}\right)\tau_{b-2}(\pt)=\tau\mbox{-quadratic terms }+\\
 \frac{(\imath u)^{-k-i}b!}{(i-2)!}\left(\sum_{|\mu|=a+i-6}\max(\max(\mu_1+1,\mu_2+1),i-2)\frac{\mu_1!\mu_2!}{\Aut(\mu)}\tau_{\mu_1-1}\tau_{\mu_2-1}\right)\tau_{b-2}\, ,
\end{multline}
\begin{multline}
  \label{eq:bump-split2}
  a!(b-2)!\tau_{a-2}(\pt)\Cor^\circ\left(
    \frac{\tch_{b}(c_1)\tch_i(\gamma)}{(\imath u)^{a-2}}\right)=\tau\mbox{-quadratic terms }+
\\  \frac{(\imath u)^{-k-i}a!}{(i-2)!}\left(\sum_{|\mu|=b+i-6}\max(\max(\mu_1+1,\mu_2+1),i-2)\frac{\mu_1!\mu_2!}{\Aut(\mu)}\tau_{\mu_1-1}\tau_{\mu_2-1}\right)\tau_{a-2}\, .
\end{multline}

The terms of \eqref{eq:bump-split1} and \eqref{eq:bump-split2} with
\(\max(\mu_1+1,\mu_2+1)\le i-2\) contribute the monomials:
\begin{equation}
  (\imath u)^{-i}b\cdot\frac{\mu_1!\mu_2!(b-1)!}{(i-3)!}\tau_{\mu_1-1}\tau_{\mu_2-1}\tau_{b-2}\, ,\ \quad (\imath u)^{-i}a\cdot\frac{\mu_1!\mu_2!(a-1)!}{(i-3)!}\tau_{\mu_1-1}\tau_{\mu_2-1}\tau_{a-2}\, .\end{equation}
Note \(a+b=k+2\) in  (\ref{eq:diffExp-long}).
Since \(\max(\mu_1+1,\mu_2+1)\le i-2\) and \(|\mu|=a+i-2\) or \(|\mu|=b+i-2\)
we imply that \(\mu_1+1,\mu_2+1\ge k-1\). Thus
 the corresponding terms of \eqref{eq:bump-split1} and \eqref{eq:bump-split2} cancel with the monomials
\eqref{eq:coeff-diff} such that there is only one \(j\) with \(\mu_j\le k-2.\)

The terms in \eqref{eq:bump-split1} and \eqref{eq:bump-split2} with \(\max(\mu_1+1,\mu_2+1)> i-2\) yield terms:
\[(\imath u)^{-i}b(\mu'+1)\cdot\frac{\mu_1!\mu_2! (b-1)!}{(i-2)!}\tau_{\mu_1-1}\tau_{\mu_2-1}\tau_{b-2}\, ,\quad\
  (\imath u)^{-i}a(\mu'+1)\cdot\frac{\mu_1!\mu_2!(a-1)!}{(i-2)!}\tau_{\mu_1-1}\tau_{\mu_2-1}\tau_{a-2}\, ,\]
where \(\mu'=\max(\mu_1,\mu_2)\).
Both of these terms are of the form:
\begin{equation}\label{cck23}
  (\imath u)^{-i}(\mu_1+1)(\mu_2+1)\cdot \frac{\mu_1!\mu_2!\mu_3!}{(i-2)!}\tau_{\mu_1-1}\tau_{\mu_2-1}\tau_{\mu_3-1}\, ,
  \end{equation}
with \(\mu_1+1>i-2\) and \(|\mu|=i+k-4\).
Since we assumed that
\(i> k\), we have \(\mu_1+\mu_2<k-4\) in \eqref{cck23}. The discussed terms
therefore combine  to yield the sum of monomials:
\begin{equation}
  \label{eq:small-terms1}
 (\imath u)^{-i}(\mu_1+\mu_2+2)(\mu_3+1) \cdot \frac{\mu_1!\mu_2!\mu_3!}{(i-2)!}\tau_{\mu_1-1}\tau_{\mu_2-1}\tau_{\mu_3- 1}\, ,
\end{equation}
where \(\mu_3+1>i-2\) and \(\mu_1,\mu_2\le k-2\).

The terms
\eqref{eq:small-terms1}
combine
with the terms from the expansion of the last two lines of \eqref{eq:diffExp-long}.
Indeed,
since \(\gamma_\bullet\), \(\gamma_\star\) in the last two lines
of \eqref{eq:diffExp-long} are of complex cohomological degree \(2\),
the \(\tau\)-terms result from use of the
\(c_1^1\)-part of \eqref{eq:decay-intro} and of
the \(c_1^0\)-part of \eqref{eq:2bump-intro}.
The expansion of these terms is
a sum of monomials
\begin{equation}
  \label{eq:bump-shift1}
  -(\imath u)^{-i}(b-1)(a-1)\frac{(a+i-4)!\mu_1!\mu_2!}{(i-2)!}
    \tau_{a+i-5}\tau_{\mu_1-1}\tau_{\mu_2-1}\, ,
\end{equation}
where \(|\mu|=b-3\).

The combination of  \eqref{eq:bump-shift1} with \(a=\mu_3-i+4\), \(b=\mu_1+\mu_2+3\)
and \eqref{eq:small-terms1}
matches \eqref{eq:coeff-diff}, since,
in \eqref{eq:bump-shift1}, we have
\[(b-1)(a-1)=(\mu_1+\mu_2+2)(\mu_3-i+3)=(\mu_1+\mu_2+2)(\mu_3+1)-(\mu_1+\mu_2+2)(i-2)\, .\]
We have cancelled all $\tau$-cubic terms.

The assumption
\(\tch_i(\gamma)\in \mathbb{D}_\PT^{X\bigstar}\) implies \(i\ge 3\). Therefore,
in the above computations, we do not see negative factorials in denominators.
\end{proof}

\subsection{Proof of Theorem \ref{thm:intertw}
for $\mathbb{D}^1_\PT \cap \mathbb{D}^{X\bigstar}_{\PT}$
}\label{mono9}
Theorem \ref{thm:intertw}, for all
$D\in \mathbb{D}^1_\PT \cap \mathbb{D}^{X\bigstar}_{\PT}$,
is an immediate consequence of Proposition \ref{prop:no-bump} for singletons
by the following simple argument. Let
$$D = \prod_{i=1}^m \tch_{k_i}(\gamma_i) \in  \mathbb{D}^1_\PT \cap \mathbb{D}^{X\bigstar}_{\PT}\, ,$$
where $\gamma_i\gamma_j=0\in H^*(X)$ for all $i\neq j$.

By definition, for $k\geq 1$,
\begin{eqnarray*}
  \Cor^\bullet\left(\rmL^\PT_k(D)\right) & = & \Cor^\bullet\Big(
      \rmL^\PT_k\big(  \prod_{i=1}^m \tch_{k_i}(\gamma_i)           \big)\Big)\\
  & = &\Cor^\bullet\Bigg(
        \CT_k  \prod_{i=1}^m \tch_{k_i}(\gamma_i)   + \sum_{j=1}^m \rmR_k(\tch_{k_j}(\gamma_j)) \prod_{i\neq j} \tch_{k_i}(\gamma_i)
        \Bigg)\, .
\end{eqnarray*}
Since $\gamma_i\gamma_j=0$ for $i\neq j$,
$$\Cor^\bullet\Big( \CT_k  \prod_{i=1}^m \tch_{k_i}(\gamma_i)\Big)=
(-m+1)\Cor^\bullet (\CT_k)  \prod_{i=1}^m \Cor^\bullet(\tch_{k_i}(\gamma_i))
+ \sum_{j=1}^m \Cor^\bullet(\CT_k \tch_{k_j}(\gamma_j)) \prod_{i\neq j} \Cor^\bullet(\tch_{k_i}(\gamma_i))\,.$$
By Proposition \ref{prop:no-bump},
\begin{eqnarray*}
(\imath u)^{-k}\, 
  \widetilde{\rmL}^{\GW}_k(\Cor^\bullet(\tch_i(\gamma_i))) & = &
           \Cor^\bullet\left(\rmL^\PT_k(\tch_i(\gamma_i))\right) \\
                                                         & = &
                                                               \Cor^\bullet\left( \CT_k\right) \Cor^\bullet\big( \tch_i(\gamma_i)\big) +
         \Cor^\bullet\big( \CT_k\tch_i(\gamma_i)\big) +                          \Cor^\bullet\left(\rmR_k\big(\tch_{k_i}(\gamma_i)\big)\right)\, .                           
\end{eqnarray*}
We conclude
\begin{multline*}
  \Cor^\bullet\left(\rmL^\PT_k(D)\right)  =\\                                               \sum_{j=1}^m (\imath u)^{-k}\, \widetilde{\rmL}^{\GW}_k(\Cor^\bullet(\tch_j(\gamma_j))) \prod_{i\neq j} \Cor^\bullet(\tch_{k_i}(\gamma_i))- (m-1)  \Cor^\bullet (\CT_k)  \prod_{i=1}^m \Cor^\bullet(\tch_{k_i}(\gamma_i)) \, . \end{multline*}
On the other hand,
\begin{multline*}
  (\imath u)^{-k}\, \widetilde{\rmL}^{\GW}_k
  \left( \Cor^\bullet(D)\right)
  =\\ 
\sum_{j=1}^m (\imath u)^{-k}\, \widetilde{\rmL}^{\GW}_k(\Cor^\bullet(\tch_j(\gamma_j))) \prod_{i\neq j} \Cor^\bullet(\tch_{k_i}(\gamma_i))- (m-1)(\imath u)^{-k}\left(
  \frac{(\imath u)^2}{2}\right)\CT_k  \prod_{i=1}^m \Cor^\bullet(\tch_{k_i}(\gamma_i)) \, .
                                              \end{multline*}
The proof is completed by applying \eqref{eq:ctrel}.\qed

\section{Intertwining III: Interacting insertions}
\label{sec:bumping-computations}

\subsection{Overview}
We complete here the proof of Theorem \ref{thm:intertw}.
Since non-interacting insertions have already
been treated in Section \ref{nonintS}, we must address the interacting
cases. In the desired equation,
\begin{equation}\label{ll3355}
  \Cor^\bullet\circ \rmL_k^\PT(D)=(\imath u)^{-k}\, \widetilde{\rmL}_k^\GW\circ \Cor^\bullet(D)\, ,
  \end{equation}
the stable pairs 
descendent insertions of  $D \in \mathbb{D}^{X\bigstar}_{\PT}$ can interact with each other
via the $\GW/\PT$ descendent correspondence on both sides of \eqref{ll3355}.
In addition, the stable pairs descendents
can also interact with constant term of the Virasoro constraints on the left
side. We must control all of  these interactions.

\subsection{Interactions among two insertions}
We start with results which control the
interactions of two descendent insertions. 

\begin{proposition} \label{prop:double-bump1} Let \(\gamma'\in H^2(X)\), \(\gamma''\in H^4(X)\), and let $i\ge 3$, $j\ge 2$.
  Then, for \(k\ge -1\), we have 
  \[(\imath u )^k\, \Cor^\circ(\rmR_k(\tch_i(\gamma')\tch_j(\gamma'')))=\rmR_k(\Cor^\circ(\tch_i(\gamma')\tch_j(\gamma'')))\, .\]
\end{proposition}
\begin{proof}
  We first compute the left side of the equation. After
  applying the shifts,  we obtain
  \[\rmR_k(\tch_i(\gamma')\tch_j(\gamma''))
    =\frac{(i+k-2)!}{(i-3)!}\tch_{i+k}(\gamma')\tch_j(\gamma'')+\frac{(j+k-1)!}{(j-2)!}\tch_i(\gamma')\tch_{j+k}(\gamma'')\, .\]
  We apply the correspondence to the both terms:
  \begin{eqnarray*}
    \Cor^\circ(\rmR_k(\tch_i(\gamma')\tch_j(\gamma'')))&=&(\imath u)^{-1}\left(\frac{1}{(i-3)!(j-2)!}+\frac{(j+k-1)}{(i-2)!(j-2)!}\right)
             \fra_{i+j+k-4}(\gamma'\gamma'')\\ &=&
               (\imath u)^{-i-j-k+4}\frac{(i+j+k-3)!}{(i-2)!(j-2)!}
               \tau_{i+j+k-5}(\gamma'\gamma'')\, .\end{eqnarray*}
 The right side of the equation is
 \begin{eqnarray*}\rmR_k(\Cor^{\circ}(\tch_i(\gamma')\tch_j(\gamma'')))
   &=&\rmR_k\bigg(\frac{(\imath u)^{-1}}{(i-2)!(j-2)!}\fra_{i+j-4}(\gamma'\gamma'')\bigg)\\ &=&
   (\imath u)^{-i-j+4}
                                                                                       \frac{(i+j-4)!}{(i-2)!(j-2)!)}\rmR_k(\tau_{i+j-5}(\gamma'\gamma''))\\
   &=&(\imath u)^{-i-j+4}\frac{(i+j+k-3)!}{(i-2)!(j-2)!}\tau_{i+j+k-5}(\gamma'\gamma'')\, ,
 \end{eqnarray*}
 which matches the left side.
\end{proof}

\begin{proposition}\label{prop:double-bump2}
  Let \(\gamma',\gamma''\in H^2(X)\),   and let $i,j\ge 3$.  
  Then, for \(k\ge -1\), we have
  \begin{multline}\label{eq:double-bump-H}
    (\imath u )^k\, \Cor^\circ(\rmR_k(\tch_i(\gamma')\tch_j(\gamma'')))-\rmR_k(\Cor^\circ(\tch_i(\gamma')\ch_j(\gamma'')))=\\
    \sum_{a+b=k+2}(a-2)!b!\, \Cor^\circ(\tch_{i}(\gamma')\cdot\tch_{j}(\gamma'')\cdot\tch_{a}(c_1))\,\Cor^\circ(\ch_{b}(\pt))\\+
    a!(b-2)!\, \Cor^\circ(\tch_{i}(\gamma')\cdot\tch_{j}(\gamma'')\cdot\tch_{b}(c_1))\,
    \Cor^\circ(\ch_{a}(\pt))\\
-    \sum_{a+b=k+2}(a-1)!(b-1)!\, \sum_{
    \bullet,\star}\alpha_{\bullet\star}\left(\Cor^\circ(\tch_{a}(\gamma_\bullet)\cdot
    \tch_i(\gamma'))\Cor^\circ(\tch_{b}(\gamma_\star)\tch_j(\gamma''))\right.\\\left.
    +\, \Cor^\circ(\tch_{a}(\gamma_\bullet)\tch_j(\gamma''))\Cor^\circ(\tch_{b}(\gamma_\star)\cdot \tch_i(\gamma'))\right)\, .
  \end{multline}
\end{proposition}
\begin{proof}
  We follow the same strategy as in the proof of Proposition \ref{prop:no-bump}.
  We first compute
  \[\rmR_k(\tch_i(\gamma')\tch_j(\gamma''))=
    \frac{(k+i-2)!}{(i-3)!}\tch_{i+k}(\gamma')\tch_j(\gamma'')
    +\frac{(k+j-2)!}{(j-3)!}\tch_i(\gamma')\tch_{j+k}(\gamma'')\, .\]
  After applying the correspondence,  we obtain
  \begin{multline}
   \label{eq:CR} \Cor^\circ(\rmR_k(\tch_i(\gamma')\tch_j(\gamma'')))=-\frac{1}{(i-3)!(j-2)!}\left[\frac{\fra_{i+j+k-4}(\gamma' \gamma'')}{\imath u}+
     \frac{\fra_{i+j+k-5}(\gamma'\gamma'' \cdot c_1)}{(\imath u)^{2}}+\right.\\\left.(\imath u)^{-2}\sum_{|\mu|=i+j+k-6}\frac{f(i+k,j;\mu_1,\mu_2)}{\Aut(\mu)}\fra_{\mu_1}\fra_{\mu_2}
     (\gamma'\gamma''\cdot c_1)\right]-
    \frac{1}{(i-2)!(j-3)!}\left[\frac{\fra_{i+j+k-4}(\gamma'\gamma'')}{\imath u}+\right.\\\left.
      \frac{\fra_{i+j+k-5}(\gamma'\gamma''\cdot c_1)}{(\imath u)^{2}}+(\imath u)^{-2}\sum_{|\mu|=i+j+k-6}\frac{f(i,j+k;\mu_1,\mu_2)}{\Aut(\mu)}\fra_{\mu_1}
      \fra_{\mu_2}(\gamma'\gamma''\cdot c_1)\right],
    \end{multline}
    where \(f(i,j;\mu_1,\mu_2)=\max(\max(i-2,j-2),\max(\mu_1+1,\mu_2+1))\).

    The second
    term of the difference is  easier:
    \begin{multline}
      \label{eq:RC}
      \rmR_k(\Cor^\circ(\tch_i(\gamma')\tch_j(\gamma'')))=
      -\frac{(\imath u)^{-i-j+4}}{(i-2)!(j-2)!}
      \rmR_k\bigg(
          (i+j-4)!\bigg(\tau_{i+j-5}(\gamma'\gamma'')+\\
          \left(\sum_{s=1}^{i+j-4}\frac1s\right)
          \tau_{i+j-6}(\gamma'\gamma''\cdot c_1)\bigg)+(\imath u)^{-2}\sum_{|\mu|=i+j-6}\frac{f(i,j;\mu_1,\mu_2)}{\Aut(\mu)}\fra_{\mu_1}\fra_{\mu_2}(\gamma'\gamma''\cdot c_1)\bigg)\, .
    \end{multline}

    We now analyze the difference. The \(\tau\)-linear terms of
    complex cohomological degree \(2\) in
   $(\imath u )^k$ times \eqref{eq:CR} and \eqref{eq:RC} are matching sums of the monomials:
    \[ (\imath u)^{-i-j+4}\frac{(i+j+k-4)!}{(i-2)!(j-2)!}(i+j-4)
      \tau_{i+j+k-5}(\gamma'\gamma'')\, .\]
    The \(\tau\)-linear terms of cohomological degree \(3\) almost match.
    To be precise, the corresponding terms in \eqref{eq:CR} are sums 
    the monomials:
    \[ (\imath u)^{-i-j+4}\frac{(i+j+k-4)!}{(i-2)!(j-2)!}(i+j-4)\left(\sum_{s=1}^{i+j+k-4}\frac1s\right)\tau_{i+j+k-6}(\gamma'\gamma''\cdot c_1)\, .\]
    Respectively, the corresponding terms in \eqref{eq:RC} are sums of the
    same monomials plus an extra term
\[ (\imath u)^{-i-j+4}\frac{(i+j+k-4)!}{(i-2)!(j-2)!}\tau_{i+j+k-6}(\gamma'\gamma''\cdot c_1)\, .\]

This extra term gets canceled by the term from the second line of \eqref{eq:double-bump-H} with \(b=0\) because of \eqref{bb55}.

    Therefore, the difference of
 $(\imath u )^k$ times
    \eqref{eq:CR} and \eqref{eq:RC} consists only of the \(\tau\)-quadratic terms of complex cohomological degree \(3\). We  omit cohomological classes
    since all the cohomological arguments are $\gamma'\gamma'' \cdot c_1$.
  The corresponding part of \eqref{eq:CR}  is 
    \begin{multline}
      \label{eq:RC-CR-quad}
      \frac{(\imath u)^{-i-j+2}}{(i-2)!(j-2)!}\sum_{|\mu|=i+j+k-6}\frac{\mu_1!\mu_2!}{\Aut(\mu)}
      \left[(i-2)f(i+k,j;\mu)+(j-2)f(i,j+k;\mu)\right]\tau_{\mu_1-1}\tau_{\mu_2-1}\, ,
    \end{multline}
    where we assume that \(f\) vanishes whenever one of the argument is negative.

    We must compare \eqref{eq:RC-CR-quad} with
    the expansion of the last four
    lines of \eqref{eq:double-bump-H}. The first two of the last four lines of \eqref{eq:double-bump-H} expand to
    \begin{multline*}
      -\frac{(\imath u)^{-i-j+2}}{(i-2)!(j-2)!}\sum_{a+b=k+2}(i+j+a-6)b\frac{(i+j+a-7)!(a-1)!}{2}\tau_{i+j+a-8}\tau_{b-2}\\+(i+j+b-6)a\frac{(i+j+b-7)!(a-1)!}{2}\tau_{i+j+b-8}\tau_{a-2}\, .
    \end{multline*}
    The last two lines of the last four lines of \eqref{eq:double-bump-H} expand to
    \begin{multline*}
      \frac{(\imath u)^{-i-j+2}}{(i-2)!(j-2)!}\sum_{a+b=k+2}
      (a-1)(b-1)\bigg((a+i-4)!(b+j-4)!\tau_{a+i-5}\tau_{b+j-5}+\\
        (a+j-4)!(b+i-4)!\tau_{a+j-5}\tau_{b+i-5}\bigg)\, .
      \end{multline*}

      These last two expressions are the \(\tau\)-cubic contribution to the \eqref{eq:double-bump-H} which result from the bumping of
      \(\tch_i(\gamma')\tch_j(\gamma'')\) with the constant term \(\mathrm{T}_k\).
      The corresponding coefficient in front of \(\tau\)-cubic monomial
      is given by the formula \eqref{eq:bump-fun2} below.
      
      To complete the proof, we must match  the coefficients in front of the
      terms in sums above.
     That is we need to compare       two expressions below
      for all \(\mu\) satisfying \(|\mu|=i+j+k-6\):
      \begin{multline}\label{eq:bump-fun1}
        (i-2)f(i+k,j;\mu)+(j-2)f(i,j+k;\mu)-(\mu_1+1)f(i,j;\mu_1-k,\mu_2)\\-
        (\mu_2+1)f(i,j;\mu_1,\mu_2-k)\, ,
      \end{multline}
      \begin{multline}
        \label{eq:bump-fun2}
        [\mu_1+1]_{\le k}(\mu_2+1)+(\mu_1+1)[\mu_2+1]_{\le k}-[\mu_1-i+3]_{\ge 0}[\mu_2-j+3]_{\ge 0}\\-[\mu_1-j+3]_{\ge 0}[\mu_2-i+3]_{\ge 0}\, ,
      \end{multline}
      where \([a]_{\le b}\) and \([a]_{\ge b}\) are cut off functions which equal \(a\) if \(a\) satisfies inequalities \(a\ge b\) and \(a\le b\) respectively (and are
      zero otherwise).
      The matching now is a long and routine check.
      We give some details.

      We can always assume  \(\mu_1\ge \mu_2\) and \(i\ge j\).
      Let us further assume \(k\) is small and 
      \(\mu_1\ge i+k\). If \(\mu_2\ge k \), then the function \eqref{eq:bump-fun1} equals
      \[(i+j-4)(\mu_1+1)-(\mu_1+1)(\mu_1-k+1)-(\mu_2-1)(\mu_1+1)=0\, .
      \]
      The assumed inequalities force all terms in \eqref{eq:bump-fun2} to vanish.

      Next, we assume all but last inequality are true, that is \(\mu_2<k\).
      Then the expression \eqref{eq:bump-fun1} becomes
      \[(i+j-4)(\mu_1+1)-(\mu_1+1)(\mu_1-k+1)=(\mu_1+1)(\mu_2+1)\, .\]
      On the other hand, in \eqref{eq:bump-fun2}, only the second expression does not
      vanish -- the second expression matches \eqref{eq:bump-fun1}.
      Rest of the case can be treated analogously.
    \end{proof}


    \subsection{Interactions among three insertions}
    The last interaction to consider is among three descendent
    insertions. Because of the stationary assumption, there
    is only one case to control.
    
\begin{proposition}\label{prop:triple-bump}
  Let \(\gamma',\gamma'',\gamma''' \in H^2(X)\),  and let \(i_1,i_2,i_3\ge 3\),
 Then, for \(k\ge -1\), we have
  \begin{equation*}
    (\imath u )^k\, \Cor^\circ\left(\rmR_k(
    \tch_{i_1}(\gamma') \tch_{i_2}(\gamma'') \tch_{i_3}(\gamma''')
    )\right)-
  \rmR_k\left(
    \Cor^\circ\big( \tch_{i_1}(\gamma') \tch_{i_2}(\gamma'') \tch_{i_3}(\gamma''')  \big)
  \right)=0\, .
     \end{equation*}
\end{proposition}

For the proof,  we will use the explicit correspondence formula
\eqref{eq:triple-bump-intro}
for the triple interaction:
\begin{equation}
  \label{eq:triple-bump}
  \Cor^\circ(\tch_{i_1}\tch_{i_2}\tch_{i_3})(\gamma)=\frac{(|i|-6)(\imath u)^{-2}}{(i_1-2)!(i_2-2)!(i_3-2)!}\fra_{|i|-7}(\gamma)\end{equation}
where \(|i|=i_1+i_2+i_3\).

\begin{proof}[Proof of Proposition~\ref{prop:triple-bump}]
  We first compute the left side of the equation. To start,
  \begin{eqnarray*}
    \rmR_k(\tch_{i_1}(\gamma')\tch_{i_2}(\gamma'')\tch_{i_3}(\gamma'''))
    &= &\ \, \frac{(i_1+k-2)!}{(i_1-3)!}\tch_{i_1+k}(\gamma')\tch_{i_2}(\gamma'')
         \tch_{i_3}(\gamma''')\\
    & & +\frac{(i_2+k-2)!}{(i_2-3)!}\tch_{i_1}(\gamma')\tch_{i_2+k}(\gamma'')
        \tch_{i_3}(\gamma''')\\
    & & +\frac{(i_3+k-2)!}{(i_3-3)!}\tch_{i_1}(\gamma')\tch_{i_2}(\gamma'')
    \tch_{i_3+k}(\gamma''')\, .
  \end{eqnarray*}
  After applying the triple bumping  and the transition from \(\fra\) descendents to \(\tau\) descendents, we obtain:
  \begin{multline}\label{fvv39}
    \Cor^\circ(\rmR_k(\tch_{i_1}(\gamma')\tch_{i_2}(\gamma'')
    \tch_{i_3}(\gamma''')))=(|i|+k-6)(\imath u)^{-2}\left(\frac{1}{(i_1-3)!(i_2-2)!(i_3-2)!}\right.\\
    \left.+\frac{1}{(i_1-2)!(i_2-3)!(i_3-2)!}+\frac{1}{(i_1-2)!(i_2-2)!(i_3-3)!}\right)\fra_{|i|+k-7}(\gamma'\gamma''\gamma''')\\
=    (\imath u)^{-|i|-k+6}(|i|-6)\frac{(|i|+k-6)!}{(i_1-2)!(i_2-2)!(i_3-2)!} \tau_{|i|+k-8}(\gamma'\gamma''\gamma''')\, .
  \end{multline}
  On the other hand, the right side of the equation equals
  \begin{multline*}
    \rmR_k(\Cor^\circ(\tch_{i_1}(\gamma')\tch_{i_2}(\gamma'')
    \tch_{i_3}(\gamma''')))=
    \frac{(\imath u)^{-2}(|i|-6)}{(i_1-2)!(i_2-2)!(i_3-2)!}\rmR_k(\fra_{|i|-7}(\gamma'\gamma''\gamma'''))\\
    = (\imath u)^{-|i|+6}   (|i|-6)\frac{(|i|+k-6)!}{(i_1-2)!(i_2-2)!(i_3-2)!} \tau_{|i|+k-8}(\gamma'\gamma''\gamma''')\, ,   
  \end{multline*}
  which matches $(\imath u)^k$ times \eqref{fvv39}.
\end{proof}

\subsection{Proof of Theorem~\ref{thm:intertw}}
  Let \(k\ge 1\), and let
  \(D\in \mathbb{D}^{X\bigstar}_{\PT}\).
  To prove the equality 
  \[\Cor^\bullet\circ \rmL_k^\PT(D)=(\imath u)^{-k}\, \widetilde{\rmL}_k^\GW\circ \Cor^\bullet(D)\, , \]
after the restrictions
  $\tau_{-2}(\pt)=1$ and $\tau_{-1}(\gamma)=0$ for \(\gamma\in H^{>2}(X)\), 
  we will expand both sides. The non-interacting case
  was already proven in Section \ref{mono9}.
  Equality in the general case will use
  Propositions~\ref{prop:CTT},~\ref{prop:shiftpt},~\ref{prop:no-bump},~\ref{prop:double-bump1},
  \ref{prop:double-bump2}, and \ref{prop:triple-bump}.

  In the formulas below, we will use short-hand
  notation for the constant term of \(\rmL_k^{\PT}\):
  \[\CT_k=\sum_j \CT_{k,j}^L\CT_{k,j}^R\, ,\]
  where  $L$ and $R$ denote the left and right sides in
  \eqref{ll99v}.

  For $D=\prod_{i=1}^\ell D_i \in \mathbb{D}^{X\bigstar}_{\PT}$, we have
  \begin{eqnarray}\label{eq:CL}
    \Cor^\bullet(\rmL_k^{\PT}(D))&=&\Cor^\bullet(\CT_k D+\rmR_k(D))\\
    \nonumber
&=&    \sum_{P'}\sum_j\prod_{S\in P'} \Cor^\circ(\CT_{k,j}^SD^S)+\sum_{P''}\sum_{t=1}^{\ell(P'')}\Cor^\circ(\rmR_k(D^{S_t}))\prod_{S\in P'',\,  S\ne S_t}\Cor^\circ(D^S)\, .
  \end{eqnarray}
 The first sum is over partitions \(P'\) of \(\{1,\dots,\ell,L,R\}\)
  and \[D^S=\prod_{i\in S\cap \{1,\dots,\ell\}}D_i\, ,\quad \CT_{k,j}^S=\prod_{\gamma\in S\cap \{L,R\}}
    \CT_{k,j}^\gamma\, .\]
The second sum is over partitions $P''$ of \(\{1,\dots,\ell\}\).
  
  We must  compare the \eqref{eq:CL}  with  $(\imath u)^{-k}$  times
  \begin{eqnarray}
    \label{eq:LC}
    \ \ \   \rmL_k^{\GW}(\Cor^\bullet(D)) & =&
   \rmL_k^{\GW}(\sum_P\prod_{S\in P}\Cor^\circ(D^S))\\ \nonumber &=&
    \sum_{P'}\CT_k\prod_{S\in P'}\Cor^\circ(D^S)+\sum_{P''}\sum_{t=1}^{\ell(P'')}
    \rmR_k(\Cor^\circ(D^{S_t}))\prod_{S\in P'', \, S\ne S_t}\Cor^\circ(D^S)\, .
  \end{eqnarray}
  where both sums run over partitions \(P',P''\) of \(\{1,\dots, \ell\}\).

  Since we only work with the stationary descendents, we can assume that the
  parts of partitions in the formulas have at most three elements.
  We will match the terms of \eqref{eq:CL} and
$(\imath u)^{-k}$  times
  \eqref{eq:LC}
  depending on the size of \(S_t\).

  \vspace{12pt}
  \noindent
 $\bullet$ If \(|S_t|=3\), then the terms in \eqref{eq:CL} and \eqref{eq:LC}
  with \(P''=\tilde{P}\sqcup S_t\) are matched by Proposition~\ref{prop:triple-bump}.

   \vspace{8pt}
  \noindent
 $\bullet$
  If \(|S_t|=2\) with \(S_t=\{p,q\}\), then we  use Propositions
  \ref{prop:double-bump1} and \ref{prop:double-bump2} to match the terms
  of \eqref{eq:CL} with \(P''=\tilde{P}\sqcup S_t\) and
 with  \(P'\) equal to \[\tilde{P}\sqcup \{S_t,L\}\sqcup \{R\}\, ,\quad \tilde{P}\sqcup
   \{S_t,R\}\sqcup \{L\}\, ,\quad
   \tilde{P}\sqcup\{p,R\}\sqcup\{q,L\}\, ,\quad
    \tilde{P}\sqcup \{p,L\}\sqcup \{q,R\}\, ,\]
  with the terms of \eqref{eq:LC} with \(P''=\tilde{P}\sqcup S_t\) .

   \vspace{8pt}
  \noindent
 $\bullet$
 If \(|S_t|=1\) with \(S_t=\{p\}\), then we  use
 Proposition~\ref{prop:shiftpt} and Proposition~\ref{prop:no-bump} to
 identify the terms of \eqref{eq:CL} with
  \(P''=\tilde{P}\sqcup S_t\) and with \(P'\) equal to
  \[\tilde{P}\sqcup\{p,L\}\sqcup\{R\}\, ,\quad \tilde{P}\sqcup\{p,R\}
   \sqcup \{L\}\]
  with the terms of \eqref{eq:LC} with
  \(P''=\tilde{P}\sqcup S_t\).

  \vspace{8pt}
  \noindent $\bullet$
  The terms of \eqref{eq:CL} with \(P'=\{L\}\sqcup \{R\}
  \sqcup \tilde{P}\) are equal to the terms of \eqref{eq:LC} with
  \(P'=\tilde{P}\) by Proposition~\ref{prop:CTT}.

  \vspace{12pt}
\noindent The above four cases match all the terms in \eqref{eq:CL}
and \eqref{eq:CL}. \qed

\section{Virasoro constraints for Hilbert schemes of points of 
surfaces}
\label{sec:viras-constr-surf}

Let $S$ be a nonsingular projective toric surface, and let
$$ X= S \times \PP^1\, .$$
As an immediate consequence of Theorem \ref{thm:Vir-toric} applied to the toric
variety $X$, we obtain the following Virasoro constraints: 
\begin{equation}\label{jj234}
\forall k\geq -1\, ,\ \ \ \ \   \Big\langle \mathcal{L}^\PT_k\prod_{i=1}^r \ch_{m_i}(\gamma_i\times \pt)\Big \rangle^{X,\PT}_{n[\PP^1]}
  =0\, ,
  \end{equation}
  where \(\gamma_i\in H^*(X)\), $\pt \in H^2(\PP^1)$ is the point class,
  and 
  \([\PP^1]\in H_2(X)\) is the fiber class.

  We can specialize the constraints \eqref{jj234}
  further to the case
  of the minimal possible Euler characteristic,
  $$P_n(S\times\PP^1,n[\PP^1]) \cong \text{Hilb}^n(S)\, .$$
  
  The above isomorphism of schemes is defined as follows.
  A point $\xi\in \text{Hilb}^n(S)$ corresponds to a 0-dimensional subscheme of $S$ of length $n$. Then, $$\xi\times \PP^1\subseteq S\times \PP^1$$
  is a curve embedded in $S\times \PP^1$ with Euler characteristic $n$ and curve class $n[\PP^1]$. The isomorphism sends $\xi$ to the corresponding stable pair
  $$\calO_{S\times \PP^1}\to\calO_{\xi\times \PP^1}\, .$$
  
  Since the moduli space of stable pairs is nonsingular
  of expected dimension
  $$\int_{n[\PP^1]}c_1(S\times \PP^1)=2n\, ,$$
   the virtual class is the standard
  fundamental class here. The result is a new set
  of Virasoro constraints for  tautological classes
  on $\text{Hilb}^n(S)$.

  To write the Virasoro constraints for $\text{Hilb}^n(S)$
  explicitly, we first define the corresponding descendent
  insertions.
  Let
  $$ 0\rightarrow \mathcal{I} \rightarrow \mathcal{O}_{{\Hilb^n(S)}\times S}
  \rightarrow \mathcal{O}_Z \rightarrow 0$$
  be the universal sequence associated to the universal
  subscheme
  $$ Z \subset S \times \text{Hilb}^n(S) \, .$$
  For $\gamma\in H^*(S)$, let
$$
  \ch_k(\gamma) = -\pi_* \big( \ch_k(\mathcal{I})\cdot \gamma\big)\, ,
  $$
  where $\pi$ is the projection to $\text{Hilb}^n(S)$. We follow as
  closely as possible the descendent notation for 3-folds in  Section \ref{spint}.

  Let \(\mathbb{D}(S)\)  be the commutative algebra
with generators
$$\big \{ \, \ch_i(\gamma)\, | \, i\geq 0\, , \
\gamma\in H^*(S)\, \big \}\, $$
following Section \ref{sec:pt-vir-constraints}.
We define derivations $\rmR_k$ by their actions on the generators:
\[\rmR_k(\ch_i(\gamma))=\left(\prod_{n=0}^k(i+d-2+n)\right)\ch_{i+k}(\gamma)\, ,\quad \gamma\in H^{2d}(S)\, .\]
For $k\geq -1$, we define differential operators
\begin{eqnarray*}
  \rmL^S_k &=&-\sum_{a+b=k+2}(-1)^{(d^L+1)(d^R+1))}(a+d^L-2)!(b+d^R-2)!\ch_a\ch_b(1)\\
  && +\frac{1}{12}
     \sum_{a+b=k}a!b!\ch_a\ch_b(c_1^2+c_2)+\rmR_k\, .
     \end{eqnarray*}
where the sum
is over ordered pairs $(a,b)$
with $a,b\geq 0$.
     
     \begin{theorem} 
       For all  \(k\geq -1\) and  \(D\in \mathbb{D}(S)\),
   we have 
   \[\int_{\Hilb^n(S)}\left(\rmL^S_k+(k+1)!\rmR_{-1} \ch_{k+1}(\pt) \right)(D)=0\,
   \]
   for all $n\geq 0$.
  \end{theorem}

  \begin{proof} For clarity,  we will use superscripts $\ch_i^{\mathrm{Hilb}}$ and $\ch_i^{\mathrm{PT}}$ here to indicate whether we are
    referring to descendents on the Hilbert scheme of $S$ as defined above
    or to stable pairs descendents on $S\times \PP^1$ as defined
    in Section \ref{spint}.

 The universal stable pair of $P_n(S\times \PP^1, n[\PP^1])$ is $\mathbb F= \calO_{Z\times \PP^1}$. Hence,
 $$\ch_i(\mathbb{F}-\calO_{S\times \PP^1\times \Hilb^n(S)})=(\rho\times
 \textup{id})^\ast \ch_i(-\mathcal{I})\, ,$$
 where  $\rho$ is the projection $\rho: S\times \PP^1\to S$.
 By the push-pull formula, for $\delta\in H^\ast(S\times \PP^1)$, we have
 \begin{align*}\ch_i^{\PT}(\delta)&=\pi_\ast \left((\rho \times
                                    \textup{id} )^\ast\left(\ch_i(-\mathcal{I})\cdot \delta \right)\right)\\
                                  &=\pi_\ast\left(\ch_i(-\mathcal I)\cdot \rho_\ast \delta \right)\\ &=
                                    \ch_i^{\Hilb}(\rho_\ast \delta)\, . 
 \end{align*}
 So, $\ch_i^{\PT}(\gamma\times 1)=0$, and $\ch_i^{\PT}(\gamma\times \pt)=\ch_i^{\Hilb}(\gamma)$.

 Since we have the Virasoro constraints \eqref{jj234}, we must only
 check that the composition
\begin{equation}\label{xx99xx}
  \mathbb D(S)\hookrightarrow \mathbb D^{X+}_{\PT}\overset{\mathcal L_k^\PT}{\rightarrow} \mathbb D^{X+}_{\PT}\rightarrow \mathbb D(S)
  \end{equation}
 is precisely $$\rmL^S_k+(k+1)!\rmR_{-1} \ch_{k+1}(\pt)\, .$$
 The first inclusion in \eqref{xx99xx} is determined by sending generators $\ch_{i}^{\Hilb}(\gamma)$ to $\ch_i^{\PT}(\gamma\times \pt)$, and the last map
of \eqref{xx99xx}
 sends $\ch_i^{\PT}(\delta)$ to $\ch_i^{\Hilb}(\rho_\ast \delta)$.
 
 The analysis of the composition
 is straightforward. For the diagonal terms, we note that
 $$c_1(X)=2(1\times \pt)+c_1(S)\times 1$$
 and 
 $$\frac{c_1c_2}{24}(X)=\mathsf{td}_3(X)=\mathsf{td}_2(S)\times \mathsf{td}_2(\PP^1)=\frac{1}{12}(c_1(S)^2+c_2(S))\times \pt.$$
 We write the K\"unneth decomposition of the diagonal as
 $$\Delta\cdot 1=\sum_{i}\theta_i^L\otimes \theta_i^R\in H^\ast(S\times S)\, .$$ Then, the K\"unneth decomposition of $\Delta\cdot c_1\in H^\ast(X\times X)$ is
 $$2\sum_{i}(\theta_i^L\times \pt)\otimes (\theta_i^R\times \pt)+\cdots\, ,$$
 where the remaining terms in the dots are killed by $\rho_\ast$.
 The matching of operators then follows from the definition of
 $\mathcal L_k^\PT$.
 \qedhere
\end{proof}

   \section{$\GW/\PT$ descendent correspondence: review}
\label{gwptrev}

\subsection{Vertex operators}
\label{sec:negative-descendents}
Our goal here is to
 review the results of \cite{OblomkovOkounkovPandharipande18} and to explain
 how Theorem \ref{thm:cor-main} can be derived
 from \cite{OblomkovOkounkovPandharipande18}. The full
 derivation is postponed to Section~\ref{sec:residue-computation}.

 To state the main result of \cite{OblomkovOkounkovPandharipande18},
 we require
 negative descendents $\left\{ \fra_k\right\}$ for \(k\in\zzz_{<0}\)
 which are defined to
 satisfy the Heisenberg relations with positive descendents:
\begin{equation}\label{eq:Heis}
[\fra_k(\alpha),\fra_m(\gamma)]=k\delta_{k+m}\int_X\alpha\cup\gamma\, .
\end{equation}
The descendents \(\left\{ \fra_k\right\} \) for ${k\in\zzz\setminus \{0\}}$
generate the \(H^*(X)\)-algebra \(\Heis_X\).

For curve class \(\beta\in H_2(X)\),
there is a geometrically defined Gromov-Witten
evaluation \(\langle\cdot\rangle_\beta\) map on the algebra generated by the
non-negative descendents.
We can extend the evaluation map to the whole algebra \(\Heis_X\) by
defining
$$\big\langle\mathfrak{a}_{k}(\gamma)\Phi\big\rangle^{X,\GW}_\beta=\left[\int_X \big(-c_1\delta_{k+1}+\delta_{k+2}iu\big)\cdot \gamma\right]\,  
\big\langle \Phi\big\rangle^{X,\GW}_\beta\, ,\quad k<0\, .$$
We assemble the operators $\mathfrak{a}_k$ in the following generating function:
\begin{equation}\label{eq:phiz}
\phi(z)=\sum_{n>0}\frac{\mathfrak{a}_n}{n} \left(\frac{\imath zc_1}{u}\right)^{-n}+\frac{1}{c_1}\sum_{n<0}\frac{\mathfrak{a}_n}{n} \left(\frac{\imath zc_1}{u}\right)^{-n}\, .
\end{equation}



The main objects of study in  \cite{OblomkovOkounkovPandharipande18} are the  vertex  operators
\begin{equation}\label{eq:H^GW}
\mathrm{H}^{\mathsf{GW}}( x)=\sum_{k=0}^{\infty} \mathrm{H}^{\mathsf {GW}}_k x^{k+1}=\mathrm{Res}_{w=\infty}\left(\frac{\sqrt{dydw}}{y-w}
:e^{\theta\phi(y)-\theta\phi(w)}:\right)\, ,
\end{equation}
where $y$, $w$, and $x$ satisfy the constraints
\begin{equation}\label{eq:master_curve}
ye^y=we^w e^{-x/\theta}\, , \  \quad \theta^{-2}=-c_2(T_X)\, .
\end{equation}
Here, \(\mathrm{Res}_{w=\infty}\) denotes $\frac{1}{2\pi \imath}$
times the integral along a small loop around 
$
w=\infty. 
$

Normally ordered monomials
$$ \fra_{i_1}\fra_{i_2}\dots\fra_{i_k},\quad i_1\le i_2\le\dots\le i_k,$$
form a linear basis of $\mathsf{Heis}$. Respectively, we use \(:\cdot:\)
for the normal ordering operation
$$:\prod_j\fra_{i_j}:\,\,\,= \fra_{i_1}\fra_{i_2}\dots\fra_{i_k},\quad i_1\le i_2\le\dots\le i_k,$$
Extended \(H^*(X)\)-linearly to the whole algebra \(\mathrm{Heis}_X\).

Let us notice that the equation \eqref{eq:master_curve} as well as the vertex
operator \eqref{eq:H^GW} have symmetry
\[y\mapsto w,\quad w\mapsto y,\quad \theta\mapsto -\theta,\quad x\mapsto x.\]
This symmetry implies that the only even powers of \(\theta\)
appear in the expansion of \eqref{eq:H^GW} (see Lemma 15 from \cite{OblomkovOkounkovPandharipande18} for more discussions
and further properties of the vertex operator).

The  operators \(\Hr^{\mathsf{GW}}_k\) are mutually commutative.
To obtain explicit formulas for \(\Hr^{\mathsf{GW}}_k\), we use the Lambert function to solve equation
(\ref{eq:master_curve}) and express \(y\) in terms of \(x,w\). 
The integral in the definition
of \(\Hr^{\mathsf{GW}}_k\) can be interpreted as an 
extraction of the coefficient of \(w^{-1}\).
The descendent classes
$$\Hr^{\mathsf{GW}}_k(\gamma) \in \mathsf{Heis}_X$$
are then obtained using the Sweedler coproduct.
We also use the Sweedler coproduct conventions in
\begin{equation}
  \label{ccswwd}
  \Hr^{\mathsf{GW}}_{\vec{k}}(\gamma)=\prod_{i=1}^m\Hr^{\mathsf{GW}}_{k_i}(\gamma)\, ,
  \ \ \quad \vec{k}=(k_1,\dots,k_m)\, .
  \end{equation}

In the Sweedler conventions \cite{Kas},
we abbreviate notation for the intersection
with the  small diagonal \(\Delta_n\subset X^n\) with the pull-back of a class
\(\gamma\in H^*(X)\):
\[ H^*(X^n)\ni [\Delta_n]\cdot \gamma=\sum_k \gamma_1^k\otimes \dots\gamma_n^k=
  \gamma_{(1)}\otimes\dots\otimes\gamma_{(n)}\, .\]
Thus, the formula \eqref{ccswwd} expands as
\[\prod_{i=1}^m \Hr^{\mathsf{GW}}_{k_i}(\gamma)=\prod_{i=1}^m \Hr^{\mathsf{GW}}_{k_i}(\gamma_{(i)})\, .\]

\subsection{Stable pairs}
The stable pairs analogues of
the operators \(\Hr^{\mathsf{GW}}_{\vec{k}}(\gamma)\) are products
of \(\Hr^{\PT}_k(\gamma)\) defined as follows.

The classes $\mathrm{H}^{\PT}_k(\gamma)$ 
 are linear combinations of  descendents 
on the moduli spaces of stable pairs. 
Let
$$\Hr^{\PT}_k(\gamma)= \pi_*\left( \Hr^{\PT}_k\cdot \gamma\right)\,
\in \bigoplus_{n\in \mathbb{Z}} H^*(P_n(X,\beta))\, ,$$
where the classes $\Hr^{\PT}_k\in \bigoplus_{n\in \mathbb{Z}}
H^*(X\times P_n(X,\beta))$
are defined by
\begin{eqnarray*}
\mathrm{H}^{\PT}(x)&=&\sum_{k=0}^\infty x^{k+1} \Hr_k^\PT \\
&=&\mathcal{S}^{-1}\left(\frac{x}{\theta}\right)
\sum_{k=0}^\infty x^k \text{ch}_k(\mathbb{F}-\mathcal{O})\, ,
\end{eqnarray*}
where 
$$ \theta^{-2}=-c_2(T_X)\, ,\ \ \ \ 
\mathcal{S}(x)=\frac{e^{x/2}-e^{-x/2}}{x}\, .$$
In particular, we have
\[\mathrm{H}^{\PT}_k=\text{ch}_{k+1}(\mathbb{F})+\frac{c_2}{24}\text{ch}_{k-1}(\mathbb{F})+\frac{7c_2^2}{5760}\text{ch}_{k-3}(\mathbb{F})
  + \ldots\, .\]

\subsection{Equivariant correspondence} \label{ccppvv}
All the definitions and
construction introduced in Section \ref{sec:negative-descendents}
have canonical lifts to the equivariant setting with respect to a group
action on the variety $X$. 
We review here   the
equivariant  $\GW/\PT$ descendent correspondence \cite{PPDC}. 

The most natural setting is the capped vertex
formalism of \cite{MOOP, PPDC}
which we review briefly here.
Let the 
3-dimensional torus 
$$\mathsf{T}=\cc^* \times \cc^* \times \cc^*$$ 
act on $\mathbf{P}^1\times\mathbf{P}^1\times\mathbf{P}^1$ diagonally.
The tangent weights of the $\mathsf{T}$-action at the point 
$$\mathsf{p}=0\times 0\times 0 \in \mathbf{P}^1\times\mathbf{P}^1\times\mathbf{P}^1$$ are $s_1,s_2,s_3$. The $\mathsf{T}$-equivariant cohomology ring of a point
is 
$$H_{\mathsf{T}}(\bullet)=\cc[s_1,s_2,s_3]\, .$$
We have the following factorization 
of the restriction of class 
$c_1c_2-c_3$ of $X$ to 
 $\mathsf{p}$,
$$ c_1c_2-c_3=(s_1+s_2)(s_1+s_3)(s_2+s_3)\, ,$$
where $c_i=c_i(T_X)$.

Let $U\subset\mathbf{P}^1\times\mathbf{P}^1\times\mathbf{P}^1$ be the 
$\mathsf{T}$-equivariant $3$-fold obtained by removing the 
three $\mathsf{T}$-equivariant lines $L_1,L_2,L_3$ passing through the point $\infty\times\infty\times\infty$. Let $D_i\subset U$
be the divisor with $i^{th}$ coordinate $\infty$.  
For a  triple of partitions $\mu_1,\mu_2,\mu_3$, let
\begin{equation}\label{j233}
 \Big\langle\prod_i\,  \tau_{k_i}(\mathsf{p})\, \Big|\, \mu_1,\mu_2,\mu_3\, \Big\rangle^{\mathsf{GW},\mathsf{T}}_{U,D}\, ,\quad
  \Big \langle\, \prod_i \text{ch}_{k_i}(\mathsf{p})\, \Big|\, \mu_1,\mu_2,\mu_3\, 
\Big \rangle^{\mathsf{PT},\mathsf{T}}_{U,D}
\end{equation}
denote the 
generating series of the $\mathsf{T}$-equivariant
relative Gromov-Witten and stable pairs
 invariants of the pair $$D=\cup_i D_i\subset U$$ 
with relative conditions $\mu_i$ along the divisor $D_i$.

The stable maps spaces are always taken with no contracted connected
components of genus great than or equal to 2.
The series \eqref{j233}  are the {\em capped descendent vertices}
following the conventions of
\cite{OblomkovOkounkovPandharipande18}.

\begin{theorem}\cite{OblomkovOkounkovPandharipande18}\label{thm:two_leg}   After the change of variables 
$-q=e^{iu}$
the following correspondence between
the 2-leg capped descendent vertices holds:
 $$ \Big\langle\, \prod_i \Hr^{\mathsf{GW}}_{k_i}(\mathsf{p})\, \Big|\, \mu_1,\mu_2,\emptyset\, 
\Big\rangle^{\mathsf{GW},\mathsf{T}}_{U,D}=
q^{-|\mu_1|-|\mu_2|}\Big\langle\, \prod_i \Hr^{\PT}_{k_i}(\mathsf{p})\, \Big|\, 
\mu_1,\mu_2,\emptyset\, \Big \rangle^{\PT,\mathsf{T}}_{U,D}$$
$\mod (s_1+s_3)(s_2+s_3)$.
\end{theorem}

The result of Theorem \ref{thm:two_leg} has two defects. Since the
third partition is empty, the result only covers the {\em 2-leg} case.
Moreover, the equality of the correspondence is not proven exactly,
but only mod $(s_1+s_3)(s_2+s_3)$. For the 1-leg vertex with partitions
$(\mu_1, \emptyset,\emptyset)$, 
Theorem \ref{thm:two_leg}
can be restricted in two ways to obtain the equality of the
correspondence
$$\text{mod} \ \ (s_1+s_3)(s_1+s_2)(s_2+s_3)\, .$$



\subsection{Non-equivariant limit} 
\label{sec:non-equiv-limit}
By following the arguments of \cite{PPDC},
a  non-equivariant $\GW/\PT$ descendent correspondence for stationary
insertions
is derived in \cite{OblomkovOkounkovPandharipande18}.
For our statements, we will follow as closely as possible
the notation of \cite{OblomkovOkounkovPandharipande18, PPDC}.

Let \(\mathsf{Heis}^{c}\) be the Heisenberg algebra with
generators \(\mathfrak{a}_{k\in\zzz\setminus \{0\}}\),
coefficients \(\cc[c_1,c_2]\), and  relations 
\[[\mathfrak{a}_k,\mathfrak{a}_m]=k\delta_{k+m}c_1c_2\, .\]
Let \(\mathsf{Heis}^c_+\subset \mathsf{Heis}^c\) be the
subalgebra generated by the elements \(\mathfrak{a}_{k>0}\), and  define the \(\cc[c_1,c_2]\)-linear map
\begin{equation}\label{cclin}
\mathsf{Heis}^c\rightarrow \mathsf{Heis}^c_+\, , \ \ \
\Phi\mapsto \widehat{\Phi} 
\end{equation}
by $\widehat{\mathfrak{a}}_k= \mathfrak{a}_k$ for $k>0$ and
\begin{equation}
  \label{eq:vac-eval}
  \widehat{\mathfrak{a}_k \Phi}=(-c_1\delta_{k+1}+\delta_{k+2}iu)\widehat{\Phi}\,,\ \ \
\text{for}\ k<0\, .
\end{equation}
When restricted to the subalgebra \(\mathsf{Heis}^{c}_+\),
the \(\cc[c_1,c_2]\)-linear map \eqref{cclin}  is an isomorphism.

For a nonsingular projective 3-fold $X$ and classes $\gamma_1, \ldots,\gamma_l\in H^*(X)$,
the hat operation make no difference inside the Gromov-Witten bracket,
\begin{equation}\label{fredfredfred}
  \langle \Hr_{\vec{k}}^{\GW}(\gamma)\rangle^{X,\GW}_\beta=\langle     \widehat{\Hr}_{\vec{k}}^{\GW}(\gamma)\rangle^{\GW}_\beta,
\end{equation}
because the treatment of the negative descendents on
the left side is compatible with the treatment of
the negative descendents by the hat operation.

Let \(\vec{k}=(k_1,\ldots,k_l)\) be a vector of non-negative integers.
Following
\cite{PPDC},
 we define the following element of \(\mathsf{Heis}^c_+\):
\[\widetilde{\Hr}_{\vec{k}}=\frac{1}{(c_1c_2)^{l-1}}
\sum_{\text{set partitions $P$ of \{1,\dots,l\}}}(-1)^{|P|-1}(|P|-1)!\prod_{S\in P}\widehat{\Hr}^{\mathsf{GW}}_{\vec{k}_S}\, ,\]
where \(\Hr^{\mathsf{GW}}_{\vec{k}_S}=\prod_{i\in S} \Hr^{\mathsf{GW}}_{k_i}\) and
  the element \(\Hr_k^{\GW}\in \mathsf{Heis}^c\) is a linear combination of monomials of \(\mathfrak{a}_i\), the expression is given by (\ref{eq:H^GW}).


For classes \(\gamma_1,\ldots,\gamma_l\in H^*(X)\)  and a
vector \(\vec{k}=(k_1,\ldots,k_l)\) of non-negative
integers, we define
\[\overline{\Hr_{k_1}(\gamma_1)\dots\Hr_{k_l}(\gamma_l)}=\sum_{\text{set partitions $P$ of \{1,\dots,l\}}}\,
\prod_{S\in P}\widetilde{\Hr}_{\vec{k}_S}(\gamma_S)\, ,\]
where \(\gamma_S=\prod_{i\in S}\gamma_i\).

\begin{theorem} \cite{OblomkovOkounkovPandharipande18}\label{thm:noneq}
 Let  \(X\) be a nonsingular projective toric 3-fold, and let 
\(\gamma_i\in H^{\geq 2}(X,\cc)\). After the change of variables
$-q=e^{iu}$, we have
 $$\Big \langle\overline{ \Hr_{k_1}(\gamma_1)\dots \Hr_{k_l}(\gamma_l)} \Big \rangle_{\beta}^{\mathsf{GW}}=
q^{-d/2}\Big \langle \Hr^{\PT}_{k_1}(\gamma_1)\dots \Hr^{\PT}_{k_l}(\gamma_l) \Big \rangle_{\beta}^{\mathsf{PT}}
\,,$$
  where \(d=\int_\beta c_1\).
\end{theorem}


\subsection{Examples for \(X=\mathbb{P}^3\)}
The prefactor
$\mathcal{S}^{-1}\left(\frac{x}{\theta}\right)$
in front of \(\sum_{k=0}^\infty x^k\ch_k(\mathbb{F}-\mathcal{O}) \) in the formula for \(\Hr^{\PT}(x)\) has an expansion
which the following initial terms:
\[1+\frac{c_2}{24}x^2+\frac{7c_2^2}{5760}x^4+\dots\, .\]
Therefore, the non-equivariant limit of \(\Hr_k^{\PT}(\gamma)\) is
\[\left(\text{ch}_{k+1}(\gamma)+\frac{1}{24}\text{ch}_{k-1}(\gamma\cdot c_2)\right)\, .\]

On the Gromov-Witten  side of the correspondence, we have
\[ \langle \Hr_1^{\GW}(\gamma)\Phi\rangle=\langle\mathfrak{a}_1(\gamma)\Phi\rangle\, ,\ \quad \langle \Hr_2^{\GW}(\gamma)\Phi\rangle=\frac12\langle \mathfrak{a}_2(\gamma)\Phi\rangle\, ,\]
\[  \langle \Hr_3^{\GW}(\gamma)\Phi\rangle=\frac16\langle \mathfrak{a}_3(\gamma)\Phi\rangle+\frac{1}{24u^2}\langle c_1^2c_2\cdot\Phi\rangle\, ,\]
\[\langle \Hr_4^{\GW}(\gamma)\Phi\rangle=
  \frac1{24}\langle \mathfrak{a}_4(\gamma)\Phi\rangle-\frac{i}{12 u}\langle\mathfrak{a}_1^2(c_1\cdot\gamma)\Phi\rangle-\frac{5i}{144u^3}\langle c_1^3c_2\cdot\Phi\rangle\, ,\]
\begin{multline*} \langle\Hr_5^{\GW}\Phi\rangle=\frac1{120}\langle\mathfrak{a}_5(\gamma)\Phi\rangle-\frac{i}{24u}\langle\mathfrak{a}_1\mathfrak{a}_2(c_1\cdot \gamma)\Phi\rangle
  -\frac{1}{48u^2}\langle\mathfrak{a}^2_1(c_1^2\cdot\gamma)\Phi\rangle\\
  +\frac{1}{24u^2}\langle\mathfrak{a}_1(c_1^2c_2\cdot\gamma)\Phi\rangle-\frac{1}{64u^4}\langle c_1^4c_2\cdot \Phi\rangle\, .\end{multline*}
The operators \(\mathfrak{a}_k\) are expressed in terms of
standard descendents{\footnote{For $\mathfrak{a}_1(\gamma)$, the
    term $-\frac{c_2}{24}$ on the right is the constant
    $-\frac{1}{24}\int_Xc_2\gamma$.}}
\begin{eqnarray}\label{eq:aak}
  \mathfrak{a}_1&=&\tau_0-\frac{c_2}{24}\, , \\  \nonumber
  iu\mathfrak{a}_2/2&=&\tau_1+c_1\cdot\tau_0\, ,\\ \nonumber
  -u^2\mathfrak{a}_3/3&=&2\tau_2+3c_1\cdot\tau_1+c_1^2\cdot\tau_0\, ,
  \\ \nonumber
  -iu^3\mathfrak{a}_4/4&=&6\tau_3+11c_1\cdot\tau_2+6c_1^2\tau_1+c_1^3\cdot\tau_0
                           \, ,\\  \nonumber
  u^4\mathfrak{a}_5/5&=&24\tau_4+50c_1\cdot\tau_3+35c_1^2\cdot\tau_2+10c_1^3\cdot\tau_1+c_1^4\cdot\tau_0\, .
 \end{eqnarray}

 The descendent correspondence of Theorem \ref{thm:noneq}
 implies relations for stable pairs and Gromov-Witten
 invariants of \(\mathbf{P}^3\). For example, for $\beta$ of degree 1,
\begin{eqnarray}\label{eq:ch5L}
 - i q^{-2}\langle \text{ch}_5(\mathsf{L})\rangle &=&
                                                     \left(\frac1{u^3}\langle\tau_3(\mathsf{L})\rangle+\frac{22}{3u^3}\langle\tau_2(\mathsf{p})\rangle-\frac1{3u}\langle\tau_0\tau_0(\mathsf{p})\rangle\right)\, , \\ \nonumber
  q^{-2}\left(\langle \text{ch}_6(\mathsf{H})\rangle+\frac{1}{4}\langle \text{ch}_4(\mathsf{p})\rangle\right)&=&\left( \frac{1}{u^4}\langle\tau_4(\mathsf{H})\rangle+\frac{25}{3u^4}\langle\tau_3(\mathsf{L})\rangle+                                    \frac{70}{3u^4}\langle\tau_2(\mathsf{p})\rangle\right.\\ \nonumber
  & & \ \ 
      \left.-\frac{1}{3u^2}\langle\tau_0\tau_1(\mathsf{L})\rangle+\frac{5}{3u^2}\langle\tau_0\tau_0(\mathsf{p})\rangle\right). 
  \end{eqnarray}
  Here, \(\mathsf{p}\) is the class of point, \(\mathsf{L}\) is the class of line and \(\mathsf{H}\)  is the class of hyperplane.{\footnote{We can
      also check the relations \eqref{eq:ch5L}
      numerically up to \(u^8\) with the help of Gathmann's code
      on the Gromov-Witten side
      and previously known computations for stable pairs \cite{P}.}}

\subsection{Residues}
\label{sec:react-as-resid}

To complete our proof of Theorem~\ref{thm:cor-main}, we will compute the residues \eqref{eq:H^GW}. More precisely, we will prove the following result.
\begin{proposition}\label{thm:H-Cor}
  For \(k_i\in \zzz_{\geq 0}\) and \(\gamma_i\in H^{\ge 2}(X)\) such that
  \(\tch_{k_i+2}(\gamma_i)\in \mathbb{D}^{X\bigstar}_{\PT}\), we have:
  \begin{eqnarray*}
  \widetilde{\Hr}_{k_1+1}(\gamma_1)&=&\Cor^\circ(\tch_{k_1+2}(\gamma_1))\, ,\\
  \widetilde{\Hr}_{k_1+1,k_2+1}(\gamma_1\cdot\gamma_2)&=&\Cor^\circ(\tch_{k_1+2}(\gamma_1)\tch_{k_2+2}(\gamma_2))\, , \\
    \widetilde{\Hr}_{k_1+1,k_2+1,k_3+1}(\gamma_1\cdot\gamma_2\cdot\gamma_3)&=&\Cor^\circ(\tch_{k_1+2}(\gamma_1)\tch_{k_2+2}(\gamma_2)\tch_{k_3+2}(\gamma_3))\, ,
    \end{eqnarray*}                                               
where the right side is defined by \eqref{eq:decay-intro}-\eqref{eq:triple-bump-intro}.
  \end{proposition}

\section{Residue computation}
\label{sec:residue-computation}

\subsection{Preliminary computations}
\label{sec:prel-comp}

Before starting the proof of the Proposition~\ref{thm:H-Cor},
we compute the
expansion of the terms of the residue formula \eqref{eq:H^GW}.

Consider first the constraint equation \eqref{eq:master_curve}. Solutions of the
equation are formal power series in the variable
\[r=1/\theta\, , \ \ \ \ \theta^{-2}=-c_2(T_X)\,.\]
We can solve the constraint equation iteratively
in powers of \(r\). Indeed, modulo \(r^1\), the
constraint equation implies \(w=y\), and we
start the expansion by
$$ w(x,y) =y+O(r)\, .$$
To find the next term of \(r\) in the
expansion of \(w(x,y)\), we substitute
$$w(x,y)=y+f_1(x,y)r$$
into \eqref{eq:master_curve} and expand the result of the
substitution in powers of \(r\).
The coefficient of \(r^1\) in
the expansion gives a linear equation which determines \(f_1\).
After iterating the above procedure three times, we  obtain
\begin{equation}
  \label{eq:w(z)}
  w(x,y)=y-xr\frac{y}{y+1}+(xr)^2\frac{y}{2(y+1)^3}+(xr)^3\frac{2y-1}{6(y+1)^5}+O(r^4)\, .
\end{equation}

To see the expansion of the residue \eqref{eq:H^GW}  has  positive powers of
\(t=c_1\),  we use a change variables:
\begin{equation}y=v/t\, .\label{eq:y-v}\end{equation}
The residue with respect to $w$ on right side of \eqref{eq:H^GW} is converted
to a residue with respect to $y$ via  \eqref{eq:w(z)}. Using \eqref{eq:y-v},
we will compute the residue with respect to \(v\).

In the new variables, we have
\[\sqrt{dwdy}=\left(1-\frac{xrt}{2(v+t)}-\frac{(xr)^2t^3(4v-t)}{8(v+t)^4}\right)\frac{dv}{t}+O(r^3)\, .\]

After we normal order the elements of the Heisenberg algebra in the expression for the vertex operator \(\Hr^{\GW}(x)\), 
the negative
Heisenberg operators end up next to the vacuum  \(\langle\, |\)
inside the bracket
\(\langle\cdot \rangle^{\GW}\).
Relation \eqref{eq:vac-eval},  which governs interaction with \(\langle\,|\), yields the following factor in the expression under
the residue:
\begin{eqnarray}
  \label{eq:E-factor}
  \mathrm{E}&=&\exp\left(-\frac{t}{2u}\left(\frac{w(y)^2-y^2}{r}\right)-\frac{t}{u}\left(\frac{w(y)-y}{r}\right)\right)\\ \nonumber
  &=&\exp\left(\frac{xv}{u}\right)
  \left(1-\frac{trx^2v}{2u(v+t)}+\frac{t^2r^2(3xv^2+3txv+4t^2u)}{24u(v+t)^3}\right)+O(r^3)\, .
\end{eqnarray}

The inverse of \(y-w\) in \eqref{eq:H^GW}  becomes the factor:
\begin{equation}
  \label{eq:D-factor}
  \mathrm{D}=-\frac{r}{w(y)-y}=\frac{v+t}{v}\left(1+\frac{t^2rx}{2(v+t)^2}+\frac{t^3r^2x^2(4v+t)}{12(v+t)^4}\right)+O(r^3)\, .
\end{equation}

The elements of the Heisenberg algebra that participate in the residue formula are packed into the vertex operator:
\begin{gather*}
  \mathrm{V}=\mathrm{V}_+\cdot\mathrm{V}_-\, ,\quad \mathrm{V}_+(x,y)=\exp\left(\frac{1}{r}\sum_{n>0}\frac{\fra_n}{n(\imath ut)^n}( y^{-n}- w(y)^{-n})\right)\, ,\\
    \mathrm{V}_-(x,y)=\exp\left(\frac{1}{rt}\sum_{n<0}\frac{\fra_n}{n(\imath ut)^n}( y^{-n}- w(y)^{-n})\right)\, .
\end{gather*}

Thus we need to compute the difference of powers in the expression for
the vertex operators. Using formula for \(w(y)\) \eqref{eq:w(z)}, we obtain:
\begin{multline}
  \label{eq:power-diff}
  \frac{( yt)^{-n}-( w(y)t)^{-n}}{t^nr}=\frac{nxt}{v^n(v+t)}+
  nx^2rt^2\frac{((n+1)v+nt)}{v^n(v+t)^3}\\
  +nx^3r^2t^3\frac{n((n+1)(n+2)v^2+(2n^2+3n-1)tz+n^2t^2)}{6v^n(v+t)^5}+O(r^3)\, .
\end{multline}

The above calculations  yield the
leading terms of all algebraic expressions occurring in
formula \eqref{eq:H^GW}
for the vertex
operator \(\Hr^{\GW}(x)\).  As we will see in Section
\ref{sec:proofs-bump-form},
the knowledge of these leading terms
almost immediately leads to the  simplest case of the descendent
correspondence \eqref{eq:decay-intro}.
For the  other two cases \eqref{eq:2bump-intro} and
\eqref{eq:triple-bump-intro},
we must analyze the interaction of two and three  vertex operators \(\Hr^{\GW}(x)\).
We apply standard vertex operator techniques to complete
the proof of Proposition \ref{thm:H-Cor} in Section \ref{sec:proofs-bump-form}.

\subsection{Proof of Proposition \ref{thm:H-Cor}}
\label{sec:proofs-bump-form}
 
\subsubsection{Case $\widetilde{\Hr}_{k_1+1}(\gamma_1)$} \label{cc1aa}
We start with the proof of
the formula for the self-reaction.
  We must analyze the \(r\) expansion of the residue
  \begin{equation}\label{eq:res-edv}\widetilde{\Hr}(x) = \widehat{\Hr}^\GW(x)=\Res_{v=\infty}\frac{1}{t}\mathrm{E}\cdot \mathrm{D}\cdot\mathrm{V}_+\, .\end{equation}
  More precisely,  we must compute the coefficients of
  $$r^it^j\, ,\ \ \  i+j\le 2\, .$$
  By the argument of \cite[Section~3.2]{OblomkovOkounkovPandharipande18},
  the coefficient of \(rt^j\) vanishes.
  From the computations of the \(v\) expansions
  \eqref{eq:w(z)}, \eqref{eq:E-factor}, \eqref{eq:D-factor} and \eqref{eq:power-diff}, the terms in front of \(r^i\), \(i>0\)
  are proportional to \(t\). 
  The expression under the residue sign becomes:
  \[\exp\left(\frac{xv}{u}\right)\left(\frac{v+t}{t}+x\Sigma+\frac{x^2t}{v+t}\Sigma^2+\frac{x^3t^2}{(v+t)^2}\Sigma^3\right)+O(t^3)+tO(r^2)\, ,\quad\Sigma=\sum_{n>0}\frac{\fra_n}{(\imath uv)^n}\, .\]
  After applying the residue operation to the last expression,
  we obtain the terms of formula \eqref{eq:decay-intro} in the coefficients of
  the \(x\)-expansion. \qed

  \subsubsection{Case $\widetilde{\Hr}_{k_1+1,k_2+1}(\gamma_1\cdot\gamma_2)$}\label{cc2aa}
  We show next that the double interaction term yields formula \eqref{eq:2bump-intro}.
  The  new computation that is needed for understanding the interaction term is
  \(\widehat{\Hr}_{k_1,k_2}\).
It is convenient  to assemble the expressions into a generating series
\(\widehat{\Hr}(x_1,x_2)\).

To compute \(\widehat{\Hr}(x_1,x_2)\),
we must move all  negative Heisenberg operators in the  product of
the vertex operators \(\Hr^{\GW}(x_1)\Hr^{\GW}(x_2)\) to the left, next to the vacuum \(\langle\, |\).
We use the standard vertex operator commutation relation to perform
this reshuffling:
  \begin{equation}
    \label{eq:vert-com}
    \rmvv_+(x_1,y_1)\rmvv_-(x_2,y_2 )=\rmbb(x_1,y_1,x_2,y_2) \rmvv_-(x_2,y_2)\rmvv_+(x_1,y_1 )\, ,
  \end{equation}
  $$ \rmbb=\frac{(w_2-y_1)(y_2-w_1)}{(y_2-y_1)(w_2-w_1)}\, , $$
  where \(w_i=w(x_i,y_i)\). Using the computations of Section
  \ref{sec:prel-comp},
  we derive the following expansion:
  \[\rmbb=1-\frac{r^2y_1y_2x_1x_2}{(y_1-y_2)^2(y_1+1)(y_2+1)}+O(r^3)\, .\]
  The negative Heisenberg operators interact with
  the vacuum \(\langle\, |\). We obtain:
  \[\widehat{\Hr}(x_1,x_2)=\Res_{y_1=\infty}(\Res_{y_2=\infty}(\rmvv_+^{(1)}\rmvv_+^{(2)}\rmdd^{(1)}\rmdd^{(2)}\rmee^{(1)}\rmee^{(2)}\rmbb^{(12)}))\, ,\]
  where \(\rmvv_+^{(i)}=\rmvv_+(x_i,y_i)\), \(\rmdd^{(i)}=\rmdd(x_i,y_i)\), \(\rmee^{(i)}=\rmee(x_i,y_i)\).

  From \eqref{eq:res-edv}, we see
  \begin{multline*}\widehat{\Hr}(x_1)\widehat{\Hr}(x_2)=\bigg(\Res_{y_1=\infty}\mathrm{E}^{(1)}\cdot \mathrm{D}^{(1)}\cdot\mathrm{V}_+^{(1)}\bigg)\bigg(\Res_{y_2=\infty}\mathrm{E}^{(2)}\cdot \mathrm{D}^{(2)}\cdot\mathrm{V}_+^{(2)}\bigg)=\\
    \Res_{y_1=\infty}(\Res_{y_2=\infty}(\rmvv_+^{(1)}\rmvv_+^{(2)}\rmdd^{(1)}\rmdd^{(2)}\rmee^{(1)}\rmee^{(2)}))\, ,\end{multline*}
where the second equality holds because \(V_+^{(i)}\) commute.
We conclude, after the change of variables, the generating
  function \(\widetilde{\Hr}(x_1,x_2)\) for \(\widetilde{\Hr}_{k_1,k_2}\) is given by
\[\widetilde{\Hr}(x_1,x_2)=\frac{1}{r^2t}\left(\widehat{\Hr}(x_1,x_2)-\widehat{\Hr}(x_1)\widehat{\Hr}(x_2)\right)=\Res(\rmvv_+^{(1)}\rmvv_+^{(2)}\rmdd^{(1)}\rmdd^{(2)}\rmee^{(1)}\rmee^{(2)}\widetilde{\rmbb}^{(12)})/(r^2t^3)\, ,\]
where \(\Res=\Res_{v_1=\infty}\Res_{v_2=\infty}\) and \(\widetilde{\rmbb}^{(12)}=\rmbb^{(12)}-1\). By expanding the scalar factor
\[\rmdd^{(1)}\rmdd^{(2)}\rmee^{(1)}\rmee^{(2)}\widetilde{\rmbb}^{(12)}/(r^2t^3)\]
in the operator inside the residue operation,  we obtain:
  \begin{multline}\label{eq:res-2bump}\frac{tv_1v_2x_1x_2}{(v_1-v_2)^2(v_1+t)(v_2+t)}\exp\left(\frac{x_1v_1+x_2v_2}{u}\right)\left(\frac{v_1+t}{t}+x_1\Sigma^{(1)}+\frac{x_1^2t}{v_2+t}\Sigma^{(1)}\Sigma^{(1)}\right)\\
    \left(\frac{v_2+t}{t}+x_2\Sigma^{(2)}+\frac{x_2^2t}{v_2+t}\Sigma^{(2)}\Sigma^{(2)}\right)+O(t^2)+O(r^2)\, .\end{multline}
  The residue of the  coefficient in front of \(t^{-1}\) in \eqref{eq:res-2bump} vanishes.  The coefficient in front of \(t^0\) is
  \[\exp\left(\frac{x_1v_1+x_2v_2}{u}\right)\frac{x_1x_2}{(v_1-v_2)^2}\left(v_2(1+x_1\Sigma^{(1)})+v_1(1+x_2\Sigma^{(2)})\right).\]
  After applying the \(\Res\) operation,
  we obtain:
  \[\Res_{v_1=\infty}\Res_{v_2=\infty}\exp\left(\frac{x_1v_1+x_2v_2}{u}\right)\frac{x_1^2x_2v_2}{(v_1-v_2)^2}\Sigma^{(1)}.\]
  The coefficient in front of \(x_1^{k_1+2}x_2^{k_2+2}\) in the last expression matches with the \(\fra\)-linear terms of right side of \eqref{eq:2bump-intro} that are proportional to \(c_1^0\).
    
  Finally, we compute the coefficient in front of \(t^1\) in \eqref{eq:res-2bump}:
  \begin{multline*}\frac{x_1x_2}{(v_1-v_2)^2}\exp\left(\frac{x_1v_1+x_2v_2}{u}\right)\left[x_1x_2\Sigma^{(1)}\Sigma^{(2)}+x_1^2v_2\Sigma^{(1)}\Sigma^{(1)}
      +x_2^2v_1\Sigma^{(2)}\Sigma^{(2)}\right.\\
    \left.+ \left(\frac1v_1+\frac1v_2\right)\left(v_2(1+x_1\Sigma^{(1)})+v_1(1+x_2\Sigma^{(2)})\right)\right].\end{multline*}
  The residue of the terms from the first line of the last expression form the generating function of the \(\fra\)-quadratic  terms of the right hand  side of
  \eqref{eq:2bump-intro}. The residue of the terms from the second line of the last expression form the generating function of the
 \(c_1\)-proportional \(\fra\)-linear terms  of the right side of \eqref{eq:2bump-intro}. \qed

   \subsubsection{Case $\widetilde{\Hr}_{k_1+1,k_2+1,k_3+1}(\gamma_1\cdot\gamma_2\cdot \gamma_3)$}
   Finally, we must analyze the triple interaction. T
   he computation here is parallel to
   computations in Sections \ref{cc1aa} and \ref{cc2aa}.
   The new ingredient for the triple bumping reaction is the residue formula:
  \[\widehat{\Hr}(x_1,x_2,x_3)=\Res\left( \rmvv_+^{(1)}\rmvv_+^{(2)}\rmvv_+^{(3)}\rmdd^{(1)}\rmdd^{(2)}\rmdd^{(3)}\rmee^{(1)}\rmee^{(2)}\rmee^{(3)}
      \rmbb^{(12)}\rmbb^{(23)}\rmbb^{(13)}/(r^4t^5)\right)\]
  for the generating function of the operators
  \(\widehat{\Hr}_{k_1,k_2,k_3}\). Here and 
  below \(\Res\) stands for the triple residue
  $$\Res_{v_1=\infty}\Res_{v_2=\infty}\Res_{v_3=\infty}\, .$$

  The generating function $\widetilde{\Hr}(x_1,x_2,x_3)$ for the operators \(\widetilde{\Hr}_{k_1,k_2,k_3}\) is given by:
  \begin{multline*}\widehat{\Hr}(x_1,x_2,x_3)-\widehat{\Hr}(x_1,x_2)\widehat{\Hr}(x_3)-
    \wdh(x_1,x_3)\wdh(x_2)-\wdh(x_2,x_3)\wdh(x_1)+2\wdh(x_1)\wdh(x_2)\wdh(x_3)\, .
\end{multline*}
We expand the above as
  \begin{multline*}
    \frac{1}{r^4t^5}\Res\left(\rmvv_+^{(1)}\rmvv_+^{(2)}\rmvv_+^{(3)}\rmdd^{(1)}\rmdd^{(2)}\rmdd^{(3)}\rmee^{(1)}\rmee^{(2)}\rmee^{(3)}\right.\\ \left.\left(
        \widetilde{\rmbb}^{(12)}\widetilde{\rmbb}^{(23)}\widetilde{\rmbb}^{(13)}+
        \widetilde{\rmbb}^{(12)}\widetilde{\rmbb}^{(23)}+\widetilde{\rmbb}^{(12)}\widetilde{\rmbb}^{(13)}
        +\widetilde{\rmbb}^{(23)}\widetilde{\rmbb}^{(13)}\right)\right)\,.
  \end{multline*}
  Since
  \( \widetilde{\rmbb}^{(12)}\widetilde{\rmbb}^{(23)}\widetilde{\rmbb}^{(13)}\) is proportional
to \(r^6\), we can write the last expression as
\begin{multline*}
  \frac{1}{r^4t^5}\Res\left(\rmvv_+^{(1)}\rmvv_+^{(2)}\rmvv_+^{(3)}\rmdd^{(1)}\rmdd^{(2)}\rmdd^{(3)}\rmee^{(1)}\rmee^{(2)}\rmee^{(3)}\left(
      \widetilde{\rmbb}^{(12)}\widetilde{\rmbb}^{(23)}+\widetilde{\rmbb}^{(12)}\widetilde{\rmbb}^{(13)}        +\widetilde{\rmbb}^{(23)}\widetilde{\rmbb}^{(13)}\right)\right)\end{multline*}
up to $O(r^2)$.

After expanding the expression inside \(\Res\), including the prefactor
  $\frac{1}{r^4t^5}$,  we obtain:
\begin{multline*}
  t^2\left(\frac{v_1+t}{t}+x_1\Sigma^{(1)}\right)\left(\frac{v_2+t}{t}+x_2\Sigma^{(2)}\right)\left(\frac{v_3+t}{t}+x_3\Sigma^{(3)}\right)\\
\times  \exp\left(\frac{x_1v_1+x_2v_2+x_3v_3}{u}\right)\cdot
\Big(f(12;23)+f(23;31)+f(31;12)\Big)\ +\ O(t)\ ,\end{multline*}
where 
    \[f(ij;jk)=\frac{v_iv_j^2v_kx_ix_j^2x_k}{(v_i-v_j)^2(v_j-v_k)^2(v_i+t)(v_j+t)^2(v_k+t)}\, .\]
    
    The application of \(\Res\) to the coefficient in front of \(t^{-1}\) in the last expression yields zero. On the other hand, the
    coefficient in front of \(t^0\) equals
    \begin{multline*}x_1x_2x_3\left(v_2v_3(1+x_1\Sigma^{(1)})+v_1v_3(1+x_2\Sigma^{(2)})+v_1v_2(1+x_3\Sigma^{(3)})\right)\\
      \\\times  \exp\left(\frac{x_1v_1+x_2v_2+x_3v_3}{u}\right)  \\ \times
      \left(\frac{x_2}{(v_1-v_2)^2(v_2-v_3)^2}+
        \frac{x_3}{(v_1-v_3)^2(v_3-v_2)^2}+\frac{x_1}{(v_3-v_1)^2(v_1-v_2)^2}\right)\, .\end{multline*}
    The result of application of \(\Res\)  is therefore
    equal to the generating function of the right side of \eqref{eq:triple-bump-intro}. \qed

  \section{Degree 1 series for $\mathbb{P}^3$}
  \label{sec:appendix}
\subsection{Stationary descendent series}
We provide a complete table of  the stationary stable pair
descendent series for projective $\mathbb{P}^3$ in degree $1$.
Our notation is given by three vectors
$V_{\pt}, V_{\mathsf{L}}, V_{\mathsf{H}}$
of non-negative integers which specify the stationary descendents with
cohomology insertions
$$\pt , {\mathsf{L}}, {\mathsf{H}} \in H^*({\mathbb{P}}^3)$$
corresponding to the point, line, and hyperplane classes respectively.
For example, the data $[1,2],[4,9],[6]$
corresponds to the descendent
$$\ch_3(\pt)\ch_4(\pt)\ch_6(\mathsf{L})\ch_{11}(\mathsf{L})\ch_8(\mathsf{H})\, .$$
In the table, below the full descendent series is given as
rational function in $q$.

\vspace{10pt}

\begin{center}
\begin{tabular}{l|l}
$[],[0,1],[1]$&   $q(3q^2-5\,{q}+3)$\\
$[1],[0],[]$&     $q(q^2-1)/2$\\
$[0],[0,0],[]$&   $q(q+1)^2$\\
$[0],[1],[]$&$3q(q^2-1)/2$\\
$[],[0,0,1],[]$&$2q({q}^{2}-1)$\\
$[],[1,1],[]$&$5q(q-1)^2/2$\\
$[],[0,2],[]$&$ q(5q^2-14q+5 )/6$\\
$[1],[],[1]$&$3q(q-1)^2/4 $\\
$[],[0,0,0],[1]$&$3q(q^2-1)$\\
$[],[2],[1]$&$ \frac{5q(q-1)^3}{4(1+q)} $\\
$[0],[],[1,1]$&$3q(3q^2-2q+3)/4$\\
$[],[0,0],[1,1]$&$q (9q^2-10q+9)/2$\\
  $[],[1],[1,1]$&$\frac{q(q-1)(9q^2-2q+9)}{2(1+q)}$\\
$[],[0],[1,1,1]$&$ \frac{q(q-1)(27q^2+14q+27)}{4(1+q)}$\\
$[0],[],[2]$&$ q(5q^2-2q+5 )/4$\\
$[],[0,0],[2]$&$2q(q^2-q+1)$\\
$[],[1],[2]$&$\frac{q(q-1)(9q^2-2q+9)}{4(1+q)}$\\
$[],[],[1,1,2]$&$q(9q^2-14{q}+9)/2$\\
$[],[],[2,2]$&$ q(17q^2-30q+17)/8$\\
$[],[0],[3]$&$  \frac{q(q-1)(9q^2-2q+9)}{12(1+q)}$\\
$[],[],[1,3]$&$ q(9q^2-22q+9)/8$\\
$[0],[0],[1]$&$ 3q(q^2-1)/2$\\
$[],[],[4]$&$ q(q^2-5q+1)/6$\\
$[],[3],[]$&$ \frac{q(q-1)(q^2-8q+1)}{6(1+q)}  $\\
$[2],[],[]$&$q(q^2-10q+1)/12$\\
$[],[0],[1,2]$&$\frac{q(q-1)(3q^2+q+3)}{(1+q)}$\\
$[0,0],[],[]$&$q(q+1)^2$\\
$[],[0,0,0,0],[]$&$2q(q+1)^2$\\
$[],[],[1,1,1,1]$&$ q(81q^2-102q+81)/2$
\end{tabular}
\end{center}

\vspace{10pt}

The symmetry in the above series is a consequence of the functional
equation, see  \cite[Section 1.7]{P}.
In the stationary case, the stable pairs
series are equal to the corresponding descendent series 
for the Donaldson-Thomas theory of ideal sheaves, see
\cite[Theorem 22]{OblomkovOkounkovPandharipande18}.

\subsection{Descendents of 1}
We tabulate here descendent series of $\mathbb{P}^3$ in degree $1$
with descendents of the identity $1\in H^*(\mathbb{P}^3)$
together with stationary descendents specified as
before by a triple of vectors.

\newpage

\vspace{5pt}

$\bullet$ 
With  $\ch_4(1)$ and
the rest stationary:

\vspace{10pt}
\begin{center}
\begin{tabular}{l|l}
$[],[1],[1]$&$\,{\frac{q \left( 21\,{q}^{4}+37\,{q}^{3}-88\,{q}^{2}+37\,q+21
 \right) }{ 6\left( 1+q \right)^{2}}}$\\
$ [],[0,1],[]$&$7\,q \left( q-1 \right)  \left( 1+q \right)/3$\\
$[],[],[1,2]$&${\frac{q \left( q-1 \right)  \left( 21\,{q}^{4}+79\,{q}^{3}+86
\,{q}^{2}+79\,q+21 \right) }{ 6\left( 1+q \right)^{3}}}$\\
$[0],[],[1]$&$7\,q \left( q-1 \right)  \left( 1+q \right)/4$\\
$[],[0,0],[1]$&$7\,q \left( q-1 \right)  \left( 1+q \right)/2$\\
$[],[0],[1,1]$&$
\,{\frac{q \left( 63\,{q}^{4}+116\,{q}^{3}-134\,{q}^{2}+116\,q+63
 \right) }{ 12\left( 1+q \right)^{2}}}$\\
$[0],[0],[]$&$\,q \left( 7\,{q}^{2}+2\,q+7 \right)/6$\\
$[1],[],[]$&${7}\,q \left( q-1 \right)  \left( 1+q \right)/12$\\
$[],[0,0,0],[]$&$\,q \left( 7\,{q}^{2}+2\,q+7 \right)/3$\\
$[],[2],[]$&$\,{\frac{q \left( 35\,{q}^{4}+56\,{q}^{3}-318\,{q}^{2}+56\,q+35
 \right) }{36 \left( 1+q \right)^{2}}}$\\
$[],[],[3]$&$\,{\frac {q \left( q-1 \right)  \left( 63\,{q}^{4}+232
\,{q}^{3}+218\,{q}^{2}+232\,q+63 \right) }{ 72\left( 1+q \right)^{3}}}$\\
$[],[0],[2]$&$\,{\frac{q \left( 7\,{q}^{4}+13\,{q}^{3}-18\,{q}^{2}+13\,q+7
 \right) }{3 \left( 1+q \right)^{2}}}$
\end{tabular}  
\end{center}

\vspace{10pt}
$\bullet$ With $\ch_5(1)$ and the 
rest stationary:

\vspace{10pt}
\begin{center}
\begin{tabular}{l|l}
$[0],[],[]$&$3\,q \left( q-1 \right)  \left( 1+q \right)/4$\\
$[],[0,0],[]$&$4\,q \left( q-1 \right)  \left( 1+q \right)/3$\\
$[],[1],[]$&$\,{\frac{q \left( 17\,{q}^{4}+24\,{q}^{3}-106\,{q}^{2}+24\,q+17
 \right) }{12 \left( 1+q \right)^{2}}}$\\
$[],[],[1,1]$&$\,{\frac{q \left( q-1 \right)  \left( 9\,{q}^{4}+31\,{q}^{3}+14
\,{q}^{2}+31\,q+9 \right) }{ 3\left( 1+q \right)^{3}}}$\\
$[],[],[2]$&$\,{\frac {q\left( q-1 \right)  \left( 33\,{q}^{4}+112\,{q}^{3}+
38\,{q}^{2}+112\,q+33 \right) }{24 \left( 1+q \right)^{3}}}$\\
$[],[0],[1]$&$\,{\frac {q \left( 3\,q+1 \right)  \left( q+3 \right)  \left( 4\,{
q}^{2}-7\,q+4 \right) }{6 \left( 1+q \right)^{2}}}$
\end{tabular}
\end{center}

\vspace{10pt}
$\bullet$ With
 $\ch_4(1)\ch_4(1)$ and the 
rest stationary:

\vspace{10pt}
\begin{center}
  \begin{tabular}{l|l}
$[],[0],[1]$&$
\,{\frac { q\left( q-1 \right)  \left( 49\,{q}^{4}+196\,{q}^{3}+
534\,{q}^{2}+196\,q+49 \right)}{12 \left( 1+q \right) ^{3}}}$\\
$[0],[],[]$&$\,q \left( 49+2\,q+49\,{q}^{2} \right)/36$\\
$[],[0,0],[]$&$\,q \left( 49+2\,q+49\,{q}^{2} \right)/18 $\\
$[],[1],[]$&$\,{\frac { q\left( q-1 \right)  \left( 49\,{q}^{4}+196\,{q}^{3}+
654\,{q}^{2}+196\,q+49 \right) }{18 \left( 1+q \right)^{3}}}$\\
$[],[],[1,1]$&$
\,{\frac {q \left( 441+1754\,q+4007\,{q}^{2}-3252\,{q}^
{3}+4007\,{q}^{4}+1754\,{q}^{5}+441\,{q}^{6} \right) }{72 \left( 1+q
 \right)^{4}}}$\\
$[],[],[2]$&$
\,{\frac {q \left( 49+195\,q+459\,{q}^{2}-454\,{q}^{3}+459\,{q}^{
4}+195\,{q}^{5}+49\,{q}^{6} \right) }{18 \left( 1+q \right)^{4}}}$
  \end{tabular}
\end{center}

\newpage

\vspace{5pt}

$\bullet$ With $\ch_6(1)$ and the 
rest of stationary:

\vspace{10pt}
\begin{center}
  \begin{tabular}{l|l}
$[],[0],[]$&$\,{\frac {q \left( 17\,{q}^{4}+20\,{q}^{3}-114\,{q}^{2}+20\,q+17
 \right) }{36 \left( 1+q \right)^{2}}}$\\
$[],[],[1]$&$\,{\frac {q \left( q-1 \right)  \left( 17\,{q}^{4}+48\,{q}^{3}-58
\,{q}^{2}+48\,q+17 \right) }{24 \left( 1+q \right)^{3}}}$
  \end{tabular}
\end{center}

\vspace{10pt}

$\bullet$ With $\ch_4(1)\ch_4(1)\ch_4(1)$ and the 
rest stationary:

\vspace{10pt}
\begin{center}
  \begin{tabular}{l|l}
  $[],[0],[]$&$\,{\frac {  q \left( 343\,{q}^{6}+1374\,{q}^{5}+
249\,{q}^{4}+11396\,{q}^{3}+249\,{q}^{2}+1374\,q+343 \right) }{ 108\left( 1
+q \right)^{4}}}$\\
$[],[],[1]$&$\,{\frac { q\left( q-1 \right)  \left( 343\,{q}^{6}+
2058\,{q}^{5}+3705\,{q}^{4}+29900\,{q}^{3}+3705\,{q}^{2}+2058\,q+343
 \right) }{72 \left( 1+q \right)^{5}}}$
\end{tabular}
\end{center}

\vspace{10pt}
$\bullet$ With $\ch_5(1)\ch_4(1)$ and the 
rest stationary:

\vspace{10pt}
\begin{center}
  \begin{tabular}{l|l}
    $[],[],[1]$&$\,{\frac {q \left( 84+331\,q+928\,{q}^{2}-1878\,{q}^{3}+928\,{q}^
{4}+331\,{q}^{5}+84\,{q}^{6} \right) }{ 36\left( 1+q \right)^{4}}}$\\
$[],[0],[]$&
$\,{\frac {2q \left( q-1 \right)  \left( 7+28\,{q}+87
\,{q}^{2}+28\,q^3+7\, q^4 \right) }{9 \left( 1+q \right)^{3}}}$
\end{tabular}
\end{center}

\vspace{10pt}
$\bullet$ Without stationary descendents:

\vspace{10pt}

\begin{center}
  \begin{tabular}{l|l}
    $\ch_7(1)$&$\,{\frac {q \left( q-1 \right)  \left( 2+3\,q-28\,{q}^{2}+3\,{q}^
{3}+2\,{q}^{4} \right)}{ 18\left( 1+q \right)^{3}}}$\\
    $\ch_5(1)\ch_5(1)$&$\,{\frac { 5q\left( 13+50\,{q}+179\,{q}^{2}-
580\,{q}^{3}+179\,{q}^{4}+50\,q^5+13\, q^6 \right) }{72 \left( 1+q \right)^{4}
}}$\\
$\ch_4(1)\ch_6(1)$&$\,{\frac {q \left( 119+462\,q+1737\,{q}^{2}-5852\,{q}^
{3}+1737\,{q}^{4}+462\,{q}^{5}+119\,{q}^{6} \right) }{216 \left( 1+q
 \right)^{4}}}$\\
$\ch_4(1)\ch_4(1)\ch_5(1)$&$\,{\frac {q \left( -49-245\,q-81\,{q}^{2}-6365\,{q}^{3}+6365\,{q}
^{4}+81\,{q}^{5}+245\,{q}^{6}+49\,{q}^{7} \right) }{27 \left( 1+q
                            \right) ^{5}}}$ \\
    $\ch_4(1)\ch_4(1)\ch_4(1)\ch_4(1)$ &
   $\frac{q\left(2401  + 14405\, q + 55690\, q^2 - 594229\, q^3
                                         + 1834570\, q^5 - 
 594229\, q^5 + 55690\, q^6 + 14405\, q^7 + 2401\, q^8\right)}{648 (1 + q)^6}$
\end{tabular}
\end{center}

\vspace{10pt}

\subsection{Examples of the Virasoro relations}
\subsubsection{ \(\calL_2^{\PT}\)}
Examples of the Virasoro relations for \(\calL_1^{\PT}\) were
given in \cite[Section 3]{P}.
We consider here the operator \(\calL_2^{\PT}\) for \(X=\mathbb{P}^3\).

The Chern classes of the tangent bundle of $\mathbb{P}^3$ are
$$ c_1= 4\mathsf{H}\, , \  \ \ c_1c_2 = 24 \mathsf{p}\, ,$$
The  constant term for \(k=2\) is
\begin{eqnarray*}
  \CT_2&=&-\frac{1}{2}\sum_{a+b=4}(-1)^{d^Ld^R}(a+d^L-3)!(b+d^R-3)!\, \ch_a\ch_b(c_1)+\frac{1}{24}\sum_{a+b=2}a!b!\, \ch_a\ch_b(c_1c_2)\\ &=&
  -8\ch_4(\mathsf{H})+8\ch_2(\mathsf{H})\ch_2(\pt)-2\ch_2(\mathsf{L})^2-4\ch_2(\pt),\end{eqnarray*}
where we used the evaluation \(\ch_0(\gamma)=-\int_X\gamma\) and dropped all the terms with \(\ch_1\).
The Virasoro operator for \(k=2\) is then
\begin{eqnarray*}
  \calL^{\PT}_2&=&   \CT_2+\rmR_2+3!\mathrm{R}_{-1}\ch_3(p) \\
  &=& -8\ch_4(\mathsf{H})+8\ch_2(\mathsf{H})\ch_2(\pt)-2\ch_2(\mathsf{L})^2-4\ch_2(\pt)+\rmR_2
 + 3!\rmR_{-1}\ch_3(\pt)\, .
\end{eqnarray*}

Since our examples will be for curves of degree 1 in $\mathbb{P}^3$
and since $$\ch_2(\mathsf{H})=\mathsf{H}\cdot\beta\, ,$$
we can simplify
the operator even further:
\[\calL_{2,\beta=\mathsf{L}}^{\PT}=-8\ch_4(\mathsf{H})+10\ch_2(\pt)-2\ch_2(\mathsf{L})^2+\rmR_2
 + 6\ch_3(\pt)\rmR_{-1}.\]

\subsubsection{Stationary example}
\label{sec:stationary-example}
Let us check the Virasoro constraints
of Theorem \ref{thm:Vir-toric}
for \(k=2\) and
\[D=\ch_3(\mathsf{H})\ch_2(\mathsf{L})\, . \]

The  constant term part of the relation has three summands:
\begin{eqnarray*}
-8\langle \ch_4(\mathsf{H})\ch_3(\mathsf{H})\ch_2(\mathsf{L})\rangle_{\mathsf{L}} &=& {-\frac{8q(q-1)(3q^2+q+3)}{1+q}}\, ,\\
  10\langle \ch_2(\pt)\ch_3(\mathsf{H})\ch_2(\mathsf{L})\rangle_{\mathsf{L}}
  &=&15q(q^2-1)\, ,\\
  -2\langle \ch_2(\mathsf{L})^2\ch_3(\mathsf{H})\ch_2(\mathsf{L})\rangle_{\mathsf{L}}&=&-6q(q^2-1)\, .
 \end{eqnarray*}                                                                                        

 The  rest of the relation  can be divided into  two parts.
 The  first part is \(\rmR_2(D)\) which has two terms:
\begin{eqnarray*}
  6\langle \ch_3(\mathsf{H})\ch_4(\mathsf{L})\rangle_{\mathsf{L}}
  &=&\frac{15 q(q-1)^3}{2(1+q)}\, ,\\
  6\langle \ch_5(\mathsf{H})\ch_2(\mathsf{L})\rangle_{\mathsf{L}}
  &=&\frac{q(q-1)(9q^2-2q+9)}{2(1+q)}\, .
      \end{eqnarray*}
      The second part is
      \begin{eqnarray*}
6\langle\ch_3(\pt)\rmR_{-1}(D)\rangle_{\mathsf{L}} &=& 6\langle\ch_3(\pt)\ch_2(\mathsf{H})\ch_2(\mathsf{L})\rangle_{\mathsf{L}}+
                                                       6\langle\ch_3(\pt)\ch_3(\mathsf{H})\ch_1(\mathsf{L})\rangle_{\mathsf{L}}
        \\
                                                   &=&6\langle\ch_3(\pt)\ch_2(\mathsf{L})\rangle_{\mathsf{L}} \\
        &=& 3q(q^2-1)\, .
\end{eqnarray*}
        

Using the cancellation of poles 
\[-8\langle \ch_4(\mathsf{H})\ch_3(\mathsf{H})\ch_2(\mathsf{L})\rangle_{\mathsf{L}}+6\langle \ch_3(\mathsf{H})\ch_4(\mathsf{L})\rangle_{\mathsf{L}}+6\langle \ch_5(\mathsf{H})\ch_2(\mathsf{L})\rangle_{\mathsf{L}}=-12q(q^2-1)\, ,\]
we  easily verify the Virasoro relation
\[\Big\langle\calL^\PT_2(\ch_3(\mathsf{H})\ch_2(\mathsf{L}) )
  \Big\rangle^{X,\PT}_{\mathsf{L}}=0\, .\]

\subsubsection{Non-stationary example}
\label{sec:non-stat-example}
Let us check the Virasoro relation \(\calL_{2,\beta=\mathsf{L}}^\PT\) for 
\[D=\ch_5(1)\, , \]
a non-stationary case (not covered by Theorem \ref{thm:Vir-toric}, but
implied by Conjecture \ref{vvirr}).

The constant term part of the relation has three summands:
\begin{eqnarray*}
-8\langle\ch_4(H)\ch_5(1)\rangle_{\mathsf{L}}&=&-\,{\frac {q\left( q-1 \right)  \left( 33\,{q}^{4}+112\,{q}^{3}+
        38\,{q}^{2}+112\,q+33 \right) }{3 \left( 1+q \right)^{3}}}\, , \\
10\langle \ch_2(\pt)\ch_5(1)\rangle_{\mathsf{L}}&=&\frac{15}{2}\,q \left( q-1 \right)  \left( 1+q \right)\, ,\\
  -2\langle\ch_2^2(\mathsf{L})\ch_5(1)\rangle_{\mathsf{L}}&=&-\frac{8}{3}\,q \left( q-1 \right)  \left( 1+q \right) \, .
\end{eqnarray*}
The  rest of the relation  can be divided into  two parts:
\begin{eqnarray*}
24\langle \ch_7(1)\rangle_{\mathsf{L}}&=&{\frac { 4q\left( q-1 \right)  \left( 2+3\,q-28\,{q}^{2}+3\,{q}^
        {3}+2\,{q}^{4} \right)}{3 \left( 1+q \right)^{3}}}\, ,\\
6\langle \ch_3(\pt)\ch_4(1)\rangle_{\mathsf{L}}&=&\frac{7}{2}\,q \left(q-1 \right)  \left( 1+q \right)\, . 
\end{eqnarray*}
  
After a remarkable cancellation of poles,
\[-8\langle\ch_4(H)\ch_5(1)\rangle_{\mathsf{L}}+24\langle \ch_7(1)\rangle_{\mathsf{L}}=-\frac{25}{3}q(q-1)(1+q)\, ,\]
we verify the Virasoro relation
\[\Big\langle\calL^\PT_2(\ch_5(1) )
  \Big\rangle^{X,\PT}_{\mathsf{L}}=0\, .\]

\Addresses

\end{document}